\documentclass[a4paper,twoside,11pt]{article}

\usepackage{a4wide, fancyhdr, amsmath, amssymb, mathtools, yfonts}
\usepackage{mathrsfs}
\usepackage{graphicx}
\usepackage{tikz}
\usepackage[all]{xy}
\usepackage[utf8]{inputenc}
\usepackage{amsthm}
\usepackage[english]{babel}
\usepackage{chngcntr}
\usepackage{ifthen}
\usepackage{calc}
\usepackage{hyperref}
\usepackage{authblk}
\numberwithin{equation}{section}

\newcommand{\Z}{\mathbb{Z}}
\newcommand{\Q}{\mathbb{Q}}

\newcommand\FF{\mathbb{F}}

\newcommand\Gal{\mathrm{Gal}}

\newtheorem{lemma}{Lemma}[section]
\newtheorem{theorem}[lemma]{Theorem}
\newtheorem{prop}[lemma]{Proposition}
\newtheorem{corollary}[lemma]{Corollary}
\newtheorem{mydef}[lemma]{Definition}

\newtheorem{remark}[lemma]{Remark}

\title{\vspace{-\baselineskip}\sffamily\bfseries Higher R\'edei reciprocity and integral points on conics}
\author[1]{Peter Koymans\thanks{Department of Mathematics, Ann Arbor, MI 48109, USA, koymans@umich.edu}}
\author[2]{Carlo Pagano\thanks{Department of Mathematics and Statistics, Montreal, Quebec H3G 1M8, Canada, carlein90@gmail.com}}
\affil[1]{University of Michigan}
\affil[2]{Concordia University}
\date{\today}

\begin{document}
\maketitle

\begin{abstract}
Fix an integer $\ell$ such that $|\ell|$ is a prime $3$ modulo $4$. Let $d > 0$ be a squarefree integer and let $N_d(x, y)$ be the principal binary quadratic form of $\Q(\sqrt{d})$. Building on a breakthrough of Smith \cite{Smith}, we give an asymptotic formula for the solubility of $N_d(x, y) = \ell$ in integers $x$ and $y$ as $d$ varies among squarefree integers divisible by $\ell$. 

As a corollary we give, in case $\ell > 0$, an asymptotic formula for the event that the Hasse Unit Index of the field $\Q(\sqrt{-\ell}, \sqrt{d})$ is $2$ as $d$ varies over all positive squarefree integers. Our principal new tool is a generalization of a classical reciprocity law due to R\'edei \cite{Redei}.
\end{abstract}

\section{Introduction}
\label{sIntro}
The study of integral points on conics goes back to at least the ancient Greeks. Much later significant progress was made by the Indian mathematicians Brahmagupta and Bhaskara II around the years 650 and 1150 respectively. Brahmagupta was able to solve the Pell equation
\begin{equation}
\label{ePell}
x^2 - dy^2 = 1 \text{ in } x, y \in \mathbb{Z}
\end{equation}
in special cases, while Bhaskara II was the first to give a method to solve the Pell equation in full generality.

Let $\ell$ be an integer such that $|\ell|$ is a prime $3$ modulo $4$. For a squarefree integer $d > 0$, we define
\[
N_d(x, y) = 
\left\{
	\begin{array}{ll}
		x^2 + xy - \frac{d - 1}{4}y^2  & \mbox{if } d \equiv 1 \bmod 4 \\
		x^2 - dy^2 & \mbox{otherwise,} 
	\end{array}
\right.
\]
which is the principal binary quadratic form of $\Q(\sqrt{d})$. In this paper we look at the equation
\begin{equation}
\label{eGenPell}
N_d(x, y) = \ell \text{ in } x, y \in \mathbb{Z}
\end{equation}
with $d$ squarefree. Unlike equation (\ref{ePell}) it is not always possible to find $x, y \in \mathbb{Z}$ that satisfy the above equation. We denote by $H(K)$ the narrow Hilbert class field of a number field $K$, which is the maximal abelian extension of $K$ that is unramified at all finite places, while the ordinary Hilbert class field must also be unramified at the infinite places. Then equation (\ref{eGenPell}) is soluble if and only if there is an ideal in $\Q(\sqrt{d})$ with norm $\ell$ and trivial Artin symbol in the narrow Hilbert class field of $\Q(\sqrt{d})$, see Lemma \ref{lCrit}. If we take $\ell = -1$ in equation (\ref{eGenPell}), then we get the classical negative Pell equation.

Given $d$, there exists an algorithm to compute the Hilbert class field of $\Q(\sqrt{d})$ both in the narrow and ordinary sense. Hence it is possible to decide given $\ell$ and $d$ whether equation (\ref{eGenPell}) is soluble. In fact, for a fixed squarefree integer $d$, an appeal to the Chebotarev Density Theorem gives an asymptotic for the number of primes $\ell$ such that equation (\ref{eGenPell}) is soluble.

In this paper we ask the opposite question. Instead of fixing $d$, we shall treat $\ell$ as fixed and vary $d$. Equivalently, we ask how often there is some ideal with norm $\ell$ and trivial Artin symbol in $H(\Q(\sqrt{d}))$ as $d$ varies. Unfortunately, the distribution of the Hilbert class field as $d$ varies is not well understood at the moment. 

In fact, the only proven results for the distribution of $\text{Cl}(K)$ with $K$ imaginary quadratic are Davenport--Heilbronn \cite{DH} on $3$-torsion, Fouvry--Kl\"uners \cite{FK1, FK2} on $4$-torsion, based on earlier work of Heath-Brown \cite{HB} on $2$-Selmer groups, and Smith \cite{Smith8, Smith} on respectively $8$-torsion and $2^\infty$-torsion. Heuristically, we understand the situation much better due to the seminal work of Cohen and Lenstra \cite{CL}, which was later extended by Gerth \cite{Gerth}.

Therefore we restrict our attention to only those squarefree integers $d$ that are divisible by $\ell$. In this case we know that $\ell$ ramifies in $\Q(\sqrt{d})$. Gauss genus theory states that the ramified primes generate $\text{Cl}(\Q(\sqrt{d}))[2]$, and that there is precisely one non-trivial relation between them. Here $\text{Cl}$ denotes the narrow class group. In particular we see that $\mathfrak{l} \in \text{Cl}(\Q(\sqrt{d}))[2]$, where $\mathfrak{l}$ is the unique ideal above $\ell$. In this case equation (\ref{eGenPell}) is soluble if and only if $\mathfrak{l}$ is the relation in $\text{Cl}(\Q(\sqrt{d}))[2]$. Hence we need to study the distribution of $\text{Cl}(\Q(\sqrt{d}))[2^\infty]$, and this naturally brings the methods of Smith \cite{Smith} into play.

Note that for equation (\ref{eGenPell}) to be soluble, it is necessary that it is soluble over $\Q$. Or formulated differently, $\mathfrak{l}$ must split in the genus field of $\Q(\sqrt{d})$, which is by definition the maximal subextension of $H(\Q(\sqrt{d}))$ that is abelian over $\Q$. By the Hasse--Minkowski theorem it is easy to determine necessary and sufficient conditions on $d$ for the solubility of equation (\ref{eGenPell}) over $\Q$. With this in mind we can state our first main theorem after introducing the following quantities
\begin{align}
\label{eGamma}
\eta_k := \prod_{j = 1}^k (1 - 2^{-j}) \text{ with } k \in \Z_{\geq 0} \cup \{\infty\}, \quad \gamma := \sum_{j = 0}^\infty \frac{2^{-j^2} \eta_\infty \eta_j^{-2}}{2^{j + 1} - 1}
\end{align}
and with $R \in \{\Z, \Q\}$ and $\ell$ any integer
\[
S_{R, X, \ell} := \{0 < d < X : d \textup{ squarefree}, \ \ell \mid d, N_d(x, y) = \ell \textup{ is soluble with } x, y \in R\}.
\]

\begin{theorem}
\label{tMain}
Let $\ell$ be an integer such that $|\ell|$ is a prime $3$ modulo $4$. Then we have 
\[
\lim_{X \rightarrow \infty} \frac{|S_{\Z, X, \ell}|}{|S_{\Q, X, \ell}|} = \gamma = \frac{1}{2}.
\]
\end{theorem}

Our main equidistribution statement is Theorem \ref{tHeuristic}. We show in Appendix \ref{aSte} how this equidistribution statement implies that the above limit equals $\gamma$. This argument bears strong similarities with an argument due to Stevenhagen \cite{Stevenhagen}. The rather surprising identity $\gamma = 1/2$ is entirely combinatorial, and is proven in Appendix \ref{aGamma}.

We remark that $\gamma$ has a very natural interpretation. Informally speaking, the quantity $2^{-j^2} \eta_\infty \eta_j^{-2}$ represents the probability that the $4$--rank of a random element in the set $S_{\Q, X, \ell}$ is equal to $j$. This will be made precise in Theorem \ref{t4rank}. Note that if the $4$--rank of $\text{Cl}(K)$ is $j$, we have a natural generating set, coming from Gauss genus theory, of size $j + 1$ for $2 \text{Cl}(K)[4]$. Furthermore, Gauss genus theory says that there is exactly one relation between the generators. Hence $1/(2^{j + 1} - 1)$ represents the probability that the ideal above $\ell$ is the relation, if one thinks of the relation as being ``random''. This is very much in spirit of Stevenhagen's conjecture \cite{Stevenhagen} on the solubility of the negative Pell equation. Although we shall not prove it, our techniques readily give the distribution of $2\text{Cl}(\Q(\sqrt{d}))[2^\infty]$ as $d$ varies in $S_{\Q, \infty, \ell}$.

We obtain the distribution of $2\text{Cl}(\Q(\sqrt{d}))[4]$ inside $S_{\Q, \infty, \ell}$ in Theorem \ref{t4rank}, which is already a new and exciting result. The $4$-rank in the family $S_{\Q, \infty, \ell}$ is elevated thanks to the solubility of equation (\ref{eGenPell}). This leads to the somewhat surprising conclusion that the distribution of $2\text{Cl}(\Q(\sqrt{d}))[4]$ inside our family is the same as the $4$-rank distribution of imaginary quadratic fields. An important step in this proof is that we need to both adapt and make effective the Markov chain analysis of Gerth \cite{Gerth}, which we achieve in our companion paper \cite{KP3}.

It is worthwhile to compare with the case $\ell = -1$ at this point, which is precisely Stevenhagen's conjecture. This conjecture was recently proven by the authors \cite{KPPell}. In this case the limit equals 
\[
1 - \prod_{j \text{ odd}} (1 - 2^{-j}) \approx 0.581,
\]
and it was already observed by Stevenhagen that this number is irrational. So despite some similarities with the case $\ell = -1$, there is at least one remarkable difference in the rationality of the limiting value.

We will now give a brief comparison between the material in \cite{KPPell} and this paper. Two key results in this paper, Theorem \ref{tRedei} (originally proven in this paper) and Theorem \ref{tProfitable}, are now superseded by respectively \cite[Theorem 3.2]{KPPell} and \cite[Theorem 4.6]{KPPell}. It is for this reason that we have opted to cite these strictly stronger versions. This makes the key new algebraic result in this paper Theorem \ref{tReflectionEll}, where we make essential use of the R\'edei reciprocity law.

By classical techniques one can give an asymptotic formula for $|S_{\Q, X, \ell}|$; this requires only slight modifications of \cite[Exercise 21, Section 6.2]{MV}, see also \cite[Section 3]{Park}. Indeed, we have
\[
|S_{\Q, X, \ell}| \sim \frac{1}{\sqrt{\pi}} \cdot \frac{C(\ell) \cdot \delta(\ell)}{|\ell|} \cdot \frac{X}{\sqrt{\log X}},
\]
where
\[
C(\ell) = \lim_{s \rightarrow 1} \left(\sqrt{s - 1} \cdot \prod_{\substack{p \text{ odd} \\ (\ell/p) = 1}} \left(1 + \frac{1}{p^s}\right)\right), \quad
\delta(\ell)
=
\left\{
\begin{array}{ll}
3/2  & \mbox{if } \ell \equiv 1 \bmod 8 \\
3/4 & \mbox{if } \ell \equiv 3 \bmod 8 \\
1 & \mbox{if } \ell \equiv 5 \bmod 8 \\
3/4 & \mbox{if } \ell \equiv 7 \bmod 8.
\end{array}
\right.
\]
This yields the following corollary of Theorem \ref{tMain}.

\begin{corollary}
Take $\ell$ to be an integer such that $|\ell|$ is a prime $3$ modulo $4$. Then
\[
|S_{\Z, X, \ell}| \sim \frac{\gamma}{\sqrt{\pi}} \cdot \frac{C(\ell) \cdot \delta(\ell)}{|\ell|} \cdot \frac{X}{\sqrt{\log X}}.
\]
\end{corollary}

Earlier work was done by Milovic \cite{Milovic}, who showed that $S_{\Z, X, \pm 2}$ has the same order of magnitude as $S_{\Q, X, \pm 2}$. It is plausible that our methods can be adapted to the case $\ell = \pm 2$ as well.

An immediate application is the following result. For a biquadratic field $\Q(\sqrt{a}, \sqrt{b})$, the Hasse Unit Index is defined to be
\[
H_{a, b} := \left[\mathcal{O}_{\Q(\sqrt{a}, \sqrt{b})}^\ast : \mathcal{O}_{\Q(\sqrt{a})}^\ast \mathcal{O}_{\Q(\sqrt{b})}^\ast \mathcal{O}_{\Q(\sqrt{ab})}^\ast\right].
\]
If the biquadratic field is totally complex, then it is known that $H_{a, b} \in \{1, 2\}$, see for example the work of Lemmermeyer \cite{Lemmermeyer}. Our next theorem determines the distribution of the Hasse Unit Index in many cases.

\begin{corollary}
\label{cHasse}
Let $\ell > 3$ be a prime $3$ modulo $4$. Then we have
\begin{multline*}
|\{0 < d < X \textup{ squarefree} : H_{-\ell, d} = 2\}| \sim |S_{\Z, X, \ell}| + |S_{\Z, X, -\ell}| \sim \\ 
\left(\frac{\gamma}{\sqrt{\pi}} \cdot \frac{C(\ell) \cdot \delta(\ell)}{\ell} + 
\frac{\gamma}{\sqrt{\pi}} \cdot \frac{C(-\ell) \cdot \delta(-\ell)}{\ell}\right) \cdot \frac{X}{\sqrt{\log X}}.
\end{multline*}
\end{corollary}

From a more geometric perspective, Theorem \ref{tMain} counts how often there exists an integral point in a family of conics. As such, it is natural to view this result from the perspective of the integral Brauer--Manin obstruction. The seminal work \cite{CTF} was the first to systematically study the integral Brauer--Manin obstruction.

Theorem \ref{tMain} critically relies on a generalization of a reciprocity law due to R\'edei \cite{Redei}. This generalization is proven in Section \ref{sRed}. An extensive treatment of the classical R\'edei reciprocity law can be found in Corsman \cite{Corsman}, and was one of the main ingredients in Smith's work on $4$-Selmer groups and $8$-torsion of class groups \cite{Smith8}. Corsman's and Smith's formulations of the R\'edei reciprocity law have a minor flaw, which was corrected by Stevenhagen \cite{Redei-Stevenhagen}.

We will now roughly explain how we make use of our new reciprocity law. Following Smith's method, we need to prove equidistribution of
\[
\text{Frob}_{K_{x_1, \dots, x_m, y}/\Q}(\ell)
\]
as we vary $y$, where $K_{x_1, \dots, x_m, y}$ is a completely explicit field depending only on $x_1, \dots, x_m$ and $y$. Our reciprocity law implies that under suitable conditions
\[
\text{Frob}_{K_{x_1, \dots, x_m, y}/\Q}(\ell) = \text{Frob}_{K_{x_1, \dots, x_m, \ell}/\Q}(y).
\]
This allows us to apply the Chebotarev Density Theorem to obtain the desired equidistribution. In the case $m = 1$, the fields $K_{x_1, y}$ are constructed by R\'edei, and one recovers the R\'edei reciprocity law. In the case $m = 2$, the field $K_{x_1, x_2, y}$ first appears in Amano \cite{Amano} for special values of $x_1$, $x_2$ and $y$, while the fields $K_{x_1, \dots, x_m, y}$ are constructed in full generality by Smith \cite{Smith}. 

In the language of Smith, these fields are the field of definition of certain maps from $G_\Q$ to $\FF_2$ that Smith calls $\phi_{x_1, \dots, x_m, y}$ or simply $\phi_{\bar{x}}$. The field of definition is an unramified multiquadratic extension of a multiquadratic extension of $\Q$. As such, they are intimately related to the $2$-torsion of the class groups of multiquadratic fields. This connection is explored in recent work of the authors \cite{KP2}.

We finish the introduction by mentioning some other important results related to class groups. A lot of attention has recently be given to providing non-trivial upper bounds for $\text{Cl}(K)[\ell]$ for a fixed prime $\ell$. This was initiated by Pierce \cite{Pierce, Pierce2} for $\ell = 3$ and continued by Ellenberg and Venkatesh \cite{EV}, Ellenberg, Pierce and Wood \cite{EPW}, Frei and Widmer \cite{FW}, Pierce, Turnage-Butterbaugh and Wood \cite{PTW}.

Instead of studying class groups of quadratic extensions of $\Q$, one can study the distribution of class groups in the family of degree $\ell$ cyclic extensions of $\Q$. This was explored by Gerth \cite{Gerth2} and Klys \cite{Klys}, whose work was later generalized by the authors \cite{KP} using the Smith method \cite{Smith}. It is natural to wonder if the methods in this paper can also be used to study norm forms coming from degree $\ell$ cyclic extensions.

\subsection*{Acknowledgements}
We are most grateful to Alexander Smith for explaining his work to us on several occasions. We are also deeply indebted to Adam Morgan for pointing out to us that $\gamma = 1/2$. Peter Stevenhagen kindly explained his proof of R\'edei reciprocity to us, which inspired us to prove a more general version of the R\'edei reciprocity law. We thank Vladimir Mitankin for showing us a useful reference and we thank Stephanie Chan and Djordjo Milovic for our many insightful conversations about negative Pell. Both authors are grateful to the Max Planck Institute for Mathematics in Bonn for its hospitality and financial support.

\section{Algebraic criteria}
\label{sAlg}
In this section we collect the algebraic lemmas that link our theorems to questions about the narrow class group. These lemmas are valid for arbitrary non-zero integers $\ell$, and we shall only later restrict to $\ell$ with $|\ell| \equiv 3 \bmod 4$ a prime. For a non-zero integer $\ell$, we define $\text{sign}(\ell) = 0$ if $\ell > 0$ and $\text{sign}(\ell) = 1$ if $\ell < 0$.

\begin{lemma}
Let $\ell$ be a non-zero integer and let $d > 0$ be a squarefree integer. Then there are $x, y \in \Z$ with $N_d(x, y) = \ell$ if and only if there is an integral ideal $I$ of $\mathcal{O}_{\Q(\sqrt{d})}$ with norm $|\ell|$ such that $I \cdot (\sqrt{d})^{\textup{sign}(\ell)}$ has trivial Artin symbol in $H(\Q(\sqrt{d}))$.
\end{lemma}

\begin{proof}
Suppose that there are $x, y \in \Z$ with $N_d(x, y) = \ell$. In case $d \equiv 1 \bmod 4$, we look at the ideal $I = (x + y\frac{\sqrt{d} + 1}{2})$. It has norm $|\ell|$, and furthermore the element $x + y\frac{\sqrt{d} + 1}{2}$ has norm $\ell$. Then $I \cdot (\sqrt{d})^{\textup{sign}(\ell)}$ is a principal ideal that has a generator with positive norm. This implies that $I \cdot (\sqrt{d})^{\textup{sign}(\ell)}$ is a principal ideal with a totally positive generator, and hence it has trivial Artin symbol in $H(\Q(\sqrt{d}))$. In case $d \not \equiv 1 \bmod 4$, we use a similar argument with the ideal $I = (x + y\sqrt{d})$.

For the other direction suppose that there is an integral ideal $I$ of $\mathcal{O}_{\Q(\sqrt{d})}$ with norm $|\ell|$ and $I \cdot (\sqrt{d})^{\textup{sign}(\ell)}$ has trivial Artin symbol in $H(\Q(\sqrt{d}))$. Then $I \cdot (\sqrt{d})^{\textup{sign}(\ell)}$ is a principal ideal with a totally positive generator $\alpha$, so $N_{\Q(\sqrt{d})/\Q}(\alpha) = d^{\text{sign}(\ell)} |\ell|$. Hence we have
\[
I = \left(\frac{\alpha}{\sqrt{d}^{\text{sign}(\ell)}}\right) \quad \text{and} \quad N_{\Q(\sqrt{d})/\Q}\left(\frac{\alpha}{\sqrt{d}^{\text{sign}(\ell)}}\right) = \ell.
\]
Expanding $\alpha/\sqrt{d}^{\text{sign}(\ell)}$ as $x + y\frac{\sqrt{d} + 1}{2}$ if $d \equiv 1 \bmod 4$ and $x + y\sqrt{d}$ otherwise, we get the desired $x, y \in \Z$ with $N_d(x, y) = \ell$. 
\end{proof}

In case that $\ell \mid d$, we see that every prime dividing $\ell$ ramifies in $\Q(\sqrt{d})$. Hence there is exactly one ideal $\mathfrak{l}$ of $\Q(\sqrt{d})$ with norm $|\ell|$. Furthermore, since $\mathfrak{l} \in \text{Cl}(\Q(\sqrt{d}))[2]$, we see that it is enough to demand that $\mathfrak{l}$ has trivial Artin symbol in the narrow $2^\infty$-Hilbert class field of $\Q(\sqrt{d})$, denoted $H_2(\Q(\sqrt{d}))$, which is the maximal abelian extension of $\Q(\sqrt{d})$ that is unramified at all finite places and has degree a power of $2$. This yields the following criterion.

\begin{lemma}
\label{lCrit}
Take a non-zero integer $\ell$ and take a squarefree integer $d > 0$ divisible by $\ell$. Then there exist $x, y \in \Z$ with $N_d(x, y) = \ell$ if and only if there is an integral ideal $I$ of $\mathcal{O}_{\Q(\sqrt{d})}$ with norm $|\ell|$ such that $I \cdot (\sqrt{d})^{\textup{sign}(\ell)}$ has trivial Artin symbol in $H_2(\Q(\sqrt{d}))$.
\end{lemma}

Our final lemma allows us to deduce Corollary \ref{cHasse} directly from Theorem \ref{tMain}.

\begin{lemma}
Suppose that $\ell > 3$ is an odd squarefree integer and let $d > 0$ be a squarefree integer with $d \neq \ell$ and $d \neq 3 \ell$. Then we have $H_{-\ell, d} = 2$ if and only if $\ell \mid d$ and there are $x, y \in \Z$ with $N_d(x, y) = \ell$ or $N_d(x, y) = -\ell$.
\end{lemma}

\begin{proof}
By our assumptions on $\ell$ and $d$ we have that the roots of unity of $\Q(\sqrt{-\ell}, \sqrt{d})$ are $\{\pm 1\}$. Let $\epsilon$ be the fundamental unit of $\Q(\sqrt{d})$. Then Kubota's work \cite[Satz 2]{Kubota} shows that $H_{-\ell, d} = 2$ if and only if $- \epsilon$ is a square in $\Q(\sqrt{-\ell}, \sqrt{d})$. By Kummer theory this is equivalent to $\epsilon = \ell z^2$ for some $z \in \Q(\sqrt{d})^\ast$. This is in turn equivalent to the requirements that every prime dividing $\ell$ must ramify in $\Q(\sqrt{d})$, and furthermore that the unique ideal with norm $|\ell|$ is principal in $\Q(\sqrt{d})$. These last two conditions are equivalent to $\ell \mid d$ and the existence of $x, y \in \Z$ with $N_d(x, y) = \pm \ell$.
\end{proof}

\section{Expansions}
The goal of this section is to briefly summarize the main results about expansion maps that we need. The significance of expansion maps is that these are the main actors in the higher R\'edei reciprocity law and our reflection principles.

\subsection{Notation}
\label{ssNot}
Fix an algebraic closure $\overline{\Q}$ of $\Q$ for the rest of the paper. All our number fields are implicitly taken inside this fixed algebraic closure $\overline{\Q}$. If $K$ is a number field, we define $G_K := \Gal(\overline{\Q}/K)$.

Throughout, we view $\FF_2$ as a discrete $G_\Q$-module with trivial action. If $\phi : G_\Q \rightarrow X$ is a continuous map with $X$ a discrete topological space, we define $L(\phi)$ to be the smallest Galois extension $K$ of $\Q$ through which $\phi$ factors via the canonical projection map $G_\Q \rightarrow \Gal(K/\Q)$. This is well-defined by \cite[Lemma 2.3]{KP}. For us an unramified extension $L/K$ shall always mean unramified at all finite places of $K$.

In this paper $X$ will always be a product set $X_1 \times \dots \times X_r$, where each $X_i$ is a finite, non-empty set of primes intersecting trivially with all the other $X_j$. This allows us to identify $(x_1, \dots, x_r) \in X$ with the squarefree integer $x_1 \cdot \ldots \cdot x_r$, and we shall often do so implicitly. For $a \in \Z_{\geq 0}$, we will write $[a]$ for the set $\{1, \dots, a\}$. If $S \subseteq [r]$, we define
\[
\overline{X}_S := \prod_{i \in S} (X_i \times X_i) \times \prod_{i \not \in [r] - S} X_i,
\]
and we let $\pi_i$ be the projection to $X_i \times X_i$ if $i \in S$ and to $X_i$ if $i \not \in S$. The natural projection maps from $X_i \times X_i$ to $X_i$ are denoted by $\text{pr}_1$ and $\text{pr}_2$. For two subsets $S, S_0 \subseteq [r]$, we let $\pi_{S, S_0}$ be the projection map from $\overline{X}_S$ to
\[
\prod_{i \in S \cap S_0} (X_i \times X_i) \times \prod_{i \in ([r] - S) \cap S_0} X_i
\]
given by $\pi_i$ on each $i \in S_0$. The set $S$ shall often be clear from context, in case we will simply write $\pi_{S_0}$ for $\pi_{S, S_0}$. Finally, take some $\bar{x} \in \overline{X}_S$ and $T \subseteq S \subseteq [r]$. Then we define $\bar{x}(T)$ to be the following multiset
\[
\{\bar{y} \in \overline{X}_T : \pi_{[r] - (S - T)}(\bar{y}) = \pi_{[r] - (S - T)}(\bar{x}) \text{ and } \forall i \in S - T \exists j \in [2] : \pi_i(\bar{y}) = \text{pr}_j(\pi_i(\bar{x}))\}
\]
with the multiplicity of $\bar{y} \in \bar{x}(T)$ being
\[
\prod_{i \in S - T} \left|\left\{j \in [2] : \pi_i(\bar{y}) = \text{pr}_j(\pi_i(\bar{x}))\right\}\right|.
\]

\subsection{Expansion maps}
In this subsection, we shall quickly recall the facts about expansions that we will need. These are treated more elaborately in \cite[Section 2.1]{Smith} and \cite[Section 7]{KP}. The paper \cite{KP2} is entirely devoted to a careful study of the properties of expansions. We start by recalling a definition from \cite[Definition 3.21]{KP2}.

\begin{mydef}
Let $X \subseteq \textup{Hom}_{\textup{top.gr.}}(G_{\Q}, \FF_2)$ be linearly independent and let $\chi_0 \in X$. An expansion map with support $X$ and pointer $\chi_0$ is a continuous group homomorphism
\[
\psi : G_\Q \rightarrow \mathbb{F}_2[\FF_2^{X - \{\chi_0\}}] \rtimes \mathbb{F}_2^{X - \{\chi_0\}}
\]
such that $\pi_\chi \circ \psi = \chi$ for every $\chi \in X - \{\chi_0\}$, where $\pi_\chi : \mathbb{F}_2[\FF_2^{X - \{\chi_0\}}] \rtimes \mathbb{F}_2^{X - \{\chi_0\}} \rightarrow \FF_2$ is the natural projection, and $\pi \circ \psi = \chi_0$, where $\pi : \mathbb{F}_2[\FF_2^{X - \{\chi_0\}}] \rtimes \FF_2^{X - \{\chi_0\}} \rightarrow \FF_2$ is the unique non-trivial character that sends the subgroup $\{0\} \rtimes \FF_2^{X - \{\chi_0\}}$ to $0$.
\end{mydef}

Note that an expansion map is automatically surjective, since an expansion map surjects modulo the Frattini by assumption. There is another characterization of expansion maps that we give now, first given in Section 3.3 of \cite{KP2}. We have an isomorphism
\[
\FF_2[\FF_2^{X - \{\chi_0\}}] \cong \FF_2[\{t_x : x \in X - \{\chi_0\}\}]/(\{t_x^2 : x \in X - \{\chi_0\}\})
\]
by sending $t_x$ to $1 \cdot \text{id} + 1 \cdot e_x$, where $e_x$ is the vector that is $1$ exactly on the $x$-th coordinate. Note that the squarefree monomials $t_Y := \prod_{y \in Y} t_y$ give a basis of 
\[
\FF_2[\{t_x : x \in X - \{\chi_0\}\}]/(\{t_x^2 : x \in X - \{\chi_0\}\}),
\]
as $Y$ varies through the subsets of $X - \{\chi_0\}$. Therefore, projection on monomials gives rise to a collection of continuous $1$-cochains
\[
\phi_Y(\psi) : G_\Q \rightarrow \FF_2
\]
for every $Y \subseteq X - \{\chi_0\}$. Together they allow us to reconstruct $\psi$ by the formula
\begin{equation}
\label{eReconstruct}
\psi(g) = \left(\sum_{Y \subseteq X - \{\chi_0\}} \phi_Y(\psi)(g) t_Y, (\chi(g))_{\chi \in X - \{\chi_0\}}\right).
\end{equation}
Now define $\chi_S := \prod_{\chi \in S} \chi$, where the product is taken in $\FF_2$. From equation (\ref{eReconstruct}) and the composition law for the semidirect product we deduce that
\begin{equation}
\label{eSmith22}
(d\phi_Y(\psi))(g_1, g_2) = \sum_{\varnothing \subsetneq S \subseteq Y} \chi_S(g_1) \phi_{Y - S}(\psi)(g_2),
\end{equation}
where $d$ is the operator that sends $\text{Map}(G_\Q, \FF_2)$ to $\text{Map}(G_\Q \times G_\Q, \FF_2)$ with the rule
\[
(d\phi)(g_1, g_2) = \phi(g_1) + \phi(g_2) + \phi(g_1g_2).
\]
Equation (\ref{eSmith22}) is simply \cite[eq. (2.2)]{Smith}. Conversely, if we are given a system of maps $(\phi_Y)_{Y \subseteq X - \{\chi_0\}}$ satisfying equation (\ref{eSmith22}) and $\phi_\varnothing = \chi_0$, we get an expansion map $\psi$ with support $X$ and pointer $\chi_0$.

For each odd prime $p$, we choose an element $\sigma_p \in G_\Q$ such that
\[
\chi(\sigma_p) = 1 \Longleftrightarrow \chi \text{ has conductor divisible by } p
\] 
for all quadratic characters $\chi$. We also choose elements $\sigma_2(1), \sigma_2(2) \in G_\Q$ such that 
\[
\chi(\sigma_2(1)) = 1 \Longleftrightarrow \chi \text{ has conductor exactly divisible by } 4
\]
and
\[
\chi(\sigma_2(2)) = 1 \Longleftrightarrow \chi \text{ has conductor exactly divisible by } 8.
\]
We define
\[
\mathfrak{S} := \{\sigma_p : p \text{ odd}\} \cup \{\sigma_2(1), \sigma_2(2)\}.
\]
Write $G_\Q^{\text{pro-}2}$ for the maximal pro-$2$-quotient of $G_\Q$.

\begin{lemma}
\label{lTopGen}
The image of $\mathfrak{S}$ in $G_\Q^{\textup{pro-}2}$ is a minimal set of topological generators for $G_\Q^{\textup{pro-}2}$.
\end{lemma}

\begin{proof}
This is a standard fact, see for example \cite[Proposition 2.2]{KPPell} for a proof.
\end{proof}

\begin{mydef}
\label{dExpansion}
For any integer $x$, we let $\chi_x : G_\Q \rightarrow \FF_2$ be the character corresponding to $\Q(\sqrt{x})$. Let $X := X_1 \times \dots \times X_r$ with $|X_i| = 2$ for $i \in [r]$. For a subset $U \subseteq [r]$, we declare $\chi_U : G_\Q \rightarrow \FF_2$ to be
\[
\chi_U(\sigma) := \prod_{i \in U} \chi_{\textup{pr}_1(\pi_i(x)) \cdot \textup{pr}_2(\pi_i(x))}(\sigma).
\]
A \textup{pre-expansion} for $X$ is a sequence $(\phi_T)_{T \subsetneq [r]}$ where each $\phi_T: G_\Q \rightarrow \FF_2$ is a continuous $1$-cochain satisfying
\begin{align}
\label{edphi}
(d\phi_T)(\sigma,  \tau) = \sum_{\varnothing \neq U \subseteq T} \chi_U(\sigma) \phi_{T - U}(\tau).
\end{align}
For the remainder of the paper, we shall always assume that $\phi_\varnothing$ is linearly independent from the space of characters spanned by $\{\chi_{\{i\}} : i \in [r]\}$. 

Furthermore, a pre-expansion is said to be \textup{good} if $\phi_T(\sigma) = 0$ for all $\sigma \in \mathfrak{S}$ and all $\varnothing \subsetneq T \subsetneq [r]$. An expansion for $X$ is a sequence $(\phi_T)_{T \subseteq [r]}$ satisfying the recursive equation (\ref{edphi}) for each $T \subseteq [r]$. An expansion is \textup{good} if $\phi_T(\sigma) = 0$ for all $\sigma \in \mathfrak{S}$ and all $\varnothing \subsetneq T \subseteq [r]$.
\end{mydef}

The following result is the key result regarding expansions. It is a rephrasing of \cite[Proposition 2.1]{Smith}, see also \cite[Proposition 7.3]{KP} for a similar statement. Informally, it shows that a pre-expansion can be completed to a good expansion under favorable circumstances.

\begin{prop}
\label{pExpansion}
Let $X := X_1 \times \dots \times X_r$ with $|X_i| = 2$ for $i \in [r]$. Let $(\phi_T)_{T \subsetneq [r]}$ be a good pre-expansion for $X$. Assume that for every $i \in [r]$, every prime $p \in X_i$ splits completely in $L(\phi_{[r] - \{i\}})$. Further assume that the character $\chi_{\{i\}}$ is locally trivial at $2$ and at each prime $p \in X_j$ for each distinct $i, j \in [r]$.

Then there is a unique continuous map $\phi_{[r]} : G_\Q \rightarrow \FF_2$ such that
\[
(d\phi_{[r]})(\sigma,  \tau) = \sum_{\varnothing \neq U \subseteq [r]} \chi_U(\sigma) \phi_{[r] - U}(\tau)
\]
and such that $(\phi_T)_{T \subseteq [r]}$ a good expansion.
\end{prop}

\begin{proof}
See \cite[Proposition 2.16]{KPPell}.
\end{proof}

In Section \ref{sMain} we work with many expansions simultaneously. For a box $X = X_1 \times \dots \times X_r$, $S \subseteq [r]$ and $\bar{x} \in \overline{X}_S$, we shall use the shorthand $\phi_{\bar{x}, a}$ for an expansion map $\phi_S$ associated to the box $\pi_S(\bar{x})$ with $\phi_\varnothing = \chi_a$. In case that $\text{pr}_1(\pi_i(\bar{x})) \neq \text{pr}_2(\pi_i(\bar{x}))$ for all $i \in S$, we can naturally view each $\pi_i(\bar{x})$ as a set with two primes as required for Definition \ref{dExpansion}. If instead $\text{pr}_1(\pi_i(\bar{x})) = \text{pr}_2(\pi_i(\bar{x}))$ for some $i \in S$, we set $\phi_{\bar{x}, a}$ to be zero. Observe that $\phi_{\bar{x}, a}$ only depends on $a$ and $\pi_S(\bar{x})$.

\subsection{Governing expansions}
\label{ssGov}
Since we have to work with many expansions simultaneously in the final section, we abstract the essential properties in the notion of governing expansions.

\begin{mydef}
\label{dGovExp}
Let $X := X_1 \times \dots \times X_r$, let $S \subseteq [r]$ and let $a \in \Z$ be squarefree. We say that there exists a governing expansion $\mathfrak{G}$ on $(X, S, a)$ if
\begin{itemize}
\item we have for all $T \subseteq S$ and all $\bar{x} \in \overline{X}_T$ a good expansion $\phi_{\bar{x}, a}$ satisfying
\[
(d\phi_{\bar{x}, a})(\sigma, \tau) = \sum_{\varnothing \subsetneq T' \subseteq T} \chi_{T'}(\sigma) \phi_{\pi_{T - T'}(\bar{x}), a}(\tau);
\]
\item take $T \subseteq S$, $i \in T$ and $\bar{x}_0, \bar{x}_1, \bar{x}_2 \in \overline{X}_T$. Suppose that
\[
\pi_{T - \{i\}}(\bar{x}_0) = \pi_{T - \{i\}}(\bar{x}_1) = \pi_{T - \{i\}}(\bar{x}_2)
\]
and that there are primes $p_0, p_1, p_2$ satisfying
\[
\textup{pr}_j(\pi_i(\bar{x}_k)) = p_{k + j - 1}
\]
for all $j \in \{1, 2\}$ and $k \in \{0, 1, 2\}$, where the indices are taken modulo $3$. Then we have
\begin{align}
\label{ePhiAdd}
\phi_{\bar{x}_0, a} + \phi_{\bar{x}_1, a} = \phi_{\bar{x}_2, a}.
\end{align}
\end{itemize} 
\end{mydef}

These conditions are rather stringent, and typically there does not exist a governing expansion $\mathfrak{G}$ on $(X, S, a)$. To construct governing expansions, we introduce additive systems.

\begin{mydef}
\label{dAS}
Let $X := X_1 \times \dots \times X_r$. An additive system $\mathfrak{A}$ on $X$ is a tuple 
\[
(\overline{Y}_S, \overline{Y}_S^\circ, F_S, A_S)_{S \subseteq[r]}
\]
satisfying
\begin{itemize}
\item for each $S \subseteq [r]$, we have that $A_S$ is a finite $\FF_2$ vector space, $\overline{Y}_S$ and $\overline{Y}_S^\circ$ are sets satisfying
\[
\overline{Y}_S^\circ \subseteq \overline{Y}_S \subseteq \overline{X}_S
\]
and $F_S : \overline{Y}_S \rightarrow A_S$ is a function such that
\[
\overline{Y}_S^\circ := \{\bar{y} \in \overline{Y}_S : F_S(\bar{y}) = 0\};
\]
\item we have for all non-empty $S \subseteq [r]$ that
\[
\overline{Y}_S = \{\bar{x} \in \overline{X}_S : \bar{x}(T) \subseteq \overline{Y}_T^\circ \textup{ for all } T \subsetneq S\}.
\]
Here we view $\bar{x}(T)$ as a set by forgetting the multiplicities;
\item take $i \in S \subseteq [r]$ and take $\bar{x}_0, \bar{x}_1, \bar{x}_2 \in \overline{Y}_S$ satisfying
\[
\pi_{S - \{i\}}(\bar{x}_0) = \pi_{S - \{i\}}(\bar{x}_1) = \pi_{S - \{i\}}(\bar{x}_2)
\]
such that there are primes $p_0, p_1, p_2$ with
\[
\textup{pr}_j(\pi_i(\bar{x}_k)) = p_{k + j - 1}
\]
for all $j \in \{1, 2\}$ and all $k \in \{0, 1, 2\}$, where the indices are taken modulo $3$. Then we have
\begin{align}
\label{eFAdd}
F_S(\bar{x}_0) + F_S(\bar{x}_1) = F_S(\bar{x}_2).
\end{align}
\end{itemize}
We will sometimes write $\overline{Y}_S(\mathfrak{A})$, $\overline{Y}_S^\circ(\mathfrak{A})$, $F_S(\mathfrak{A})$ and $A_S(\mathfrak{A})$ to stress that this data is associated to the additive system $\mathfrak{A}$.
\end{mydef}

We remark that equation (\ref{eFAdd}) implies that
\[
F_S(\bar{x}) = 0
\]
in case $\text{pr}_1(\pi_i(\bar{x})) = \text{pr}_2(\pi_i(\bar{x}))$. We will now construct an additive system that will help us find governing expansions.

\begin{lemma}
\label{lASGov}
Let $r \geq 2$ be an integer and let $X := X_1 \times \dots \times X_r$ be such that $X_j$ contains only odd primes. Take an odd squarefree integer $a = q_1 \cdot \ldots \cdot q_t$ such that 
\begin{align}
\label{eLegendreAss}
\left(\frac{a}{p}\right) = 1 \textup{ and } \left(\frac{pp'}{q_i}\right) = 1 \textup{ for all } i \in [t]
\end{align}
for all $p, p' \in X_j$ with $j \in [r]$. Let $\Omega$ be a set of places of $\Q$ disjoint from the $X_j$ and $q_i$. We assume that every $v \in \Omega$ splits in $\Q(\sqrt{a})$. Let $W \subseteq X$ be a subset such that for all $w_1, w_2 \in W$, for all distinct $i, j \in [r]$ and for all $v \in \Omega$
\[
\left(\frac{\pi_i(w_1)}{\pi_j(w_1)}\right) = \left(\frac{\pi_i(w_2)}{\pi_j(w_2)}\right), \quad \pi_i(w_1) \pi_i(w_2) \equiv 1 \bmod 8, \quad \pi_i(w_1) \pi_i(w_2) \equiv \square \bmod v.
\]
Then there exists an additive system $\mathfrak{A}$ on $X$ such that
\begin{itemize}
\item we have $\overline{Y}_\varnothing^\circ(\mathfrak{A}) = W$;
\item we have $|A_S(\mathfrak{A})| \leq 2^{r - 1 + |\Omega|}$ for all $S \subseteq [r]$;
\item suppose that $Z := Z_1 \times \dots \times Z_r$ satisfies $Z_i \subseteq X_i$ and suppose that 
\[
\overline{Z}_{[r]} \subseteq \overline{Y}_{[r]}(\mathfrak{A}).
\]
Then there exists a governing expansion $\mathfrak{G}$ on $(Z, [r], a)$ such that every $v \in \Omega$ splits completely in $\phi_{\bar{z}, a}$ for $\bar{z} \in \overline{Z}_{[r]}$.
\end{itemize}
\end{lemma}

\begin{proof}
Note that an additive system $\mathfrak{A}$ is uniquely specified by the maps $F_S(\mathfrak{A})$ and the set $\overline{Y}_\varnothing(\mathfrak{A})$. We take $\overline{Y}_\varnothing(\mathfrak{A}) = W$ and we will inductively construct the maps $F_S(\mathfrak{A})$. If $S = \varnothing$, we take $F_\varnothing(\mathfrak{A})$ to be the zero map. 

Now suppose that $S = \{i\}$. For $\bar{x} \in \overline{Y}_{\{i\}}(\mathfrak{A})$, Proposition \ref{pExpansion} and equation (\ref{eLegendreAss}) imply that there is a good expansion $\phi_{\bar{x}, a}$. We claim that equation (\ref{ePhiAdd}) holds. Indeed, suppose that $\bar{x}_0, \bar{x}_1, \bar{x}_2 \in \overline{Y}_{\{i\}}(\mathfrak{A})$ satisfy the assumptions for equation (\ref{ePhiAdd}). Then we have
\[
d\left(\phi_{\bar{x}_0, a} + \phi_{\bar{x}_1, a} + \phi_{\bar{x}_2, a}\right) = 0.
\]
This shows that $\phi_{\bar{x}_0, a} + \phi_{\bar{x}_1, a} + \phi_{\bar{x}_2, a}$ is a quadratic character, which vanishes on all $\sigma \in \mathfrak{S}$. Since $\mathfrak{S}$ is a set of topological generators by Lemma \ref{lTopGen}, we obtain the equality
\[
\phi_{\bar{x}_0, a} + \phi_{\bar{x}_1, a} + \phi_{\bar{x}_2, a} = 0, 
\]
proving the claim. Then we define $F_{\{i\}}(\mathfrak{A})$ by sending $\bar{x} \in \overline{Y}_{\{i\}}(\mathfrak{A})$ to
\[
\phi_{\bar{x}, a}(\text{Frob } k)
\]
as $k$ runs through $\pi_j(\bar{x})$ for $j \in [r] - \{i\}$ and $\Omega$. We observe that $\text{Frob}(k)$ lands in the center of $\Gal(L(\phi_{\bar{x}, a})/\Q)$ by the assumption $\bar{x} \in \overline{Y}_{\{i\}}(\mathfrak{A})$ and the assumptions on $W$. Then equation (\ref{eFAdd}) follows from equation (\ref{ePhiAdd}).

Now we proceed inductively. We see that Proposition \ref{pExpansion} implies that there is a good expansion $\phi_{\bar{x}, a}$ for $\bar{x} \in \overline{Y}_S(\mathfrak{A})$. Once more we have that
\[
d\left(\phi_{\bar{x}_0, a} + \phi_{\bar{x}_1, a} + \phi_{\bar{x}_2, a}\right) = 0,
\]
which follows from the fact that equation (\ref{ePhiAdd}) holds for all the $\phi_{\pi_T(\bar{x}_j), a}$ for $T \subsetneq S$ by the induction hypothesis. From this, we deduce just like before that
\[
\phi_{\bar{x}_0, a} + \phi_{\bar{x}_1, a} + \phi_{\bar{x}_2, a} = 0.
\]
We define $F_S(\mathfrak{A})$ by sending $\bar{x} \in \overline{Y}_S(\mathfrak{A})$ to
\[
\phi_{\bar{x}, a}(\text{Frob } k)
\]
as $k$ runs through $\pi_j(\bar{x})$ for $j \in [r] - S$ and $\Omega$. This defines our additive system $\mathfrak{A}$, which indeed satisfies the listed properties.
\end{proof}

\section{Higher R\'edei reciprocity}
\label{sRed}
This section summarizes what was formerly the main algebraic innovation of this paper, a generalization of the classical R\'edei reciprocity law (in turn a generalization of quadratic reciprocity). This \emph{higher R\'edei reciprocity} law is now superseded by \cite[Section 3]{KPPell}, which is why we have opted to state this new slightly stronger reciprocity law.

\subsection{Statement of the reciprocity law}
Let $n \in \mathbb{Z}_{\geq 1}$ and let $A \subseteq \text{Hom}_{\text{top.gr.}}(G_{\Q}, \FF_2)$ with $|A| = n$. Let $\chi_1, \chi_2 \not \in A$ be two distinct elements of $\text{Hom}_{\text{top.gr.}}(G_{\Q}, \FF_2)$. Write
$$
A_1 := A \cup \{\chi_1\}, \quad A_2: = A \cup \{\chi_2\}. 
$$ 
For a finite extension $L/\Q$, we denote by $\text{Ram}(L/\Q)$ the set of places of $\Q$ that ramify in $L/\Q$. Furthermore, for a collection of characters $T \subseteq \text{Hom}_{\text{top.gr.}}(G_{\Q}, \FF_2)$, we denote by $\Q(T)$ the corresponding multiquadratic extension of $\Q$. We make the following assumptions throughout this subsection:
\begin{itemize}
\item as $\chi$ varies in $A$, the $n$ sets $\text{Ram}(\Q(\chi)/\Q)$ are non-empty and pairwise disjoint;
\item the sets $\text{Ram}(\Q(\chi_1)/\Q)$ and $\text{Ram}(\Q(\chi_2)/\Q)$ are non-empty and disjoint from 
\[
\bigcup_{\chi \in A} \text{Ram}(\Q(\chi)/\Q);
\]
\item we have the inclusion $\text{Ram}(\Q(\chi_1)/\Q) \cap \text{Ram}(\Q(\chi_2)/\Q) \subseteq \{(2)\}$ and $\text{inv}_2(\chi_1 \cup \chi_2) = 0$, where $\text{inv}_2$ is the invariant map of $\Q_2$;
\item the conductor of $\chi_1$ or the conductor of $\chi_2$ is not divisible by $8$;
\item the places $(2)$ and $\infty$ split completely in $\Q(A)/\Q$.
\end{itemize}
In particular, the first assumption yields that $A$ is a set of $n$ linearly independent characters over $\FF_2$. Furthermore, $\chi_1$ is linearly independent from $A$, and the same holds for $\chi_2$.

Now suppose that we have two expansion maps
$$
\psi_1, \psi_2: G_{\Q} \twoheadrightarrow \FF_2[\FF_2^A] \rtimes \FF_2^A,
$$
with supports $A_1,A_2$ and pointers $\chi_1,\chi_2$ respectively. Let $(\phi_{1, B})_{B \subseteq A}, (\phi_{2, B})_{B \subseteq A}$ be the corresponding system of continuous $1$-cochains from $G_\Q$ to $\FF_2$ with $\phi_{1, \varnothing} = \chi_1, \phi_{2, \varnothing} = \chi_2$. Note that $L(\psi_1), L(\psi_2)$ are central $\FF_2$-extensions of
$$
M(\psi_1) := \Q(A) \prod_{B \subsetneq A}L(\phi_{1, B})/ \Q, \quad M(\psi_2) := \Q(A) \prod_{B \subsetneq A}L(\phi_{2, B})/ \Q.
$$
We need one more definition before stating our reciprocity law. In the next definition, we will use index notation modulo $2$. Since $(2)$ splits completely in $\Q(A)$, it follows from equation (\ref{eSmith22}) that $\phi_T(\psi_i)|_{G_{\Q_2}}$ is a quadratic character.

\begin{mydef}
\label{dRedeiAdmissible}
Let $(A_1, A_2, \psi_1, \psi_2)$ be a $4$-tuple as above. We say $(A_1, A_2, \psi_1, \psi_2)$ is \emph{R\'edei admissible} if the following conditions simultaneously hold
\begin{itemize}
\item if the infinite place $\infty$ of $\Q$ splits completely in $\Q(A_1 \cup A_2)$, then $\infty$ splits completely in $M(\psi_1)M(\psi_2)/\Q$ as well;
\item if $i \in \{1, 2\}$ and $\infty$ ramifies in $L(\psi_i)/\Q$, then $\infty$ splits completely in $M(\psi_{i + 1 \bmod 2})/\Q$;
\item if $i \in \{1, 2\}$ and an odd place $w$ of $\Q(A)$ ramifies in $L(\psi_i)/\Q(A)$, then the unique place $v$ of $\Q$ below $w$ splits completely in $M(\psi_{i + 1 \bmod 2})/\Q$. Furthermore, $v$ is unramified in $L(\psi_{i + 1 \bmod 2})/\Q$;
\item if $L(\psi_1)L(\psi_2)/\Q$ is ramified at $(2)$, then there exists $i$ such that $\psi_i$ is a good expansion map, $\chi_i$ has conductor not divisible by $8$ and $\phi_T(\psi_i)$ restricts trivially to $G_{\Q_2}$ for all $\varnothing \subsetneq T \subsetneq A$. Furthermore, we have $\textup{inv}_2(\chi_i \cup \phi_T(\psi_{i + 1 \bmod 2})|_{G_{\Q_2}}) = 0$ for all $T$.
\end{itemize}
We call $(\chi_1,\chi_2)$ the \emph{pointer vector} of the $4$-tuple and we call $A$ the \emph{base set} of the $4$-tuple. 
\end{mydef}

Let now $(A_1, A_2, \psi_1, \psi_2)$ be a R\'edei admissible $4$-tuple, as above, with pointer vector $(\chi_1,\chi_2)$. Then it follows by the definition that each odd place $v \in \text{Ram}(\Q(\chi_1)/\Q)$ is unramified in $L(\psi_2)/\Q$ and furthermore the consequently defined Artin class $\text{Art}(v, L(\psi_2)/\Q)$ lands in $\text{Gal}(L(\psi_2)/M(\psi_2))$, which is a central subgroup of $\text{Gal}(L(\psi_2)/\Q)$ of size equal to $2$ and hence can uniquely be identified with $\mathbb{F}_2$. We conclude that $\text{Art}(v, L(\psi_2)/\Q)$ is a well-defined element of $\FF_2$. Symmetrically, the same holds if we swap the role of $1$ and $2$. 

Now suppose that $(2)$ ramifies in $L(\psi_1)L(\psi_2)/\Q$. Let $i$ be as guaranteed in the last part of the above definition. If $(2)$ splits in $\chi_i$, then $(2)$ splits completely in $M(\psi_i)/\Q$ and is unramified in $L(\psi_i)/\Q$. Therefore $\text{Art}((2), L(\psi_i)/\Q)$ is again a well-defined element of $\FF_2$. Now suppose that $(2)$ does not split in $\chi_i$. Write $\mathfrak{t}$ for the unique place of $\Q(\chi_i)$ above $(2)$. Then the Artin symbol $\text{Art}(\mathfrak{t}, L(\psi_i)/\Q(\chi_i))$ is once more a well-defined element of $\FF_2$. We will often abuse notation, and simply write this Artin symbol as $\text{Art}((2), L(\psi_i)/\Q)$.

Finally, for a quadratic extension $\Q(\sqrt{d})/\Q$, we put $\widetilde{\text{Ram}}(\Q(\sqrt{d})/\Q)$ to be the set of places in $\text{Ram}(\Q(\sqrt{d})/\Q)$ with the only exception of $(2)$, which is excluded in case $d$ has even $2$-adic valuation. 

\begin{theorem} 
\label{tRedei}
Let $(A_1, A_2, \psi_1, \psi_2)$ be a R\'edei admissible $4$-tuple with pointer vector $(\chi_1, \chi_2)$. Then we have
$$
\sum_{v \in \widetilde{\emph{Ram}}(\Q(\chi_1)/\Q)} \textup{Art}(v, L(\psi_2)/\Q) = \sum_{v' \in \widetilde{\emph{Ram}}(\Q(\chi_2)/\Q)} \textup{Art}(v', L(\psi_1)/\Q).
$$
\end{theorem}

\begin{proof}
See \cite[Theorem 3.2]{KPPell}.
\end{proof}

\subsection{Raw cocycles}
We will now introduce raw cocycles. These are the fundamental objects of interest in our paper, but difficult to calculate with. Our main idea is to eventually relate raw cocycles to governing expansions, which are more workable. To achieve this, we shall need to develop a substantial amount of theory.

Let $N := \Q_2/\Z_2$, which we endow with the discrete topology. We view $N$ as a $G_\Q$-module with trivial action. For any $x \in X$, we let $N(x)$ be the $G_\Q$-module $N$ twisted with the action of $\Q(\sqrt{x})$, i.e. $\sigma \cdot_x n = n$ if $\chi_x(\sigma) = 0$ and $\sigma \cdot_x n = -n$ if $\chi_x(\sigma) = 1$.

\begin{mydef}
We define $\textup{Cocy}(G_\Q, N(x)[2^k])$ to be the set of continuous $1$-cocycles $\psi$ such that $\Q(\sqrt{x}) L(\psi)/\Q(\sqrt{x})$ is unramified.
\end{mydef}

\begin{remark}
If $k > 1$, then $L(\psi)$ automatically contains $\Q(\sqrt{x})$. However, this need not be the case if $k = 1$.
\end{remark}

By definition of $\textup{Cocy}(G_\Q, N(x)[2^k])$, inflation of cocycles induces an isomorphism
\[
\text{Cocy}(G_\Q, N(x)[2^k]) \cong \text{Cocy}(\text{Gal}(H(\Q(\sqrt{x}))/\Q), N(x)[2^k]).
\]
Of fundamental importance is the split exact sequence
\[
0 \rightarrow N(x)[2^k] \rightarrow \textup{Cocy}(\text{Gal}(H(\Q(\sqrt{x}))/\Q), N(x)[2^k]) \rightarrow \text{Cl}(\Q(\sqrt{x}))^\vee[2^k] \rightarrow 0.
\]
Fix an element $\sigma \in \text{Gal}(H(\Q(\sqrt{x}))/\Q)$ projecting non-trivially in $\text{Gal}(\Q(\sqrt{x})/\Q)$. The first map is given by sending $n$ to the unique cocycle that sends $\sigma$ to $n$ and sends the group $\text{Gal}(H(\Q(\sqrt{x}))/\Q(\sqrt{x}))$ to zero, while the second map is simply restriction of cocycles. The exact sequence is split, since all the groups appearing are killed by $2^k$ and $N(x)[2^k] \cong \Z/2^k\Z$ as abelian groups. This allows us to work with cocycles instead of the class group.

If $x$ is a squarefree integer, we have a natural map
\[
f: \{d \text{ squarefree} : d \mid \Delta_{\Q(\sqrt{x})}\} \rightarrow \text{Cl}(\Q(\sqrt{x}))[2],
\]
where $\Delta_K$ denotes the discriminant of $K$. We now define the $m$-th Artin pairing
\[
\text{Art}_{m, x} : f^{-1}(2^{m - 1} \text{Cl}(\Q(\sqrt{x}))[2^m]) \times 2^{m - 1} \text{Cocy}(\text{Gal}(H(\Q(\sqrt{x}))/\Q), N(x)[2^m]) \rightarrow \FF_2
\]
by sending $(b, \chi)$ to $\psi(\text{Frob } \mathfrak{b})$, where $\mathfrak{b}$ is the unique ideal of norm $b$ and furthermore $\psi \in \text{Cocy}(\text{Gal}(H(\Q(\sqrt{x}))/\Q), N(x)[2^m])$ is any lift of $\chi$ satisfying $2^{m - 1} \psi = \chi$. The left kernel of this pairing is $f^{-1}(2^m \text{Cl}(\Q(\sqrt{x}))[2^{m + 1}])$ and the right kernel of this pairing is $2^m \text{Cocy}(\text{Gal}(H(\Q(\sqrt{x}))/\Q), N(x)[2^{m + 1}])$.

We introduce the notion of a raw cocycle.

\begin{mydef}
A raw cocycle for $x$ is a sequence $(\psi_i)_{0 \leq i \leq x}$ with $\psi_i \in \textup{Cocy}(G_\Q, N(x)[2^k])$ and $2\psi_i = \psi_{i - 1}$ for all $1 \leq i \leq k$. 
\end{mydef}

\subsection{Profitability}
With the higher R\'edei reciprocity law, we have taken our first major step towards the reflection principles that we need. This reciprocity law is also the key input for the theory of \emph{profitable triples} first developed in \cite[Section 4]{KPPell}. We will now replicate the main result on profitable triples.

\begin{mydef}
Let $X := X_1 \times \dots \times X_r$, let $S \subseteq [r]$ with $s := |S| \geq 1$ and let $\bar{x} \in \overline{X}_S$. We assume that $\textup{pr}_1(\pi_i(\bar{x})) \cdot \textup{pr}_2(\pi_i(\bar{x})) \equiv 1 \bmod 8$ for all $i \in S$. Let $x_0 \in \bar{x}(\varnothing)$. Let $(\psi_{s + 1}(x))_{x \in \bar{x}(\varnothing) - \{x_0\}}$ be a tuple with each $\psi_{s + 1}(x)$ a raw cocycle for $x$. We assume that there exists a character $\chi_a$ such that
\[
\psi_1(x) = 2^s \cdot \psi_{s + 1}(x) = \chi_a
\]
for all $x \in \bar{x}(\varnothing) - \{x_0\}$. We call the triple $(\bar{x}, \chi_a, (\psi_{s + 1}(x))_{x \in \bar{x}(\varnothing) - \{x_0\}})$ profitable if the following properties are all satisfied for all $i \in S$
\begin{itemize}
\item we have
\[
\left(\sum_{\substack{x \in \bar{x}(\varnothing) \\ \forall j \in T : \pi_j(x) \neq \pi_j(x_0)}} \psi_{s + 1 - |T|}(x)\right)(\sigma_p) = 0
\]
for each $T \subseteq S$ containing $i$ and each $p \in \{\textup{pr}_1(\pi_i(\bar{x})), \textup{pr}_2(\pi_i(\bar{x}))\}$;
\item we have
\[
\sum_{\substack{x \in \bar{x}(\varnothing) \\ \pi_i(x) \neq \pi_i(x_0)}} \psi_{s - 1}(x) = 0;
\]
\item there exists a good expansion map $\phi$ for $\pi_{S - \{i\}}(\bar{x})$ with pointer $\phi_\varnothing = \chi_{\textup{pr}_1(\pi_i(\bar{x})) \cdot \textup{pr}_2(\pi_i(\bar{x}))}$. Furthermore, for every $j \in [r] - S$, the prime $\pi_j(\bar{x})$ splits completely in $L(\phi)$. Finally, $(2)$ and $\infty$ split completely in $M(\phi)$.
\end{itemize}
Given a profitable triple $(\bar{x}, \chi_a, (\psi_{s + 1}(x))_{x \in \bar{x}(\varnothing) - \{x_0\}})$ and a subset $\varnothing \subseteq T \subsetneq S$, we attach the map
\[
\psi_T(\bar{x}, \chi_a, (\psi_{s + 1}(x))_{x \in \bar{x}(\varnothing) - \{x_0\}}) := \sum_{\substack{x \in \bar{x}(\varnothing) \\ \forall j \in S - T : \pi_j(x) \neq \pi_j(x_0)}} \psi_{|T| + 1}(x).
\]
We will also use this notation for $T = S$ in case we are further given a raw cocycle $\psi_{s + 1}(x_0)$ for $x_0$.
\end{mydef}

\begin{remark}
\label{rPositivea}
Since $s \geq 1$, we see that $\chi_a$ is a double in $\textup{Cl}(\Q(\sqrt{x}))^\vee$ for all $x \in \bar{x}(\varnothing) - \{x_0\}$. This forces $a > 0$ by checking the local embedding problem at $\mathbb{R}$.
\end{remark}

\begin{theorem}
\label{tProfitable}
Let $(\bar{x}, \chi_a, (\psi_{s + 1}(x))_{x \in \bar{x}(\varnothing) - \{x_0\}})$ be a profitable triple. Then there exists a raw cocycle $\psi_{s + 1}(x_0)$ for $x_0$ satisfying the following properties
\begin{itemize}
\item we have for all $1 \leq i \leq s$
\[
\sum_{x \in \bar{x}(\varnothing)} \psi_i(x) = 0;
\]
\item the tuple
\[
(\psi_T(\bar{x}, \chi_a, (\psi_{s + 1}(x))_{x \in \bar{x}(\varnothing) - \{x_0\}}))_{T \subseteq S}
\]
is an expansion map $\psi$ for $\pi_S(\bar{x})$ with pointer $\chi_a$. If some place above an odd prime $p$ ramifies in the extension
\[
L(\psi)/\Q(\{\sqrt{\textup{pr}_1(\pi_i(\bar{x})) \cdot \textup{pr}_2(\pi_i(\bar{x}))} : i \in S\}),
\]
then there exists $j \in [r] - S$ such that $p = \pi_j(\bar{x})$. Furthermore, the ramification index of any prime in $L(\psi)/\Q$ is at most $2$. 

Finally, the restriction of $\psi_T(\bar{x}, \chi_a, (\psi_{s + 1}(x))_{x \in \bar{x}(\varnothing) - \{x_0\}})$ to $G_{\Q_2}$ is a quadratic character contained in the span of $\{\chi_5, \chi_{x_0}\}$ for each $T \subseteq S$.
\end{itemize}
\end{theorem}

\begin{proof}
See \cite[Theorem 4.6]{KPPell}.
\end{proof}

\section{Reflection principles}
\label{sRef}
The section proves three reflection principles, which relate the class group structure of different fields. This will serve as the algebraic input for our analytic machinery. We will start with minimality and agreement, which correspond directly to reflection principles in \cite[Section 2]{Smith}. In our final subsection, we prove a completely new reflection principle. This new reflection principle is the main innovation of our work, and relies critically on the higher R\'edei reciprocity law and profitability.

\subsection{Minimality}
\begin{mydef}
Let $X := X_1 \times \dots \times X_r$, let $S \subseteq [r]$ and let $\bar{x} \in \overline{X}_S$. Suppose that we are given for each $x \in \bar{x}(\varnothing)$ a raw cocycle $(\psi_i(x))_{0 \leq i \leq |S|}$. We say that the set of raw cocycles is minimal at $\bar{x}$ if
\begin{align}
\label{eMinimal}
\sum_{y \in \bar{y}(\varnothing)} \psi_{|T|}(y) = 0
\end{align}
for all $T \subseteq S$ and all $\bar{y} \in \bar{x}(T)$. Note that the sum here has to be taken with multiplicities.
\end{mydef}

We have all the necessary notation to state our first reflection principle, which is directly based on part (i) of Theorem 2.8 of Smith \cite{Smith}.

\begin{theorem}
\label{tRefMin}
Let $X := X_1 \times \dots \times X_r$, let $S \subseteq [r]$ with $|S| \geq 2$ and let $\bar{x} \in \overline{X}_S$. Take $x_0 \in \bar{x}(\varnothing)$. Let $(\psi_i(x))_{0 \leq i \leq |S|}$ be a raw cocycle for all $x \in \bar{x}(\varnothing) - \{x_0\}$. We assume that for all $T \subseteq S$ and all $\bar{y} \in \bar{x}(T)$ not containing $x_0$, we have that $(\psi_i(y))_{0 \leq i \leq |T|}$ is minimal at $\bar{y}$. Also suppose that for every $i \in S$ we have that $\textup{pr}_1(\pi_i(\bar{x})) \cdot \textup{pr}_2(\pi_i(\bar{x}))$ is a square locally at $2$ and at all primes in $\pi_j(\bar{x})$ with $i \neq j$. Then there is a raw cocycle $(\psi_i(x_0))_{0 \leq i \leq |S|}$ such that $\psi_1(x_0) = \psi_1(x)$ for all $x \in \bar{x}(\varnothing)$.

Additionally, suppose that there is an integer $b$ such that $b \in f^{-1}(2^{m - 1} \textup{Cl}(\Q(\sqrt{x}))[2^m])$ for all $x \in \bar{x}(\varnothing)$. Then we also have
\[
\sum_{x \in \bar{x}(\varnothing)} \textup{Art}_{|S|, x}(b, \psi_1(x)) = 0.
\]
\end{theorem}

\begin{proof}
We note that minimality implies that $\psi_1(x) = \psi_1(x')$ for all $x, x' \in \bar{x}(\varnothing)$ not equal to $x_0$. For the first part, we put
\begin{align}
\label{ex0corner}
\psi_{|S|}(x_0) := - \sum_{\substack{x \in \bar{x}(\varnothing) \\ x \neq x_0}} \psi_{|S|}(x).
\end{align}
Then it is readily seen that $2^{|S| - 1} \psi_{|S|}(x_0) = \psi_1(x)$ for any $x \in \bar{x}(\varnothing)$. Furthermore, our minimality assumption (\ref{eMinimal}) and \cite[Proposition 2.9]{KPPell} show that $\psi_{|S|}(x_0)$ is a cocycle from $G_\Q$ to $N(x_0)[2^k]$. It remains to deal with the ramification locus of $L(\psi_{|S|}(x_0))$.

We first claim that $2\psi_{|S|}(x_0) \in \text{Cocy}(\text{Gal}(H(\Q(\sqrt{x_0}))/\Q), N(x_0)[2^{|S| - 1}])$. But observe that the minimality assumptions and equation (\ref{ex0corner}) imply that
\[
2\psi_{|S|}(x_0) = - \sum_{\substack{x \in \bar{x}(\varnothing) - \bar{y}(\varnothing) \\ x \neq x_0}} \psi_{|S| - 1}(x)
\]
for any $\bar{y} \in \bar{x}(T)$ not containing $x_0$ such that $|T| = |S| - 1$. Therefore we can find for any prime $p$ not dividing $\Delta_{\Q(\sqrt{x_0})}$ a cube $\bar{y}$ such that $p$ is unramified in every $L(\psi_{|S| - 1}(x))$ for $x \in \bar{x}(\varnothing) - \bar{y}(\varnothing)$. This gives the claim, since equation (\ref{ex0corner}) implies that the ramification degree of any prime $p$ in the larger field $L(\psi_{|S|}(x_0))$ is at most $2$.

Since $L(\psi_{|S|}(x_0))/L(2\psi_{|S|}(x_0)) \Q(\sqrt{x_0})$ is a central $\FF_2$-extension, we see that the extension $L(\psi_{|S|}(x_0))/L(2\psi_{|S|}(x_0)) \Q(\sqrt{x_0})$ can be made unramified over $\Q$ at all primes, except those that ramify in $L(2\psi_{|S|}(x_0)) \Q(\sqrt{x_0})$, by twisting with a character $\chi$, see for example \cite[Proposition 4.8]{KP}. Using once more that the ramification degree of any prime $p$ in the field $L(\psi_{|S|}(x_0))$ is at most $2$, it follows that the primes that already ramify in the field $L(2\psi_{|S|}(x_0)) \Q(\sqrt{x_0})$ can not ramify further in $L(\psi_{|S|}(x_0))/L(2\psi_{|S|}(x_0)) \Q(\sqrt{x_0})$. Therefore $L(\psi_{|S|}(x_0) + \chi)$ is an unramified extension of $\Q(\sqrt{x_0})$ for some character $\chi$, proving the first part of the theorem.

For the second part, recall that $\text{Art}_{|S|, x}(b, \psi_1(x))$ does not depend on the choice of the lift $\psi$ with $2^{|S| - 1} \psi = \psi_1(x)$, and hence we may choose the lift $\psi_{|S|}(x_0) + \chi$ of $\psi_1(x_0)$. By definition we have that
\[
\sum_{x \in \bar{x}(\varnothing)} \text{Art}_{|S|, x}(b, \psi_1(x)) = \sum_{x \in \bar{x}(\varnothing)} \psi_{|S|}(x)(\text{Frob } \mathfrak{b}),
\]
where $\mathfrak{b}$ is the unique ideal of $\Q(\sqrt{x})$ of norm $b$. Now locally at any prime $p$ dividing $b$, we know that $\Q_p(\sqrt{x}) = \Q_p(\sqrt{x'})$ for all $x, x' \in \overline{x}(\varnothing)$ by our assumptions. Since $\psi_{|S|}(x)$ becomes a character when restricted to $\Q_p(\sqrt{x})$, the relation
\[
\sum_{x \in \bar{x}(\varnothing)} \psi_{|S|}(x) = \chi
\]
yields
\[
\sum_{x \in \bar{x}(\varnothing)} \psi_{|S|}(x)(\text{Frob } \mathfrak{b}) = \chi(\text{Frob } \mathfrak{b}).
\]
We claim that the last expression is trivial. Indeed, $\mathfrak{b}$ is in $2\text{Cl}(\Q(\sqrt{x}))[4]$ for all $x \in \bar{x}(\varnothing)$, and therefore pairs trivially with any character that is unramified outside the union of the primes dividing $\Delta_{\Q(\sqrt{x})}$ as $x$ ranges through $\bar{x}(\varnothing)$. But equation (\ref{ex0corner}) shows that the character $\chi$ is of this shape, concluding the proof of our theorem.
\end{proof}

\subsection{Agreement}
Our second reflection principle is based on the notion of agreement.

\begin{mydef}
Let $X := X_1 \times \dots \times X_r$, let $S \subseteq [r]$, let $i_a \in S$ and let $\bar{x} \in \overline{X}_S$ be given. Take for each $x \in \bar{x}(\varnothing)$ a raw cocycle $(\psi_i(x))_{0 \leq i \leq |S|}$. We further assume that we have a good expansion $\phi_{\pi_{S - \{i_a\}}(\bar{x}), \textup{pr}_1(\pi_{i_a}(\bar{x})) \cdot \textup{pr}_2(\pi_{i_a}(\bar{x}))}$. We say that the set of raw cocycles agrees with the expansion at $\bar{x}$ if
\[
\sum_{y \in \bar{y}(\varnothing)} \psi_{|T|}(y)
=
\left\{
\begin{array}{ll}
\phi_{\pi_{T - \{i_a\}}(\bar{y}), \textup{pr}_1(\pi_{i_a}(\bar{x})) \cdot \textup{pr}_2(\pi_{i_a}(\bar{x}))}  & \mbox{\textup{if} } i_a \in T \\
0 & \mbox{\textup{if} } i_a \not \in T
\end{array}
\right.
\]
for all $T \subseteq S$ and all $\bar{y} \in \bar{x}(T)$.
\end{mydef}

We are now ready to state a reflection principle that is very similar to part (ii) of \cite[Theorem 2.8]{Smith}.

\begin{theorem}
\label{tRefAgr}
Let $X := X_1 \times \dots \times X_r$, let $S \subseteq [r]$ with $|S| \geq 2$, let $i_a \in S$ and let $\bar{x} \in \overline{X}_S$ be given. Take $x_0 \in \bar{x}(\varnothing)$. Let $(\psi_i(x))_{0 \leq i \leq |S|}$ be a raw cocycle for all $x \in \bar{x}(\varnothing) - \{x_0\}$ such that there is a character $\chi$ with the property
\[
\psi_1(x) = \chi + \chi_{\pi_{i_a}(x)} \textup{ for all } x.
\]
Assume that there is a good expansion $\phi_{\pi_{S - \{i_a\}}(\bar{x}), \textup{pr}_1(\pi_{i_a}(\bar{x})) \cdot \textup{pr}_2(\pi_{i_a}(\bar{x}))}$. We further assume that for all $T \subseteq S$ and all $\bar{y} \in \bar{x}(T)$ not containing $x_0$, we have that $(\psi_i(y))_{0 \leq i \leq |T|}$ agrees with the expansion at $\bar{y}$. Also suppose that for every $i \in S$, $\textup{pr}_1(\pi_i(\bar{x})) \cdot \textup{pr}_2(\pi_i(\bar{x}))$ is a square locally at $2$ and at all primes in $\pi_j(\bar{x})$ with $i \neq j$. Then there is a raw cocycle $(\psi_i(x_0))_{0 \leq i \leq |S|}$ such that $\psi_1(x_0) = \psi_1(x)$ for all $x \in \bar{x}(\varnothing)$ with $\pi_{i_a}(x_0) = \pi_{i_a}(x)$.

Moreover, assume that there exists an integer $b$ satisfying $b \in f^{-1}(2^{m - 1} \textup{Cl}(\Q(\sqrt{x}))[2^m])$ for all $x \in \bar{x}(\varnothing)$. Then we have
\begin{align}
\label{eAgrN}
\sum_{x \in \bar{x}(\varnothing)} \textup{Art}_{|S|, x}(b, \psi_1(x)) = \sum_{p \mid b} \phi_{\pi_{S - \{i_a\}}(\bar{x}), \textup{pr}_1(\pi_{i_a}(\bar{x})) \cdot \textup{pr}_2(\pi_{i_a}(\bar{x}))}(\textup{Frob } p).
\end{align}
If instead $x/b \in f^{-1}(2^{m - 1} \textup{Cl}(\Q(\sqrt{x}))[2^m])$ for all $x \in \bar{x}(\varnothing)$, we have
\begin{align}
\label{eAgrI}
\sum_{x \in \bar{x}(\varnothing)} \textup{Art}_{|S|, x}(x/b, \psi_1(x)) = \sum_{p \mid b \infty} \phi_{\pi_{S - \{i_a\}}(\bar{x}), \textup{pr}_1(\pi_{i_a}(\bar{x})) \cdot \textup{pr}_2(\pi_{i_a}(\bar{x}))}(\textup{Frob } p).
\end{align}
\end{theorem}

\begin{proof}
This can be proven in the same way as Theorem \ref{tRefMin}.
\end{proof}

\begin{remark}
Write $\phi = \phi_{\pi_{S - \{i_a\}}(\bar{x}), \textup{pr}_1(\pi_{i_a}(\bar{x})) \cdot \textup{pr}_2(\pi_{i_a}(\bar{x}))}$. The agreement assumptions imply that $p$ splits completely in $M(\phi)$, so that $\textup{Frob}(p)$ lands in $\Gal(L(\phi)/M(\phi)) \cong \FF_2$.
\end{remark}

\subsection{New reflection principles}
Let $\ell$ be a prime congruent to $3$ modulo $4$. We say that a squarefree integer $d > 0$ is $\ell$-special if $\ell \mid d$ and all odd primes $p \neq \ell$ dividing $d$ satisfy $(\ell/p) = 1$. We say that $\bar{x}$ is $\ell$-special in case all its elements are $\ell$-special. If $\ell \mid x$, we write $\textup{Up}_{\Q(\sqrt{x})/\Q}(\ell)$ for the unique prime of $\Q(\sqrt{x})$ above $\ell$. We now define $\ell$-profitable triples and $-\ell$-profitable triples.

\begin{mydef}
\label{def:l-profitable}
Let $(\bar{x}, \chi_a, (\psi_{s + 1}(x))_{x \in \bar{x}(\varnothing) - \{x_0\}})$ be a $\ell$-special, profitable triple with $\bar{x} \in \overline{X}_S$. We say that  $(\bar{x}, \chi_a, (\psi_{s + 1}(x))_{x \in \bar{x}(\varnothing) - \{x_0\}})$ is a $\ell$-profitable triple if
\begin{itemize}
\item there exists a good expansion map $\phi_{\pi_S(\bar{x}), \ell}$. Furthermore, $\infty$ splits completely in the field $M(\phi_{\pi_S(\bar{x}), \ell})$ and the unique prime above $(2)$ splits completely in $M(\phi_{\pi_S(\bar{x}), \ell})/\Q(\sqrt{\ell})$;
\item all odd primes $\pi_j(\bar{x}) \neq \ell$ split completely in the field $M(\phi_{\pi_S(\bar{x}), \ell})$ for all $j \in [r] - S$;
\item we demand that $\textup{Up}_{\Q(\sqrt{x})/\Q}(\ell) \in 2^s\textup{Cl}(\mathbb{Q}(\sqrt{x}))[2^{s + 1}]$ for each $x \in \bar{x}(\varnothing)$;
\item we have for all $\varnothing \subsetneq T \subseteq S$
$$ 
\left(\sum_{\substack{x \in \bar{x}(\varnothing) \\ \forall j \in T: \pi_j(x) \neq \pi_j(x_0)}} \psi_{s - |T| + 1}(x) \right)(\sigma_\ell) = 0.
$$
\end{itemize}
We say that the triple is $-\ell$-profitable if
\begin{itemize}
\item there exists a good expansion map $\phi_{\pi_S(\bar{x}), -\ell}$. Furthermore, all places above $(2)$ split completely in $M(\phi_{\pi_S(\bar{x}), -\ell})/\Q(\sqrt{-\ell})$;
\item all odd primes $\pi_j(\bar{x}) \neq \ell$ split completely in the field $M(\phi_{\pi_S(\bar{x}), -\ell})$ for all $j \in [r] - S$;
\item we demand that $(\sqrt{x}) \cdot \textup{Up}_{\Q(\sqrt{x})/\Q}(\ell) \in 2^s\textup{Cl}(\mathbb{Q}(\sqrt{x}))[2^{s + 1}]$ for each $x \in \bar{x}(\varnothing)$;
\item we have for all $\varnothing \subsetneq T \subseteq S$
$$ 
\left(\sum_{\substack{x \in \bar{x}(\varnothing) \\ \forall j \in T: \pi_j(x) \neq \pi_j(x_0)}} \psi_{s - |T| + 1}(x) \right)(\sigma_\ell) = 0.
$$
\end{itemize}
\end{mydef}

We have the following tailor-made reflection principle for $\ell$-special integers. We remind the reader of our convention on $\text{Art}((2),L(\psi_i))$ given right after Definition \ref{dRedeiAdmissible}.

\begin{theorem} 
\label{tReflectionEll}
$(a)$ Let $(\bar{x}, \chi_a, (\psi_{s + 1}(x))_{x \in \bar{x}(\varnothing) - \{x_0\}})$ be a $\ell$-profitable triple. Then 
\[
\chi_a \in  2^s\textup{Cl}(\mathbb{Q}(\sqrt{x_0}))^\vee[2^{s + 1}].
\]
Furthermore, each prime number $p \mid a$ splits completely in $M(\phi_{\pi_S(\bar{x}), \ell})/\Q$, the unique prime above $(2)$ splits completely in $M(\phi_{\pi_S(\bar{x}), \ell})/\Q(\sqrt{\ell})$ and
$$
\sum_{x \in \bar{x}(\varnothing)} \textup{Art}_{s + 1, x}(\ell, \chi_a) = \sum_{p \mid a} \phi_{\pi_S(\bar{x}), \ell}(\textup{Frob}(p)).
$$
$(b)$ Let $(\bar{x}, \chi_a, (\psi_{s + 1}(x))_{x \in \bar{x}(\varnothing) - \{x_0\}})$ be a $-\ell$-profitable triple. Then
\[
\chi_a \in 2^s\textup{Cl}(\mathbb{Q}(\sqrt{x_0}))^\vee[2^{s + 1}]
\]
Moreover, each prime number $p \mid a$ splits completely in $M(\phi_{\pi_S(\bar{x}), -\ell})/\Q$ and
$$
\sum_{x \in \bar{x}(\varnothing)} \textup{Art}_{s + 1, x}(x/\ell, \chi_a) = \sum_{p \mid a} \phi_{\pi_S(\bar{x}), -\ell}(\textup{Frob}(p)).
$$
\end{theorem}

\begin{proof}
Let us prove part $(a)$. We will then explain the changes for part $(b)$. We choose a prime $\mathfrak{l}$, lying above $\ell$, of the field
\[
K := \Q(\{\sqrt{\text{pr}_1(\pi_i(\bar{x})) \cdot \text{pr}_2(\pi_i(\bar{x}))} : i \in S\}, \sqrt{x_0}).
\]
For each $x \in \bar{x}(\varnothing)$, the extension $\Q(\sqrt{x})$ is contained in $K$. Hence we see that
$$
K L(\psi_{s + 1}(x))/K
$$ 
is an abelian extension unramified at $\mathfrak{l}$. As such, for each $x \in \bar{x}(\varnothing)$, the Artin symbol of $\mathfrak{l}$ is a well-defined element of the group $\text{Gal}(K L(\psi_{s + 1}(x))/K)$. We now claim that
$$
\text{Art}_{s + 1, x}(\ell, \chi_a) = \psi_{s + 1}(x)|_{G_K}(\textup{Frob}(\mathfrak{l})).
$$
Indeed, $\ell$ splits completely in $\Q(\{\sqrt{\text{pr}_1(\pi_i(\bar{x})) \cdot \text{pr}_2(\pi_i(\bar{x}))} : i \in S\})$ thanks to our assumptions. This shows that 
\[
\psi_{s + 1}(x)|_{G_{\Q(\sqrt{x})}}(\text{Frob}(\text{Up}_{\Q(\sqrt{x})/\Q}(\ell))) = \psi_{s + 1}(x)|_{G_K}(\text{Frob}(\mathfrak{l})).
\]
Hence our claim follows from the definition of the Artin pairing. Summing up all the contributions, we obtain the relation
\begin{align} 
\label{elrelation}
\sum_{x \in \bar{x}(\varnothing)} \text{Art}_{s + 1, x}(\ell, \chi_a) = \psi(\bar{x}, \chi_a, (\psi_{s + 1}(x))_{x \in \bar{x}(\varnothing) - \{x_0\}})(\text{Frob}(\mathfrak{l})).
\end{align}
Observe that thanks to the last point of Definition \ref{def:l-profitable}, we have that
$$
M(\psi(\bar{x}, \chi_a, (\psi_{s + 1}(x))_{x \in \bar{x}(\varnothing) - \{x_0\}}))/\Q
$$
is unramified at $\ell$. Also observe that
$$
\psi(\bar{x}, \chi_a, (\psi_{s + 1}(x))_{x \in \bar{x}(\varnothing) - \{x_0\}})(\text{Frob}(\mathfrak{l})) = (\psi(\bar{x}, \chi_a, (\psi_{s + 1}(x))_{x \in \bar{x}(\varnothing) - \{x_0\}}) + \chi_{\ell^\ast})(\text{Frob}(\mathfrak{l})),
$$
where $\ell^\ast = -\ell$. Note that $L(\psi')/\Q$ is unramified above $\ell$ for precisely one element of the set
$$
\psi' \in \{\psi(\bar{x}, \chi_a, (\psi_{s + 1}(x))_{x \in \bar{x}(\varnothing) - \{x_0\}}), \psi(\bar{x}, \chi_a, (\psi_{s + 1}(x))_{x \in \bar{x}(\varnothing) - \{x_0\}}) + \chi_{\ell^\ast}\}.
$$ 
Denote that element by $\psi'(\bar{x}, \chi_a, (\psi_{s + 1}(x))_{x \in \bar{x}(\varnothing) - \{x_0\}})$. We now show that the $4$-tuple
\begin{align*}
(\{\chi_{\text{pr}_1(\pi_i(\bar{x})) \cdot \text{pr}_2(\pi_i(\bar{x}))}& : i \in S\} \cup \{\chi_a\}, \{\chi_{\text{pr}_1(\pi_i(\bar{x})) \cdot \text{pr}_2(\pi_i(\bar{x}))} : i \in S\} \cup \{\chi_\ell\}, \\
&\psi'(\bar{x}, \chi_a, (\psi_{s + 1}(x))_{x \in \bar{x}(\varnothing) - \{x_0\}}), (\phi_{\pi_T(\bar{x}), \ell})_{T \subseteq S})
\end{align*}
is R\'edei admissible. Once we have established this, the theorem follows from equation (\ref{elrelation}) and Theorem \ref{tRedei}. We will adopt the notation of Section \ref{sRed}, and denote the two maps in the $4$-tuple by $\psi_1$ and $\psi_2$. 

Observe that $a$ and $\ell$ are coprime in virtue of the fourth point of Definition \ref{def:l-profitable} applied with $T = S$. Since $a$ is positive by Remark \ref{rPositivea}, we therefore conclude that $\text{Ram}(\Q(\sqrt{a})) \cap \text{Ram}(\Q(\sqrt{\ell})) \subseteq \{(2)\}$. Let us now check that $\text{inv}_2(\chi_a \cup \chi_\ell) = 0$.

We distinguish two cases. If $(2)$ ramifies in $\Q(\sqrt{x_0})$, then $\chi_a$ must be in the span of $\chi_{x_0}$ locally at $2$, since $\chi_a$ is a double in the dual class group. Since $\text{Up}_{\Q(\sqrt{x_0})/\Q}(\ell) \in 2\text{Cl}(\Q(\sqrt{x_0}))[4]$, it follows that $\chi_{x_0} \cup \chi_\ell$ is locally trivial everywhere, in particular at $2$. This implies that $\chi_a \cup \chi_\ell$ is certainly locally trivial at $2$. If instead $(2)$ does not ramify in $\Q(\sqrt{x_0})$, then $\chi_a$ must be in the span of $\chi_5$ locally at $2$. Since $\ell \equiv 3 \bmod 4$, we see that $\chi_a \cup \chi_\ell$ is again locally trivial at $2$.

Hence it remains to verify the four conditions of Definition \ref{dRedeiAdmissible}. We first consider the infinite place. Note that $\ell$ is positive and so is $a$ by Remark \ref{rPositivea}. Hence we need to check that $\infty$ splits completely in $M(\psi_1)M(\psi_2)$. For $\psi_1$, this follows from the fact that the field $L(\psi_s(x))$ is totally real for each $x \in \bar{x}(\varnothing)$, in virtue of the equation $2\psi_{s + 1}(x) = \psi_s(x)$. Hence $M(\psi_1)$ is contained in the compositum of totally real fields, and so is totally real. For $\psi_2$, this is explicitly asked in Definition \ref{def:l-profitable}.

Let us now check the fourth condition of Definition \ref{dRedeiAdmissible}. Note that $\psi_2$ is a good expansion map with pointer $\chi_\ell$. Recalling that $\ell \equiv 3 \bmod 4$, we see that the extension $\Q(\sqrt{\ell})/\Q$ is ramified at $2$ and that $\chi_\ell$ vanishes at $\sigma_2(2)$. Furthermore, $\psi_2$ is a good expansion map, which forces the characters $\phi_T(\psi_2)|_{G_{\Q_2}}$ to be in the span of $\chi_5$ for each non-empty subset $\varnothing \subsetneq T \subseteq S$. The first condition of Definition \ref{def:l-profitable} implies that the unique prime of $\Q(\sqrt{\ell})$ above $(2)$ splits completely in $M(\psi_2)/\Q(\sqrt{\ell})$. This implies the triviality of $\phi_T(\psi_2)|_{G_{\Q_2}}$ for each proper non-empty subset $\varnothing \subsetneq T \subsetneq S$. 

We still need to check that the map $\phi_T(\psi_1)|_{G_{\Q_2}}$ is orthogonal to $\chi_\ell$ with respect to the local Hilbert pairing at $2$. Thanks to the third condition of Definition \ref{def:l-profitable}, applied to $x := x_0$, we see that $\chi_{x_0} \cup \chi_\ell$ is locally trivial at all places of $\Q$. Therefore the desired orthogonality follows from the last part of Theorem \ref{tProfitable}, since the character $\chi_5$ is certainly orthogonal to $\chi_\ell$ and so is $\chi_{x_0}$ as we just argued.
 
Let us now show that the third condition of Definition \ref{dRedeiAdmissible} holds. If $p$ is an odd prime not in the cube $\bar{x}$, then $\sigma_p$ is sent to $0$ by any good expansion map. Therefore such a $p$ is unramified in the field of definition of a good expansion map. We therefore derive, as a consequence of Theorem \ref{tProfitable}, that the only odd ramified primes in $L(\psi_1)L(\psi_2)/\Q$ are odd primes in the cube $\bar{x}$. Let $p$ be an odd prime of the shape $\pi_i(\bar{x})$ for some $i \in [r] - S$. We only need to check that $p$ splits completely in $M(\psi_2)/\Q$. But this is explicitly asked in the second condition of Definition \ref{def:l-profitable}. 

We will now consider the prime $\ell$. Since $\ell$ in particular divides the pointer of $\psi_2$, we need to guarantee that it is unramified in $L(\psi_1)$ and splits completely in $M(\psi_1)$. That $\ell$ is unramified in $M(\psi_1)$ follows upon combining equation (\ref{eReconstruct}) with the fourth condition of Definition \ref{def:l-profitable}. By construction of $\psi'$, this ensures that $\ell$ is also unramified in $L(\psi_1)$.

We still have to show that $\ell$ splits completely in $M(\psi_1)/\Q$. This extension is contained in the compositum of $L(\psi_s(x))$ as $x$ runs through $\bar{x}(\varnothing)$. Therefore the equation $2 \cdot \psi_{s + 1}(x) = \psi_s(x)$ proves that $\psi_s(x)$ pairs trivially with $\text{Cl}(\Q(\sqrt{x}))[2]$ and so in particular with $\text{Up}_{\Q(\sqrt{x})/\Q}(\ell)$. Hence $\ell$ is both unramified in $M(\psi_1)$ and has residue field degree $1$, thus $\ell$ must split completely in $M(\psi_1)$.

Take $i \in S$ and $j \in [2]$. We need to check the third condition of Definition \ref{dRedeiAdmissible} for places above $\text{pr}_j(\pi_i(\bar{x}))$. We claim that all places of $\Q(\{\sqrt{\text{pr}_1(\pi_i(\bar{x})) \cdot \text{pr}_2(\pi_i(\bar{x}))} : i \in S\})$ above some $\text{pr}_j(\pi_i(\bar{x}))$ are unramified in $L(\psi_1)L(\psi_2)/\Q(\{\sqrt{\text{pr}_1(\pi_i(\bar{x})) \cdot \text{pr}_2(\pi_i(\bar{x}))} : i \in S\})$. To verify our claim, we use the last statement in Theorem \ref{tProfitable} for $\psi_1$. Then, for $\psi_2$, we remark that $\sigma_{\text{pr}_j(\pi_i(\bar{x}))}$ is sent to an element of order $2$. Therefore the ramification index of $\sigma_{\text{pr}_j(\pi_i(\bar{x}))}$ in $L(\psi_2)$ equals $2$. Since $\text{pr}_j(\pi_i(\bar{x}))$ already ramifies in $\Q(\{\sqrt{\text{pr}_1(\pi_i(\bar{x})) \cdot \text{pr}_2(\pi_i(\bar{x}))} : i \in S\})$, the claim follows. This ends the proof that the $4$-tuple is R\'edei admissible, which, as explained above, ends the proof of part $(a)$. 

Let us now briefly summarize the changes for part $(b)$. Now $-\ell$ is negative so we have to make sure that $a$ is positive, which, again, follows from Remark \ref{rPositivea}. Now $-\ell \equiv 1 \bmod 4$, so we have to examine orthogonality at $2$ only in case $-\ell \equiv 5 \bmod 8$. It follows from the third condition of Definition \ref{def:l-profitable}, applied to $x := x_0$, that $x$ is odd. But then $a$ is also odd, and therefore $\chi_a$ is orthogonal to $\chi_{-\ell}$ under the local Hilbert pairing at $2$.

Finally, we examine the changes in the various splitting conditions. Since $-\ell$ is negative, we need to check that $\infty$ splits completely in $M(\psi_1)$. But this can be proved verbatim as above. The argument for $(2)$ also goes through, except that we need to check that 
\[
\text{inv}_2(\chi_{-\ell} \cup \phi_T(\psi_1)|_{G_{\Q_2}}) = 0
\]
for all $T$. If $\ell \equiv 7 \bmod 8$, this is immediate. If $\ell \equiv 3 \bmod 8$, then we have just argued that $x$ is odd for all $x \in \bar{x}(\varnothing)$. Hence it follows from the last part of Theorem \ref{tProfitable} that the quadratic character $\phi_T(\psi_1)|_{G_{\Q_2}}$ is contained in the span of the set $\{\chi_5, \chi_{-1}\}$. This gives the desired local triviality of the cup product.
\end{proof}

\section{Analytic prerequisites}
\label{sAna}
Throughout the paper our implied constants may depend on $\ell$. We shall not record this dependence. The material in this section is rather similar to \cite[Section 5]{Smith} and \cite[Section 6]{Smith}, but there is one major hurdle to overcome. Indeed, Smith does not prove the analogue of Corollary 6.11 for class groups. To do so, one needs to make the Markov chain analysis of Gerth effective. 

Another significant complicating factor is that the $4$-rank distribution in our family is different due to the fact that $N_d(x, y) = \ell$ is soluble over $\Q$ by assumption. This requires some changes to be made to the Markov chains appearing in Gerth \cite{Gerth}. These problems are dealt with in our companion paper \cite{KP3}.

\subsection{Combinatorial results}
In Section \ref{sMain} we will make essential use of the following two combinatorial results first proven in Smith \cite{Smith} with slightly different notation.

Let $X := X_1 \times \dots \times X_r$, let $S \subseteq [r]$ and let $Z \subseteq X$ with $|\pi_{[r] - S}(Z)| = 1$. We define the $\FF_2$-vector spaces
\[
V := \text{Map}(Z, \FF_2), \quad W := \text{Map}(\text{Cube}_S(Z), \FF_2),
\]
where $\text{Cube}_S(Z)$ is the set of $\bar{x} \in \overline{X}_S$ such that $\bar{x}(\varnothing) \subseteq Z$. We define a linear map $d : V \rightarrow W$, not to be confused with the map $d$ on $1$-cochains, by
\[
dF(\bar{x}) = \sum_{x \in \bar{x}(\varnothing)} F(x),
\]
where the sum has to be taken with multiplicities. This has the effect that
\[
dF(\bar{x}) = 0
\]
as soon as there exists some $i \in S$ with $\text{pr}_1(\pi_i(\bar{x})) = \text{pr}_2(\pi_i(\bar{x}))$ just like in Smith \cite[Definition 4.2]{Smith}. Define $\mathscr{G}_S(Z)$ to be the image of $d$.

\begin{lemma}
\label{lImd}
Let $X = X_1 \times \dots \times X_r$, let $S \subseteq [r]$ and suppose that $|X_i| = 1$ for $i \in [r] - S$. We have that
\[
\dim_{\FF_2} \mathscr{G}_S(X) = \prod_{i \in S} (|X_i| - 1).
\]
\end{lemma}

\begin{proof}
See \cite[Proposition 9.3]{KP}.
\end{proof}
 
Recall the definition of an additive system given in Definition \ref{dAS}. Given an additive system $\mathfrak{A}$, we set
\[
C(\mathfrak{A}) := \bigcap_{i \in S} \left\{\bar{x} \in \overline{X}_S : \bar{x}(S - \{i\}) \cap \overline{Y}_{S - \{i\}}^\circ(\mathfrak{A}) \neq \varnothing \right\}.
\]
We call an additive system $\mathfrak{A}$ on $X$ $(a, S)$-acceptable if
\begin{itemize}
\item $|A_T(\mathfrak{A})| \leq a$ for all subsets $T$ of $S$;
\item $\bar{x} \in C(\mathfrak{A})$ implies $\bar{x}(\varnothing) \subseteq \overline{Y}_\varnothing^\circ(\mathfrak{A})$.
\end{itemize}
 
\begin{prop}
\label{pdF}
There exists an absolute constant $A > 0$ such that the following holds. Let $r > 0$ be an integer, let $X_1, \dots, X_r$ be finite non-empty sets and let $X$ be their product. Take $S \subseteq [r]$ with $|S| \geq 2$, $|\pi_{[r] - S}(X)| = 1$ and put $n := \textup{min}_{i \in S} |X_i|$. Let $a \geq 2$ and $\epsilon > 0$ be given. Assume that $\epsilon < a^{-1}$ and
\[
\log n \geq A \cdot 6^{|S|} \cdot \log \epsilon^{-1}.
\]
Then there exists $g \in \mathscr{G}_S(X)$ such that for all $(a, S)$-acceptable additive systems $\mathfrak{A}$ on $X$ and for all $F: \overline{Y}_\varnothing^\circ(\mathfrak{A}) \rightarrow \mathbb{F}_2$ satisfying $d F(\bar{x}) = g(\bar{x})$ for all $\bar{x} \in C(\mathfrak{A})$, we have 
\[
\frac{|\overline{Y}_\varnothing^\circ(\mathfrak{A})|}{2} - |X| \cdot \epsilon \leq |F^{-1}(0)| \leq \frac{|\overline{Y}_\varnothing^\circ(\mathfrak{A})|}{2} + |X| \cdot \epsilon.
\]
\end{prop}

\begin{proof}
This is Proposition 4.4 in Smith \cite{Smith} and reproven in a slightly more general setting in Proposition 8.7 of \cite{KP}.
\end{proof}

\subsection{Prime divisors}
Let $\ell$ be an integer such that $|\ell|$ is a prime congruent to $3$ modulo $4$. In Theorem \ref{tMain} we are only interested in those squarefree integers $d > 0$ with the properties
\begin{align}
\label{ed}
\ell \mid d, \quad p \mid \frac{d}{\ell} &\text{ implies } \left(\frac{\ell}{p}\right) = 1 \text{ or } p = 2 \\
\label{ed2}
&\left(\frac{-d/\ell}{|\ell|}\right) = 1.
\end{align}
Indeed, this is equivalent to $\ell \mid d$ and the solubility of the equation
\[
x^2 - dy^2 = \ell \text{ in } x, y \in \Q.
\]
We remark that equation (\ref{ed2}) is equivalent to a set of congruence conditions for $d$ modulo $8$. Hence we need to insert congruence conditions in Section 5 of Smith \cite{Smith}. This has already been done in Section 10 of \cite{KP} for squarefree integers $d$ such that $p \mid d$ implies $p \equiv 0,1 \bmod \ell$, and in Section 4 of \cite{CKMP} with completely different techniques for squarefree integers $d$ such that $p \mid d$ implies $p \equiv 1, 2 \bmod 4$. Both these techniques are straightforward to generalize to obtain the following results, here we shall follow \cite[Section 5]{Smith}.

Define
\[
S(N, \ell) := \{1 \leq d < N : d \text{ squarefree and satisfies equation (\ref{ed}) and (\ref{ed2})}\} 
\]
and
\[
S_r(N, \ell) := \{d \in S(N, \ell) : \omega(d) = r\}.
\]
We list the distinct prime divisors of $d$ as $p_1 < p_2 < \dots < p_r$. With these notations we can state our next theorem.

\begin{theorem}
\label{tSpacing}
Fix an integer $\ell$ such that $|\ell|$ is a prime congruent to $3$ modulo $4$. Then there are constants $A_1, A_2 > 0$, depending only on $\ell$, such that
\begin{align}
\label{eSrNl}
\frac{A_1N}{\log N} \cdot \frac{\mu^{r - 1}}{(r - 1)!} \leq |S_r(N, \ell)| \leq \frac{A_2N}{\log N} \cdot \frac{\mu^{r - 1}}{(r - 1)!} 
\end{align}
for all $1 \leq r \leq 200\mu$ and all $N > A_2$, where we put $\mu := \frac{1}{2} \log \log N$. Furthermore, we have
\begin{align}
\label{eErdosKac}
\frac{|\{d \in S(N, \ell) : |\omega(d) - \mu| > \mu^{2/3}\}|}{|S(N, \ell)|} \ll \exp\left(-\frac{1}{3}\mu^{1/3}\right).
\end{align}
Now assume that $r$ is such that
\begin{align}
\label{errange}
|r - \mu| \leq \mu^{2/3}
\end{align}
and take $D_1 > 3$ and $C_0 > 1$. In this case we have
\begin{enumerate}
\item[(i)] the bound
\[
1 - \frac{|\{d \in S_r(N, \ell) : 2D_1 < p_i < p_{i + 1}/2 \textup{ for all } p_i > D_1\}|}{|S_r(N, \ell)|} \ll \frac{1}{\log D_1} + \frac{1}{(\log N)^{1/4}};
\]
\item[(ii)] the bound
\begin{multline*}
1 - \frac{\left|\left\{d \in S_r(N, \ell) : \left|\frac{1}{2} \log \log p_i - i\right| < C_0^{1/5} \max(i, C_0)^{4/5} \textup{ for all } i < \frac{1}{3}r\right\}\right|}{|S_r(N, \ell)|} \\
\ll \exp(-kC_0)
\end{multline*}
for some absolute constant $k$;
\item[(iii)] the bound
\begin{multline*}
\frac{\left|\left\{d \in S_r(N, \ell) : \frac{\log p_i}{\log \log p_i} \leq (\log \log \log N)^{1/2} \cdot \sum_{j = 1}^{i - 1} \log p_j \textup{ for all } \frac{1}{2}r^{1/2} < i < \frac{1}{2}r\right\}\right|}{|S_r(N, \ell)|} \\
\ll \exp\left(-(\log \log \log N)^{1/4}\right).
\end{multline*}
\end{enumerate}
\end{theorem}

\begin{proof}
Condition (\ref{ed}) is incorporated in Smith's argument just as in \cite{KP} by inserting a congruence condition in the definition of $F(x)$ \cite[p.\ 51]{Smith}. To deal with the congruence condition (\ref{ed2}), we simply impose further congruence conditions on the primes for each invertible residue class in $(\Z/8\Z)^\ast$ and then sum up the contributions. 

More explicitly, we define for a congruence class $a \in (\Z/8\Z)^\ast$ the sum
\[
F_a(x) = \sum_{\substack{p \leq x \\ (\ell/p) = 1 \\ p \equiv a \bmod 8}} \frac{1}{p}.
\]
Then there exist constants $B_a$, $A$, $c > 0$ such that
\[
\left|F_a(x) - \frac{\log \log x}{8} - B_a\right| \leq A \cdot e^{-c \sqrt{\log x}}.
\]
Let $(a_1, \dots, a_r) \in {(\Z/8\Z)^\ast}^r$ be a vector. Then for $T$ any set of tuples of primes $(p_1, \dots, p_r)$ of length $r$ with $p_i \equiv a_i \bmod 8$ and $(\ell/p) = 1$, we define the grid
\[
\text{Grid}(T) = \bigcup_{(p_1, \dots, p_r) \in T} \prod_{1 \leq i \leq r} \left[8 \cdot \left(F_{a_i}(p_i) - \frac{1}{p_i} - B_{a_i}\right), 8 \cdot \left(F_{a_i}(p_i) - B_{a_i}\right)\right].
\]
Following the proof of Smith \cite[p.\ 52]{Smith}, we compare the quantity
\[
\sum_{\substack{p_1 \cdot \ldots \cdot p_r < N \\ (\ell/p) = 1 \\ p_i \equiv a_i \bmod 8}} \frac{8^r}{p_1 \cdot \ldots \cdot p_r}
\]
against the integral $I_r$ (see \cite[p.\ 42]{Smith} for the definition of the integral $I_r$) for each vector $(a_1, \dots, a_r) \in {(\Z/8\Z)^\ast}^r$. Then equation (\ref{eSrNl}) follows from a version of \cite[Proposition 5.5]{Smith} for the set of squarefree integers $d$ satisfying equation (\ref{ed}) and the condition $p_i \equiv a_i \bmod 8$ after summing over all possible vectors $(a_1, \dots, a_r) \in {(\Z/8\Z)^\ast}^r$ such that $x^2 - a_1 \cdot \ldots \cdot a_ry^2 = \ell$ is soluble in $\Q_2$.

The assertion (\ref{eErdosKac}) is deduced from equation (\ref{eSrNl}), from standard bounds on the tails of the Poisson distribution and from a good bound for the number of integers with more than $100 \log \log N$ prime divisors. Such a bound follows immediately when one computes the average of $\tau(n)$. The claims (i), (ii) and (iii) are a straightforward generalization of the material in Section 5 of Smith \cite{Smith}.
\end{proof}

\subsection{Equidistribution of Legendre symbol matrices}
\label{ssLegendre}
In this subsection we state several equidistribution results pertaining to matrices of Legendre symbols. These results are straightforward modifications of the material in \cite{CKMP, Smith}. We start with two definitions.

\begin{mydef}
\label{dSiegel}
Suppose that $\ell$ is a non-zero integer. A prebox is a pair $(X, P)$ satisfying
\begin{itemize}
\item $P$ consists entirely of prime numbers such that the images of $P$, $\ell$ and $-1$ are linearly independent in $\frac{\Q^\ast}{\Q^{\ast2}}$;
\item $X = X_1 \times \dots \times X_r$, where each $X_i$ consists entirely of prime numbers with $X_i \cap P = \varnothing$;
\item there exists a sequence of real numbers
\[
0 < s_1 < t_1 < s_2 < t_2 < \dots < s_r < t_r
\]
such that every prime $p \in X_i$ satisfies $s_i < p < t_i$ and $(\ell/p) = 1$.
\end{itemize}
Define the (potentially infinite) sequence $d_1, d_2, \ldots$ as in Definition 6.2 of Smith \cite{Smith}. Then we have $d_i^2 < |d_{i + 1}|$. We say that $(X, P)$ is Siegel-less above $t$ if for all $x \in X$ we have that $d_i \mid \ell x \prod_{p \in P} p$ implies $|d_i| < t$.
\end{mydef}

\begin{mydef}
Write $A \sqcup B$ for the disjoint union of two sets $A$ and $B$. Let $(X, P)$ be a prebox. Put
\[
M := \{(i, j) : 1 \leq i < j \leq r\}, \quad M_P := [r] \times (P \sqcup \{-1\}).
\]
Let $\mathcal{M} \subseteq M$ and let $\mathcal{N} \subseteq M_P$. Given a map $a : \mathcal{M} \sqcup \mathcal{N} \rightarrow \{\pm 1\}$, we define $X(a)$ to be the subset of tuples $(x_1, \ldots, x_r) \in X$ with
\[
\left(\frac{x_i}{x_j}\right) = a(i, j) \textup{ for all } (i, j) \in \mathcal{M}, \quad \left(\frac{p}{x_i}\right) = a(i, p) \textup{ for all } (i, p) \in \mathcal{N}.
\]
\end{mydef}

Ideally we would like to show that every $X(a)$ is of the expected size. Although we are not able to prove this in full generality, we will prove slightly weaker results that still suffice for our application. 

\begin{prop}
\label{p6.3}
Let $\ell$ be a non-zero integer. For every choice of positive constants $c_1, \dots, c_8$ satisfying $c_3 > 1$, $c_5 > 3$ and
\[
\frac{1}{8} > c_8 + \frac{c_7 \log 2}{2} + \frac{1}{c_1} + \frac{c_2c_4}{2},
\]
there exists a constant $A$ such that the following holds. 

Let $A < t < s_1$ and suppose that $(X, P)$ is a prebox that is Siegel-less above $t$. Let $\mathcal{M} \subseteq M$ and let $\mathcal{N} \subseteq M_P$. Let $1 \leq k \leq r$ be an integer such that $(i, p) \in \mathcal{M}$ implies $i > k$. Furthermore, if $i > k$, we assume that $X_i$ equals the set of primes in $s_i < q < t_i$ satisfying
\[
\left(\frac{\ell}{q}\right) = 1 \textup{ and } \left(\frac{p}{q}\right) = a(i, p) \textup{ for all } (i, p) \in \mathcal{N}.
\]
Finally, assume that
\begin{itemize}
\item[(i)] $p \in P$ implies $p < s_1$ and $|P| \leq \log t_i - i$ for all $1 \leq i \leq r$;
\item[(ii)] $\log t_k < t_1^{c_2}$ and if $k < r$, we assume that $\log t_{k + 1} > \max((\log t_1)^{c_5}, t^{c_6})$;
\item[(iii)] we assume that for all $1 \leq i \leq r$
\[
|X_i| \geq \frac{2^{c_3i} \cdot t_i}{(\log t_i)^{c_4}};
\]
\item[(iv)] $r^{c_1} < t_1$;
\item[(v)] putting $j_i := i - 1 + \lfloor c_7 \log t_i \rfloor$, we assume that $j_1 > k$. Furthermore, $j_i \leq r$ implies
\[
(\log t_i)^{c_5} < \log t_{j_i}.
\]
\end{itemize}
Then we have for all $a : \mathcal{M} \sqcup \mathcal{N} \rightarrow \{\pm 1\}$
\[
\left||X(a)| - \frac{|X|}{2^{|\mathcal{M}|}}\right| \leq t_1^{-c_8} \cdot \frac{|X|}{2^{|\mathcal{M}|}}.
\]
\end{prop}

\begin{proof}
This is a straightforward generalization of \cite[Proposition 6.3]{Smith} and \cite[Proposition 5.7]{CKMP}.
\end{proof}

Our next proposition deals with the small primes but at the cost of introducing permutations. We define $\mathcal{P}(k_2)$ to be the set of permutations $\sigma : [r] \rightarrow [r]$ that fix every $k_2 < i \leq r$. Furthermore, if $a: M \sqcup M_P \rightarrow \{\pm 1\}$, we define $\sigma(a)$ to be
\[
\sigma(a)(i, j) = a(\sigma(i), \sigma(j)), \quad \sigma(a)(i, p) = a(\sigma(i), p).
\]
Here we use the convention that for $i > j$
\[
a(i, j) := a(j ,i) \cdot (-1)^{\frac{a(i, -1) - 1}{2} \cdot \frac{a(j, -1) - 1}{2}}.
\]

\begin{prop}
\label{p6.4}
Let $\ell$ be a non-zero integer. For all choices of positive constants $c_1, \dots, c_{12}$ satisfying $c_3 > 1$, $c_5 > 3$ and
\[
\frac{1}{8} > c_8 + \frac{c_7 \log 2}{2} + \frac{1}{c_1} + \frac{c_2c_4}{2}, \quad c_{10} \log 2 + 2c_{11} + c_{12} < 1 \quad \textup{and} \quad c_{12} + c_{11} < c_9,
\]
there exists a constant $A$ such that the following holds. 

Let $A < t$ and suppose that $(X, P)$ is a prebox that is Siegel-less above $t$ such that $X_i$ equals the set of primes $p$ in the interval $(s_i, t_i)$ satisfying $(\ell/p) = 1$. Let $k_0, k_1, k_2$ be integers such that $0 \leq k_0 < k_1 < k_2 \leq r$. We assume that
\begin{itemize}
\item[(i)] $p \in P$ implies $p < s_{k_0 + 1}$ and $|P| \leq \log t_i - i$ for all $i > k_0$;
\item[(ii)] $\log t_{k_1} < t_{k_0 + 1}^{c_2}$ and $\log t_{k_1 + 1} > \max((\log t_{k_0 + 1})^{c_5}, t^{c_6})$;
\item[(iii)] for all $i > k_0$
\[
|X_i| \geq \frac{2^{|P| + c_3i} \cdot k_2^{c_9} \cdot t_i}{(\log t_i)^{c_4}};
\]
\item[(iv)] $r^{c_1} < t_{k_0 + 1}$;
\item[(v)] we assume that $k_1 - k_0 < c_7 \log t_{k_0 + 1}$. Furthermore, $i > k_0$ and $i - 1 + \lfloor c_7 \log t_i\rfloor \leq j \leq r$ implies
\[
(\log t_i)^{c_5} < \log t_j;
\]
\item[(vi)] $k_2 > A$ and $s_{k_0 + 1} > t$;
\item[(vii)] $c_{10} \log k_2 > |P| + k_0$ and $c_{11} \log k_2 > \log k_1$.
\end{itemize}
Then we have
\[
\sum_{a : M \sqcup M_P \rightarrow \{\pm 1\}} \left|2^{-|M \sqcup M_P|} \cdot k_2! \cdot |X| - \sum_{\sigma \in \mathcal{P}(k_2)} |X(\sigma(a))|\right| \leq \left(k_2^{-c_{12}} + t_{k_0 + 1}^{-c_8}\right) \cdot k_2! \cdot |X|.
\]
\end{prop}

\begin{proof}
This is a straightforward generalization of \cite[Theorem 6.4]{Smith} and \cite[Theorem 5.9]{CKMP}.
\end{proof}

\subsection{Boxes}
We now define boxes. Boxes are product spaces of the shape $X := X_1 \times \dots \times X_r$, where $X_1, \dots, X_r$ are ``nice'' sets of primes. Then we state an important proposition that allows us to transition from squarefree integers to boxes.

\begin{mydef}
Let $\ell$ be such that $|\ell|$ is a prime $3$ modulo $4$. Suppose that $D_1 > \max(100, |\ell|)$ is a real number and let $1 \leq k \leq r$ be integers. Let $\mathfrak{t} := (p_1, \dots, p_k, t_{k + 1}, \dots, t_r)$ be a tuple satisfying the following properties
\begin{itemize}
\item the $p_i$ are prime numbers satisfying $p_1 < \dots < p_k < D_1$ and the $t_j$ are real numbers with $D_1 < t_{k + 1} < \dots < t_r$;
\item we have $|\ell| \in \{p_1, \dots, p_k\}$ and we have for all $i = 1, \dots, k$ that $\gcd(2 \ell, p_i) > 1$ or $\left(\frac{\ell}{p_i}\right) = 1$.
\end{itemize}
To $\mathfrak{t}$ we associate a box $X := X_1 \times \dots \times X_r$ as follows: we set $X_i := \{p_i\}$ for $1 \leq i \leq k$, while for $i > k$ we let $X_i$ be the set of prime numbers $p$ with $\left(\frac{\ell}{p}\right) = 1$ in the interval
\[
\left(t_i, \left(1 + \frac{1}{e^{i - k} \log D_1}\right) \cdot t_i\right).
\]
\end{mydef}

Note that for $\ell = -1$ this is the same as Definition 5.12 in \cite{CKMP}. Furthermore, we can turn any box into a prebox by removing $\{|\ell|\}$ and taking $P = \varnothing$. We define
\[
S^\ast(N, \ell) := \{1 \leq d < N : d \text{ squarefree and satisfies equation (\ref{ed})}\} 
\]
and
\[
S_r^\ast(N, \ell) := \{d \in S^\ast(N, \ell) : \omega(d) = r\}.
\]
Then there is a natural injective map $i: X \rightarrow S_r^\ast(\infty, \ell)$, which is a superset of $S_r(\infty, \ell)$. Hence it makes sense to speak of the intersection $i(X) \cap V$ for $V$ a subset of $S_r(\infty, \ell)$. We can now state our analogue of Proposition 6.9 in Smith \cite{Smith}.

\begin{prop}
\label{pRed}
Take $\ell$ to be an integer such that $|\ell|$ is a prime $3$ modulo $4$. Let $N \geq D_1 > \max(100, |\ell|)$ and $\log \log N \geq 2 \log \log D_1$. Take any $r$ satisfying equation (\ref{errange}). Let $V, W$ be subsets of $S_r(N, \ell)$ with the additional requirement that
\[
W \subseteq \{d \in S_r(N, \ell) : 2D_1 < p_i < p_{i + 1}/2 \textup{ for all } p_i > D_1\}.
\]
Take any $\epsilon > 0$ with
\[
|W| > (1 - \epsilon) \cdot |S_r(N, \ell)|.
\]
Assume that there exists a real number $\delta > 0$ such that for all boxes $X$ with $i(X) \subseteq S_r^\ast(N, \ell)$ and $i(X) \cap W \neq \varnothing$ we have
\begin{align}
\label{eBoxControl}
(\delta - \epsilon) \cdot |i(X) \cap S_r(N, \ell)| \leq |i(X) \cap V| \leq (\delta + \epsilon) \cdot |i(X) \cap S_r(N, \ell)|.
\end{align}
Then
\[
|V| = \delta \cdot |S_r(N, \ell)| + O\left(\left(\epsilon + \frac{1}{\log D_1}\right) \cdot |S_r(N, \ell)|\right).
\]
\end{prop}

\begin{proof}
This is a straightforward adaptation of Proposition 6.9 in Smith \cite{Smith}.
\end{proof}

When we apply Propositions \ref{p6.3} and \ref{p6.4}, we need to ensure the Siegel-less condition, i.e. we need to avoid all boxes $X$ such that there are $x \in X$ and some $i$ with $|d_i| > D_1$ and $d_i \mid x$. To do so, we shall add the union of all such boxes $X$ to $W$. Therefore it is important to show that this union is small, and this is exactly what the following proposition does.

\begin{prop}
\label{pSiegel}
Let $\ell$ be an integer such that $|\ell|$ is a prime. Take $N$ and $r$ satisfying equation (\ref{errange}). Also take $N \geq D_1 > \max(100, |\ell|)$ with $\log \log N \geq 2 \log \log D_1$. Let $f_1, f_2, \ldots$ be any sequence of squarefree integers greater than $D_1$ satisfying $f_i^2 < f_{i + 1}$. Define
\[
W_i := \{d \in S_r(N, \ell) : \textup{there is a box } X \textup{ with } d \in X \textup{ and } f_i \mid x \textup{ for some } x \in X\}.
\]
Then we have
\[
\left|\bigcup_{i = 1}^\infty W_i\right| \ll \frac{|S_r(N, \ell)|}{\log D_1}.
\]
\end{prop}

\begin{proof}
This is a small generalization of Theorem 5.14 in \cite{CKMP}, which is based on Proposition 6.10 in Smith \cite{Smith}.
\end{proof}

\subsection{R\'edei matrices}
The previous subsections provide us with enough tools to deal with the $4$-rank distribution in our family of discriminants. Analogous results for $\ell = -1$ can be found in \cite[Section 5]{CKMP}. We now define the R\'edei matrix associated to a squarefree integer $d > 0$.

\begin{mydef}
\label{dMatrix}
Let $d > 0$ be a squarefree integer and suppose that $\Delta_{\Q(\sqrt{d})}$ has $t$ prime divisors, say $p_1 < \dots < p_t$. We can uniquely decompose $\chi_d$ as 
\[
\chi_d = \sum_{i = 1}^t \chi_i,
\]
where $\chi_i: G_\Q \rightarrow \FF_2$ has conductor a power of $p_i$. In case $p_i \neq 2$, we have $\chi_i = \chi_{p_i^\ast}$, where $p_i^\ast$ has the same absolute value as $p_i$ and is $1$ modulo $4$. When $p_i = 2$, we have $\chi_i \in \{\chi_{-4}, \chi_{-8}, \chi_8\}$.

The R\'edei matrix $R(d)$ is a $t \times t$ matrix with entry $(i, j)$ equal to
\[
\chi_j(\textup{Frob } p_i) \textup{ if } i \neq j, \quad \sum_{k \neq i} \chi_k(\textup{Frob } p_i) \textup{ if } i = j,
\]
so the sum of every row is zero.
\end{mydef}

It is a classical fact, going back to its namesake R\'edei, that
\[
\textup{rk}_4 \ \textup{Cl}(\Q(\sqrt{d})) = t - 1 - \textup{rk} \ R(d).
\]
One of the pleasant properties of $X(a)$ is that all $x \in X(a)$ have the same R\'edei matrix, and hence the same $4$-rank. There are several constraints for the possible shapes of the R\'edei matrix. First of all, there is quadratic reciprocity that relates the entry $(i, j)$ with $(j, i)$. Second of all, if $d \in S(N, \ell)$, then there are further constraints coming from equation (\ref{ed}) and equation (\ref{ed2}). We will now indicate what conditions this forces on $a$.

\begin{mydef}
\label{dValida}
Let $X$ be a box corresponding to $\mathbf{t} = (p_1, \dots, p_k, t_{k + 1}, \dots, t_r)$ and let $\tilde{j}$ be the index for which $X_{\tilde{j}} = \{|\ell|\}$. We define $\textup{Map}(M \sqcup M_\varnothing, \{\pm 1\})$ to be the set of maps from $M \sqcup M_\varnothing$ to $\{\pm 1\}$. Put $\widetilde{\textup{Map}}(M \sqcup M_\varnothing, \{\pm 1\}, \tilde{j}, \ell)$ to be the subset of $\textup{Map}(M \sqcup M_\varnothing, \{\pm 1\})$ satisfying 
\begin{itemize}
\item if $X_1 \neq \{2\}$ and $\ell > 0$, then $a(i, \tilde{j}) = a(i, -1)$ for all $i < \tilde{j}$, $a(\tilde{j}, i) = 1$ for all $i > \tilde{j}$ and
\[
\prod_{i = 1}^r a(i, -1) = 1;
\]
\item if $X_1 \neq \{2\}$ and $\ell < 0$, then $a(i, \tilde{j}) = 1$ for all $i < \tilde{j}$, $a(\tilde{j}, i) = a(i, -1)$ for all $i > \tilde{j}$;
\item if $X_1 = \{2\}$ and $\ell > 0$, then $a(i, \tilde{j}) = a(i, -1)$ for all $2 \leq i < \tilde{j}$, $a(\tilde{j}, i) = 1$ for all $i > \tilde{j}$ and
\[
\prod_{i = 1}^r a(i, -1) = \left(\frac{2}{|\ell|}\right);
\]
\item if $X_1 = \{2\}$ and $\ell < 0$, then $a(i, \tilde{j}) = 1$ for all $2 \leq i < \tilde{j}$, $a(\tilde{j}, i) = a(i, -1)$ for all $i > \tilde{j}$ and $\ell \equiv 1 \bmod 8$.
\end{itemize}
\end{mydef}

We will now describe exactly the kind of boxes that we will be working with for the rest of the paper.

\begin{mydef}
\label{dW}
Let $X$ be a box and let $N$ be a real number. Put
\[
D_1 := e^{(\log \log N)^{\frac{1}{10}}}, \quad C_0 := \frac{\log \log \log N}{100}, \quad C_0' := \sqrt{\log \log \log N}.
\]
We let $W$ be the largest subset of $S_r(N, \ell)$ satisfying
\begin{itemize}
\item the requirement $W \cap W_i = \varnothing$ for all $i \geq 1$, where $W_i$ is the set constructed in Proposition \ref{pSiegel} using the sequence $d_i$ from Definition \ref{dSiegel};
\item the requirement
\begin{align}
\label{eComfortable}
W \subseteq \{d \in S_r(N, \ell) : 2D_1 < p_i < p_{i + 1}/2 \textup{ for all } p_i > D_1\};
\end{align}
\item and the requirement
\[
W \subseteq \left\{d \in S_r(N, \ell) : \left|\frac{1}{2} \log \log p_i - i\right| < C_0^{1/5} \max(i, C_0)^{4/5}\right\}.
\]
\end{itemize}
We say that $X$ is $N$-decent if $r$ satisfies equation (\ref{errange}), $i(X) \subseteq S_r^\ast(N, \ell)$ and $i(X) \cap W \neq \varnothing$. Now let $W'$ be the largest subset of $W$ satisfying
\begin{itemize}
\item the requirement
\begin{align}
\label{eRegular}
W' \subseteq \left\{d \in S_r(N, \ell) : \left|\frac{1}{2} \log \log p_i - i\right| < C_0'^{1/5} \max(i, C_0')^{4/5}\right\};
\end{align}
\item and the requirement that for every $d \in W'$ there is some $i$ with $\frac{1}{2} r^{1/2} < i < \frac{1}{2} r$ and
\begin{align}
\label{eExtravagant}
\frac{\log p_i}{\log \log p_i} > (\log \log \log N)^{1/2} \cdot \sum_{j = 1}^{i - 1} \log p_j.
\end{align}
\end{itemize}
We say that $X$ is $N$-good if $X$ is $N$-decent and $i(X) \cap W' \neq \varnothing$. 
\end{mydef}

The main point of Definition \ref{dW} is that we can apply the results in Subsection \ref{ssLegendre} to these boxes provided that $N$ is sufficiently large. Let $P(m, n, j)$ be the probability that a randomly chosen $m \times n$ matrix with coefficients in $\FF_2$ has right kernel of rank $j$. Then we have the explicit formula
\[
P(m, n, j) = \frac{1}{2^{nm}} \prod_{i = 0}^{n - j - 1} \frac{(2^m - 2^i)(2^n - 2^i)}{2^{n - j} - 2^i},
\]
which we will use throughout the paper. For the remainder of this paper, $\iota$ denotes the unique group isomorphism between $\{\pm 1\}$ and $\FF_2$. To prove the next theorem, it suffices to work with $N$-decent boxes $X$, while we will work with $N$-good boxes in Section \ref{sMain}.

\begin{theorem}
\label{t4rank}
Let $\ell$ be such that $|\ell|$ is a prime $3$ modulo $4$. Then we have for all $k \geq 0$
\[
\left|\lim_{s \rightarrow \infty} P(s, s, k) \cdot |S(N, \ell)| - \left|\left\{d \in S(N, \ell) : \textup{rk}_4 \ \textup{Cl}(\Q(\sqrt{d})) = k\right\}\right|\right| = O\left(\frac{|S(N, \ell)|}{(\log \log N)^c}\right)
\]
for some absolute constant $c > 0$.
\end{theorem}

\begin{proof}
Thanks to equation (\ref{eErdosKac}), it suffices to prove the theorem for $S_r(N, \ell)$ with $r$ an integer satisfying equation (\ref{errange}). Since one easily bounds the differences
\[
\left|\lim_{s \rightarrow \infty} P(s, s, k) - P(r - 1, r - 1, k)\right|,
\]
we may work with $P(r - 1, r - 1, k)$ instead. We now follow the proof of Smith \cite[Corollary 6.11]{Smith}. The first step is to reduce to $N$-decent $X$, for which we use our Theorem \ref{tSpacing}, Proposition \ref{pRed} and Proposition \ref{pSiegel}.

Now let $X = X_1 \times \dots \times X_r$ be an $N$-decent box, so in particular $i(X) \subseteq S_r^\ast(N, \ell)$. It suffices to prove that
\begin{multline}
\label{e4rankX}
\left|P(r - 1, r - 1, k) \cdot |i(X) \cap S_r(N, \ell)| - \left|\left\{d \in i(X) \cap S_r(N, \ell) : \textup{rk}_4 \ \textup{Cl}(\Q(\sqrt{d})) = k\right\}\right|\right| \\
= O\left(\frac{|X|}{(\log \log N)^c}\right)
\end{multline}
for some absolute constant $c > 0$, since $|X| \ll |i(X) \cap S_r(N, \ell)|$. We now apply Proposition \ref{p6.4} to the box $X'$ with $X_{\tilde{j}}$ removed. Then we obtain an absolute constant $c'$ such that
\begin{align}
\label{ePermaPrime}
\sum_{a \in \widetilde{\text{Map}}(M' \sqcup M'_\varnothing, \{\pm 1\})} \left| 2^{-|M'| - |M'_\varnothing|} \cdot (r - 1)! \cdot |X'| - \sum_{\sigma \in \mathcal{P}(r - 1)} |X'(\sigma(a))| \right| \leq \frac{(r - 1)! \cdot |X'|}{(\log \log N)^{c'}},
\end{align}
where $M' = \{(i, j) : 1 \leq i < j \leq r - 1\}$ and $M'_\varnothing = [r - 1] \times \{-1\}$. Let $S$ be the set of permutations of $[r]$ that fix $\tilde{j}$. Then equation (\ref{ePermaPrime}) implies
\begin{align}
\label{ePerma}
\sum_{a \in \widetilde{\text{Map}}(M \sqcup M_\varnothing, \{\pm 1\}, \tilde{j}, \ell)} \left| 2^{-|M'| - |M'_\varnothing|} \cdot (r - 1)! \cdot |X| - \sum_{\sigma \in S} |X(\sigma(a))| \right| \leq \frac{(r - 1)! \cdot |X|}{(\log \log N)^{c'}}.
\end{align}
Note that if $a \in \widetilde{\text{Map}}(M \sqcup M_\varnothing, \{\pm 1\}, \tilde{j}, \ell)$, then so is $\sigma(a)$ for any permutation $\sigma \in S$. Also observe that $a \in \widetilde{\text{Map}}(M \sqcup M_\varnothing, \{\pm 1\}, \tilde{j}, \ell)$ implies $i(X(a)) \subseteq S_r(N, \ell)$. Furthermore, if $i(X(a)) \cap S_r(N, \ell) \neq \varnothing$, then we certainly have $a \in \widetilde{\text{Map}}(M \sqcup M_\varnothing, \{\pm 1\}, \tilde{j}, \ell)$.

Set
\[
Q(X, k, \ell) := \frac{|\{a \in \widetilde{\text{Map}}(M \sqcup M_\varnothing, \{\pm 1\}, \tilde{j}, \ell) : \dim_{\FF_2} \text{ker}(A) = k + 1\}|}{|\widetilde{\text{Map}}(M \sqcup M_\varnothing, \{\pm 1\}, \tilde{j}, \ell)|},
\]
where $A$ is the R\'edei matrix associated to $a$ in the obvious way. Note that the the matrix $A'$ associated to $\sigma(a)$ has the same rank as the matrix $A$ associated to $a$. Then, because of equation (\ref{ePerma}), it is enough to show that there is an absolute constant $c > 0$ such that
\[
\left|P(r - 1 , r - 1, k) - Q(X, k, \ell)\right| = O\left(\frac{1}{(\log \log N)^c}\right)
\]
for every $r$ satisfying equation (\ref{errange}). But this follows from \cite[Theorem 4.8]{KP3}.
\end{proof}

\section{Proof of main theorems}
\label{sMain}
The aim of this section is to prove Theorem \ref{tMain}. Let $D_{\ell, k}(n)$ be the set of squarefree integers $d$ divisible by $\ell$ such that $\text{rk}_{2^k} \text{Cl}(\Q(\sqrt{d})) = n$ and furthermore $\text{Up}_{\Q(\sqrt{d})/\Q}(\ell) \in 2^{k - 1} \text{Cl}(\Q(\sqrt{d}))[2^k]$ if $\ell > 0$ and $\text{Up}_{\Q(\sqrt{d})/\Q}(\ell) \cdot (\sqrt{d}) \in 2^{k - 1} \text{Cl}(\Q(\sqrt{d}))[2^k]$ if $\ell < 0$. Then we have the decomposition
\[
\bigcup_{n = 0}^\infty D_{\ell, 2}(n) = S(\infty, \ell).
\]
Our next theorem is very much in spirit of the heuristical assumptions that led to Stevenhagen's conjecture \cite{Stevenhagen}. The first equality in Theorem \ref{tMain} is an immediate consequence of Theorem \ref{t4rank}, the material in Appendix \ref{aSte} and the theorem below. The second equality in Theorem \ref{tMain} is a consequence of the identity $\gamma = 1/2$, proven in Appendix \ref{aGamma}.

\begin{theorem}
\label{tHeuristic}
Let $\ell$ be an integer such that $|\ell|$ is a prime $3$ modulo $4$. There are $c, A, N_0 > 0$ such that for all integers $N > N_0$, all integers $m \geq 2$ and all sequences of integers $n_2 \geq \dots \geq n_m \geq n_{m + 1} \geq 0$
\[
\left|\left|[N] \cap \bigcap_{i = 2}^{m + 1} D_{\ell, i}(n_i)\right| - \frac{P(n_m, n_m, n_{m + 1})}{2^{n_m}} \cdot \left|[N] \cap \bigcap_{i = 2}^m D_{\ell, i}(n_i)\right|\right| \leq \frac{A \cdot |S(N, \ell)|}{(\log \log \log \log N)^{\frac{c}{m^2 6^m}}}.
\]
\end{theorem}

To prove Theorem \ref{tHeuristic}, our first step is to reduce to boxes with some nice properties. Definition \ref{dW} precisely pinpoints the boxes for which we will prove the desired equidistribution. We will now state a proposition and prove that the proposition implies Theorem \ref{tHeuristic}, so that it remains to prove the proposition.

\begin{prop}
\label{pBox}
Let $\ell$ be an integer such that $|\ell|$ is a prime $3$ modulo $4$. There are $c, A, N_0 > 0$ such that for all integers $N > N_0$, all integers $m \geq 2$, all sequences of integers $n_2 \geq \dots \geq n_m \geq n_{m + 1} \geq 0$ and all $N$-good boxes $X$
\begin{multline*}
\left|\left|i(X) \cap \bigcap_{i = 2}^{m + 1} D_{\ell, i}(n_i)\right| - \frac{P(n_m, n_m, n_{m + 1})}{2^{n_m}} \cdot \left|i(X) \cap \bigcap_{i = 2}^m D_{\ell, i}(n_i)\right|\right| \leq \\
\frac{A \cdot |i(X) \cap S_r(N, \ell)|}{(\log \log \log \log N)^{\frac{c}{m^2 6^m}}}.
\end{multline*}
\end{prop}

\begin{proof}[Proof that Proposition \ref{pBox} implies Theorem \ref{tHeuristic}] Due to equation (\ref{eErdosKac}) we may restrict to $S_r(N, \ell)$ with $r$ satisfying equation (\ref{errange}). Let $D_1$ and $W'$ be as in Definition \ref{dW}. Part (i), (ii) and (iii) of Theorem \ref{tSpacing} give upper bounds for the complements of the sets appearing in equation (\ref{eComfortable}), equation (\ref{eRegular}) and equation (\ref{eExtravagant}) respectively. Furthermore, Proposition \ref{pSiegel} shows that most $d \in S_r(N, \ell)$ are outside the union of the $W_i$. Therefore we see that there is an absolute constant $C > 0$ with
\[
|W'| > \left(1 - \frac{C}{\exp\left(\left(\log \log \log N\right)^{1/4}\right)}\right) \cdot |S_r(N, \ell)|.
\]
We now apply Proposition \ref{pRed} two times with respectively
\[
V _1 := S_r(N, \ell) \cap \bigcap_{i = 2}^{m + 1} D_{\ell, i}(n_i), \quad V_2 := S_r(N, \ell) \cap \bigcap_{i = 2}^m D_{\ell, i}(n_i)
\]
and with the above $D_1$ and $W'$ in both cases. Theorem \ref{t4rank} and Proposition \ref{pBox} ensure that equation (\ref{eBoxControl}) is satisfied. Then we get
\[
V_1 = \lim_{s \rightarrow \infty} P(s, s, n_2) \cdot \prod_{i = 2}^m \frac{P(n_i, n_i, n_{i + 1})}{2^{n_i}} \cdot |S_r(N, \ell)| + O\left(\frac{|S_r(N, \ell)|}{(\log \log \log \log N)^{\frac{c}{m^2 6^m}}}\right)
\]
and
\[
V_2 = \lim_{s \rightarrow \infty} P(s, s, n_2) \cdot \prod_{i = 2}^{m - 1} \frac{P(n_i, n_i, n_{i + 1})}{2^{n_i}} \cdot |S_r(N, \ell)| + O\left(\frac{|S_r(N, \ell)|}{(\log \log \log \log N)^{\frac{c}{m^2 6^m}}}\right).
\]
This quickly implies Theorem \ref{tHeuristic}.
\end{proof}

Our next goal is to fix the first R\'edei matrix. In other words, we split $X$ into the union $X(a)$ with $a$ running over all maps from $M \sqcup M_\varnothing$ to $\{\pm 1\}$. If $S \subseteq [r]$, $Q \in \prod_{i \in S} X_i$ and $j \in S$, we write
\[
X_j(a, Q) := \left\{x_j \in X_j : \left(\frac{\pi_i(Q)}{x_j}\right) = a(i, j) \textup{ for } i \in S \textup{ and } \left(\frac{p}{x_j}\right) = a(j, p) \textup{ for } p \in P\right\}.
\]
We also define $X(a, Q)$ to be the subset of $x \in X(a)$ with $\pi_S(x) = Q$. Smith's method does not prove equidistribution for all $a$, but only for most $a$. This prompts our next definition. 

\begin{mydef}
\label{dRedei}
Let $X$ be a $N$-good box and let $a \in \textup{Map}(M \sqcup M_\varnothing, \{\pm 1\})$. Set
\[
r'(a, X) :=
\left\{
	\begin{array}{ll}
		r  & \mbox{\textup{if} } X_1 = \{2\} \textup{ or } \prod_{i = 1}^r a(i, -1) = 1 \\
		r + 1 & \mbox{\textup{otherwise.}}
	\end{array}
\right.
\]
Recall that we associated a $r'(a, X) \times r'(a, X)$ matrix $A$ with coefficients in $\FF_2$ to the map $a \in \textup{Map}(M \sqcup M_\varnothing, \{\pm 1\})$ during the proof of Theorem \ref{t4rank}, which is simply the R\'edei matrix of $x$ for any choice of $x \in X(a)$. Let $V$ be the vector space $\FF_2^{r'(a, X)}$. We define
\[
D_{a, 2} := \{v \in V : v^TA = 0\}, \quad D_{a, 2}^\vee := \{v \in V : Av = 0\}.
\]
Put
\[
n_{\textup{max}} := \left \lfloor \sqrt{\frac{c'}{m^2 6^m} \log \log \log \log \log N} \right \rfloor, \quad n_{a, 2} := \dim_{\FF_2} D_{a, 2} - 1,
\]
where $c'$ is a constant specified later. Let $X$ be a $N$-good box and let $\tilde{j}$ be the index such that $X_{\tilde{j}} = \{|\ell|\}$. We define the vectors $R := (1, 1, \dots, 1) \in D_{a, 2}^\vee$ and
\[
C :=
\left\{
\begin{array}{ll}
(1, 1, \dots, 1)  & \mbox{\textup{if} } X_1 = \{2\} \textup{ or } \prod_{i = 1}^r a(i, -1) = 1 \\
(0, 1, 1, 1, \dots, 1) & \mbox{\textup{otherwise.}} 
\end{array}
\right.
\]
We next define the vector
\[
L := 
\left\{
\begin{array}{ll}
(0, \dots, 0, 1, 0, \dots, 0) \in D_{a, 2}  & \mbox{\textup{if} } \ell > 0 \\
(0, \dots, 0, 1, 0, \dots, 0) + C \in D_{a, 2} & \mbox{\textup{if} } \ell < 0,
\end{array}
\right.
\]
where $(0, \dots, 0, 1, 0, \dots, 0)$ has a $1$ exactly on the $\tilde{j}$-th position. Since $\ell \mid d$, the solubility of $x^2 - dy^2 = \ell$ in $x, y \in \Q$ is precisely equivalent to $L$ being in $D_{a, 2}$. We fix a choice of an index $i$ satisfying equation (\ref{eExtravagant}) and we call it $k_{\textup{gap}}$. Then we say that $a \in \textup{Map}(M \sqcup M_\varnothing, \{\pm 1\})$ is $(N, m, X)$-acceptable if the following conditions are satisfied
\begin{itemize}
\item $n_{a, 2} \leq n_{\textup{max}}$;
\item we have $a \in \widetilde{\textup{Map}}(M \sqcup M_\varnothing, \{\pm 1\}, \tilde{j}, \ell)$, see Definition \ref{dValida};
\item we have for all $j > k$
\begin{align}
\label{eXjlarge}
|X_j(a, Q)| \geq \frac{|X_j|}{(\log t_{k + 1})^{100}},
\end{align}
where $Q$ is the unique point in $X_1 \times \dots \times X_k$;
\item putting
\[
S_{\textup{pre}} := \left\{i \in [r] : \frac{k_{\textup{gap}}}{2} \leq i < k_{\textup{gap}}\right\}, \quad \alpha_{\textup{pre}} := \left|S_{\textup{pre}}\right|
\]
and
\[
S_{\textup{post}} := \left\{i \in [r] : k_{\textup{gap}} < i \leq 2k_{\textup{gap}}\right\}, \quad \alpha_{\textup{post}} := \left|S_{\textup{post}}\right|,
\]
we have for all $T_1 \in D_{a, 2}$, $T_2 \in D_{a, 2}^\vee$ such that $T_1 \not \in \langle L \rangle$ or $T_2 \not \in \langle R \rangle$
\begin{align}
\label{eGen2}
\left|\left|\left\{i \in S_{\textup{pre}} : \pi_i(T_1 + T_2) = 0\right\}\right| - \frac{\alpha_{\textup{pre}}}{2}\right| \leq \frac{\alpha_{\textup{pre}}}{\log \log \log \log N}
\end{align}
and
\begin{align}
\label{eGen3}
\left|\left|\left\{i \in S_{\textup{post}} : \pi_i(T_1 + T_2) = 0\right\}\right| - \frac{\alpha_{\textup{post}}}{2}\right| \leq \frac{\alpha_{\textup{post}}}{\log \log \log \log N}.
\end{align}
\end{itemize}
\end{mydef}

Let us explain the second condition. Note that given $a \in \text{Map}(M \sqcup M_\varnothing, \{\pm 1\})$, $i(X(a))$ is entirely contained in $S_r(N, \ell)$ or completely disjoint from $S_r(N, \ell)$. Since we only care about the intersection $i(X) \cap S_r(N, \ell)$, we restrict to only those $a$ with $i(X(a)) \subseteq S_r(N, \ell)$, and this is exactly what the second condition does. The importance of the fourth condition will be explained after our next two definitions.

Once we have fixed the first R\'edei matrix, we are ready to study all the higher R\'edei matrices. In fact, we will prove equidistribution of all the relevant higher R\'edei matrices. We formalize this as follows.

\begin{mydef}
Let $a \in \textup{Map}(M \sqcup M_\varnothing, \{\pm 1\})$ and $m \in \Z_{\geq 2}$ be given. Choose filtrations of vector spaces
\[
D_{a, 2} \supseteq \dots \supseteq D_{a, m}, \quad D_{a, 2}^\vee \supseteq \dots \supseteq D_{a, m}^\vee
\]
with $L \in D_{a, m}$ and $R \in D_{a, m}^\vee$. Define for $2 \leq i \leq m$
\[
n_{a, i} := \dim_{\FF_2} D_{a, i} - 1.
\]
If $\textup{Art}_{a, i} : D_{a, i} \times D_{a, i}^\vee \rightarrow \FF_2$ are bilinear pairings for $2 \leq i \leq m$, we call $(\textup{Art}_{a, i})_{2 \leq i \leq m}$ a sequence of Artin pairings if for every $2 \leq i < m$ the left kernel of $\textup{Art}_{a, i}$ is $D_{a, i + 1}$ and the right kernel of $\textup{Art}_{a, i}$ is $D_{a, i + 1}^\vee$. We say that a bilinear pairing
\[
\textup{Art}_{a, i} : D_{a, i} \times D_{a, i}^\vee \rightarrow \FF_2
\]
is valid if $L$ and $R$ are respectively in the left and right kernel. We call a sequence of Artin pairings valid if every element of the sequence is.

Let $X$ be an $N$-good box, let $a \in \textup{Map}(M \sqcup M_\varnothing, \{\pm 1\})$ be $(N, m, X)$-acceptable and also let $d \in i(X(a))$. We can naturally associate an infinite sequence of Artin pairings to $d$ as follows. Write the prime divisors of the discriminant of $\Q(\sqrt{d})$ as $p_1, \dots, p_{r'(a, X)}$ with $p_1 < \dots < p_{r'(a, X)}$. By construction, we have that for each $v \in D_{a, 2}$ that the unique ideal in $\Q(\sqrt{d})$ with norm
\[
\prod_{i = 1}^{r'(a, X)} p_i^{\pi_i(v)}
\]
is in $2 \textup{Cl}(\Q(\sqrt{d}))[4]$. Similarly we have for each $v \in D_{a, 2}^\vee$ that the character
\[
\sum_{i = 1}^{r'(a, X)} \pi_i(v) \chi_i
\]
is in $2 \textup{Cl}(\Q(\sqrt{d}))^\vee[4]$, where $\chi_i$ is as in Definition \ref{dMatrix}. In other words, we have natural epimorphisms
\[
D_{a, 2} \rightarrow 2 \textup{Cl}(\Q(\sqrt{d}))[4] \textup{ and } D_{a, 2}^\vee \rightarrow 2 \textup{Cl}(\Q(\sqrt{d}))^\vee[4].
\]
Now we declare $D_{a, i, d}$ and $D_{a, i, d}^\vee$ to be the inverse image of respectively $2^{i - 1} \textup{Cl}(\Q(\sqrt{d}))[2^i]$ and $2^{i - 1} \textup{Cl}(\Q(\sqrt{d}))^\vee[2^i]$ under these maps. Furthermore, we let $\textup{Art}_{a, i, d}$ be the natural pairing 
\[
2^{i - 1} \textup{Cl}(\Q(\sqrt{d}))[2^i] \times 2^{i - 1} \textup{Cl}(\Q(\sqrt{d}))^\vee[2^i] \rightarrow \FF_2
\]
pulled back to $D_{a, i, d}$ and $D_{a, i, d}^\vee$. 

This gives an infinite sequence of Artin pairings $\textup{Art}_{a, i, d}$ for every $d$. Furthermore, the sequence is valid if and only if equation (\ref{eGenPell}) is soluble. Finally, we define for a sequence of Artin pairings
\[
X(a, (\textup{Art}_{a, i})_{2 \leq i \leq m}) := \{d \in X(a) : \textup{Art}_{a, i, d} = \textup{Art}_{a, i} \textup{ for } 2 \leq i \leq m\}.
\]
\end{mydef}

For $d \in i(X(a))$, we have a natural isomorphism between $D_{a, 2}$ and $f^{-1}(2\textup{Cl}(\Q(\sqrt{d}))[4])$. Similarly, we have a natural isomorphism between the dual $D_{a, 2}^\vee$ and $2\textup{Cocy}(G_\Q, N(d)[4])$. Furthermore, the resulting Artin pairings are compatible. This is also true for $D_{a, m}$ and $f^{-1}(2^{m - 1}\textup{Cl}(\Q(\sqrt{d}))[2^m])$ provided that $\text{Art}_{a, i, d}$ and $\text{Art}_{a, i}$ are equal for $2 \leq i < m$, and the same holds on the dual side.

Take a non-trivial character $F: \text{Mat}(n_m + 1, n_m, \FF_2) \rightarrow \FF_2$. Our goal is to prove equidistribution of $F(\text{Art}_{a, m, d})$, where we view $\text{Art}_{a, m, d}$ as a matrix using fixed bases of $D_{a, m}$ and $D_{a, m}^\vee$. In order to prove equidistribution of $F$, we will start by finding a suitable set of \emph{variable indices} $S \subseteq [r]$. Then we shall fix one choice of prime in $X_i$ for the indices $i \in [r] - S$ and we will only vary over the primes in $X_i$ for the remaining $i \in S$. We make this precise in our next definition, where we shall also fix the bases used to identify $\text{Art}_{a, m, d}$ with a matrix.

\begin{mydef}
\label{dVar}
Let $a \in \widetilde{\textup{Map}}(M \sqcup M_\varnothing, \{\pm 1\}, \tilde{j}, \ell)$ and let $m \in \Z_{\geq 2}$ be an integer. We fix a basis $v_1, \ldots, v_{n_{a, 2}}, L$ for $D_{a, 2}$ and a basis $w_1, \ldots, w_{n_{a, 2}}, R$ for $D_{a, 2}^\vee$ for the rest of the paper in such a way that $v_1, \ldots, v_{n_{a, i}}, L$ is a basis for $D_{a, i}$ for $2 \leq i \leq m$, such that $w_1, \ldots, w_{n_{a, i}}, R$ is a basis for $D_{a, i}^\vee$ for $2 \leq i \leq m$ and such that $\pi_{\tilde{j}}(w_k) = 0$ for all $1 \leq k \leq n_{a, 2}$. We can decompose any $F : \textup{Mat}(n_{a, m} + 1, n_{a, m}, \FF_2) \rightarrow \{\pm 1\}$ as
\[
F = \prod_{\substack{1 \leq j_1 \leq n_{a, m} + 1 \\ 1 \leq j_2 \leq n_{a, m}}} E_{j_1, j_2}^{c_{j_1, j_2}(F)},
\]
where $c_{j_1, j_2}(F) \in \FF_2$ and $E_{j_1, j_2}$ is the map that sends a matrix to the coefficient in the entry $(j_1, j_2)$ viewed as an element of $\{\pm 1\}$ via $\iota^{-1}$. We say that $S \subseteq [r]$ is a set of variable indices for a non-zero character $F : \textup{Mat}(n_{a, m} + 1, n_{a, m}, \FF_2) \rightarrow \{\pm 1\}$ if there are integers $1 \leq j_1 \leq n_{a, m} + 1$ and $1 \leq j_2 \leq n_{a, m}$ and integers $i_1(j_1, j_2), i_2(j_1, j_2)$ with $c_{j_1, j_2}(F) \neq 0$, $i_2(j_1, j_2) \in S_{\textup{post}} \cap S$ and $S - \{i_2(j_1, j_2)\} \subseteq S_{\textup{pre}}$ such that
\begin{itemize}
\item in case $j_1 = n_{a, m} + 1$, we have $c_{j_1, j_2}(F) = 0$ for all $1 \leq j_1, j_2 \leq n_{a, m}$, $|S| = m$, $i_1(j_1, j_2) = \tilde{j}$ and
\[
S - \{i_1(j_1, j_2), i_2(j_1, j_2)\} \subseteq \bigcap_{i = 1}^{n_{a, 2}} \{j \in [r] : \pi_j(v_i) = 0\} \cap \bigcap_{i = 1}^{n_{a, 2}} \{j \in [r] : \pi_j(w_i) = 0\}
\]
and
\[
i_2(j_1, j_2) \in \{j \in [r] : \pi_j(w_{j_2}) = 1\} \cap \bigcap_{i = 1}^{n_{a, 2}} \{j \in [r] : \pi_j(v_i) = 0\} \cap \bigcap_{\substack{i = 1 \\ i \neq j_2}}^{n_{a, 2}} \{j \in [r] : \pi_j(w_i) = 0\}.
\]
\item in case $j_1 \neq n_{a, m} + 1$, we have $|S| = m + 1$, $i_1(j_1, j_2) \in S$ and
\[
S - \{i_1(j_1, j_2), i_2(j_1, j_2)\} \subseteq \bigcap_{i = 1}^{n_{a, 2}} \{j \in [r] : \pi_j(v_i) = 0\} \cap \bigcap_{i = 1}^{n_{a, 2}} \{j \in [r] : \pi_j(w_i) = 0\}
\]
and
\[
i_1(j_1, j_2) \in \{j \in [r] : \pi_j(w_{j_2}) = 1\} \cap \bigcap_{i = 1}^{n_{a, 2}} \{j \in [r] : \pi_j(v_i) = 0\} \cap \bigcap_{\substack{i = 1 \\ i \neq j_2}}^{n_{a, 2}} \{j \in [r] : \pi_j(w_i) = 0\}
\]
and
\[
i_2(j_1, j_2) \in \{j \in [r] : \pi_j(v_{j_1}) = 1\} \cap \bigcap_{\substack{i = 1 \\ i \neq j_1}}^{n_{a, 2}} \{j \in [r] : \pi_j(v_i) = 0\} \cap \bigcap_{i = 1}^{n_{a, 2}} \{j \in [r] : \pi_j(w_i) = 0\}.
\]
\end{itemize}
\end{mydef}

In simple words, the last row of $\text{Art}_{a, m, d}$ corresponds to the Artin pairing with $\ell$. This is exactly the case $j_1 = n_{a, m} + 1$ in the above definition. To prove equidistribution of these entries, we will use our higher R\'edei reciprocity law. It is for this reason that in this case our choice of variable indices is different than Smith's choice \cite[p.\ 32]{Smith}, while if $j_1 \leq n_{a, m}$ we make exactly the same choice as Smith.

We will now find our variable indices. It is here that the fourth condition in Definition \ref{dRedei} turns out to be crucial.

\begin{lemma}
\label{lVar}
Suppose that $a \in \widetilde{\textup{Map}}(M \sqcup M_\varnothing, \{\pm 1\}, \tilde{j}, \ell)$ satisfies equation (\ref{eGen2}). Assume that $v_1, \ldots, v_d, L \in D_{a, 2}$ and $v_{d + 1}, \ldots, v_e, R \in D_{a, 2}^\vee$ are linearly independent. Then we have for all $\mathbf{v} \in \FF_2^e$ the estimate
\[
\left|\left|\left\{i \in S_{\textup{pre}} : \pi_i(v_j) = \pi_j(\mathbf{v}) \textup{ for all } 1 \leq j \leq e\right\}\right| - \frac{\alpha_{\textup{pre}}}{2^e}\right| \leq 
\frac{100^e \cdot  k_{\textup{gap}}}{\log \log \log \log N}.
\]
\end{lemma}

\begin{proof}
This is a small adjustment of Lemma 13.7 in \cite{KP}. We stress that the term generic in \cite[Lemma 13.7]{KP} is an unfortunate clash of terminology, and refers to $a$ satisfying the natural analogue of our equation (\ref{eGen2}).
\end{proof}

We have a completely similar result for the range $k_{\text{gap}} < i \leq 2k_{\text{gap}}$ using equation (\ref{eGen3}). This brings us to our next reduction step.

\begin{prop}
\label{pArt}
Let $\ell$ be an integer such that $|\ell|$ is a prime $3$ modulo $4$. There are $c, A, N_0 > 0$ such that for all integers $N > N_0$, all integers $m \geq 2$, all sequences of integers $n_2 \geq \dots \geq n_m \geq 0$, all $N$-good boxes $X$, all $(N, m, X)$-acceptable $a \in \textup{Map}(M \sqcup M_\varnothing, \{\pm 1\})$, all sequences of valid Artin pairings $(\textup{Art}_{a, i})_{2 \leq i \leq m - 1}$ with $n_{a, i} = n_i$ for $2 \leq i \leq m$ and an Artin pairing $\textup{Art}_{a, m} : D_{a, m} \times D_{a, m}^\vee \rightarrow \FF_2$ with $R$ in the right kernel
\begin{multline*}
\left||X(a, (\textup{Art}_{a, i})_{2 \leq i \leq m})|  - 2^{-n_m(n_m + 1)} \cdot |X(a, (\textup{Art}_{a, i})_{2 \leq i \leq m - 1})|\right| \leq 
\frac{A \cdot |X(a)|}{(\log \log \log \log N)^{\frac{c}{m 6^m}}}.
\end{multline*}
\end{prop}

\begin{remark}
We do not need to assume that $(\textup{Art}_{a, i})_{2 \leq i \leq m - 1}$ is valid, but it suffices for our purposes and avoids some casework later on.
\end{remark}

\begin{proof}[Proof that Proposition \ref{pArt} implies Proposition \ref{pBox}]
We observe that $i(X(a)) \cap S_r(N, \ell) \neq \varnothing$ implies $a \in \widetilde{\text{Map}}(M \sqcup M_\varnothing, \{\pm 1\}, \tilde{j}, \ell)$. Hence we can bound
\begin{align*}
&\left|\left|i(X) \cap \bigcap_{i = 2}^{m + 1} D_{\ell, i}(n_i)\right| - \frac{P(n_m, n_m, n_{m + 1})}{2^{n_m}} \cdot \left|i(X) \cap \bigcap_{i = 2}^m D_{\ell, i}(n_i)\right|\right| \leq \\
&\sum_{a \in \widetilde{\text{Map}}(M \sqcup M_\varnothing, \{\pm 1\}, \tilde{j}, \ell)} \left|\left|i(X(a)) \cap \bigcap_{i = 2}^{m + 1} D_{\ell, i}(n_i)\right| - \frac{P(n_m, n_m, n_{m + 1})}{2^{n_m}} \cdot \left|i(X(a)) \cap \bigcap_{i = 2}^m D_{\ell, i}(n_i)\right|\right|.
\end{align*}
We split this sum over the $(N, m, X)$-acceptable $a \in \text{Map}(M \sqcup M_\varnothing, \{\pm 1\})$ and the remaining $a$. For the $(N, m, X)$-acceptable $a$ we may apply Proposition \ref{pArt} by further splitting the sum over all possible sequences of valid Artin pairings and an Artin pairing $D_{a, m} \times D_{a, m}^\vee \rightarrow \FF_2$ with $R$ in the right kernel. 

Note that in the set of bilinear pairings $D_{a, m} \times D_{a, m}^\vee \rightarrow \FF_2$ with $R$ in the right kernel, there are precisely 
\[
2^{n_m(n_m + 1)} \cdot \frac{P(n_m, n_m, n_{m + 1})}{2^{n_m}}
\]
such that the left kernel has dimension $n_{m + 1} + 1$ and $L$ is in the left kernel. There are at most $2^{mn_{\text{max}}^2}$ sequences of Artin pairings, so we stay within the error term of Proposition \ref{pBox} provided that we take the constant $c'$ in the definition of $n_{\text{max}}$ smaller than the constant $c$ guaranteed by Proposition \ref{pArt}.

Hence it suffices to bound
\[
\sum_{\substack{a \in \widetilde{\text{Map}}(M \sqcup M_\varnothing, \{\pm 1\}, \tilde{j}, \ell) \\ a \text{ not } (N, m, X)\text{-acceptable}}} |i(X(a))|.
\]
We first tackle those $a$ for which $n_{a, 2} > n_{\text{max}}$. These $a$ can easily be dealt with using equation (\ref{e4rankX}) for $k \leq n_{\text{max}}$ inducing an error of size
\[
O\left(\frac{|i(X) \cap S_r(N, \ell)|}{(\log \log \log \log N)^{\frac{c}{m^2 6^m}}}\right)
\]
for some absolute constant $c > 0$.

We will now dispatch those $a$ that fail equation (\ref{eXjlarge}). We declare two maps $a, a' \in \widetilde{\text{Map}}(M \sqcup M_\varnothing, \{\pm 1\}, \tilde{j}, \ell)$ to be equivalent at some integer $i > k$, written as $a \sim_i a'$, if
\[
a(j, i) = a'(j, i) \text{ for all } 1 \leq j \leq k \text{ and } a(i, -1) = a'(i, -1).
\]
Observe that if $a$ fails equation (\ref{eXjlarge}), then so does any $a'$ with $a \sim_i a'$. We call an equivalence class bad if there exists some $a$ in its equivalence classes failing equation (\ref{eXjlarge}). In a given bad equivalence class we clearly have the bound
\[
\left|\bigcup_{a' : a \sim_i a'} X(a')\right| \leq \frac{|X|}{(\log t_{k + 1})^{100}}.
\]
A simple computation shows that we stay within the error term of Proposition \ref{pBox} when we sum over all $i$ and all bad equivalence classes.

We still have to deal with those $a$ failing equation equation (\ref{eGen2}) or equation (\ref{eGen3}). Call $a$ generic if $D_{a, 2} \cap D_{a, 2}^\vee = \{0\}$, where we view $D_{a, 2}$ and $D_{a, 2}^\vee$ as subspaces of $V$. Let us now suppose that $r = r'(a, X)$, the other case can be dealt with similarly. Take a non-zero vector $v \in \FF_2^r$ with $\lambda$ ones with $v \neq L$ and $v \neq R$. We claim that
\[
\frac{|\{a \in \widetilde{\text{Map}}(M \sqcup M_\varnothing, \{\pm 1\}, \tilde{j}, \ell) : v \in D_{a, 2} \cap D_{a, 2}^\vee, r = r'(a, X)\}|}{|\{a \in \widetilde{\text{Map}}(M \sqcup M_\varnothing, \{\pm 1\}, \tilde{j}, \ell) : r = r'(a, X)\}|} = O(2^{-r - \lambda}).
\]
We have that the proportion of $a$ with $v \in D_{a, 2}$ is equal to $O(2^{-r})$. Furthermore, the condition that also $v \in D_{a, 2}^\vee$ implies that for every $i$ with $\pi_i(v) = 1$ we have $a(i, -1) = 1$. These are $O(2^{-\lambda})$ independent extra conditions giving a total of $O(2^{-r-\lambda})$. This establishes the claim. For the case that $v = L$ or $v = R$, we make fundamental use of the fact that $|\ell|$ is equivalent to $3$ modulo $4$ to show that the above proportion is still $O(2^{-r})$.

Summing over all non-zero vectors $v \in V$ then gives that the proportion of $a \in \widetilde{\text{Map}}(M \sqcup M_\varnothing, \{\pm 1\}, \tilde{j}, \ell)$, which are not generic, is bounded by
\[
O\left(\sum_{\lambda = 1}^r 2^{-r - \lambda} \binom{r}{\lambda}\right) = O\left(0.75^r\right).
\]
Take some $v, w \in V$. Recall that the proportion of $a$ with $v \in D_{a, 2}$ is bounded by $O(2^{-r})$ provided that $v \not \in \langle L \rangle$. Similarly, the proportion of $a$ with $w \in D_{a, 2}^\vee$ is bounded by $O(2^{-r})$ if $w \not \in \langle R \rangle$. Finally, if $a$ is generic, the proportion of $a$ with $(v, w) \in D_{a, 2} \times D_{a, 2}^\vee$ is bounded by $O(4^{-r})$ as long as $v \not \in \langle L \rangle$ and $w \not \in \langle R \rangle$. 

But Hoeffding's inequality yields that the proportion of $(v, w) \in V \times V$ satisfying
\[
\left|\left|\left\{i \in S_{\textup{pre}} : \pi_i(v + w) = 0\right\}\right| - \frac{\alpha_{\textup{pre}}}{2}\right| > \frac{\alpha_{\textup{pre}}}{\log \log \log \log N}
\]
is at most
\[
O\left(\exp\left(-(\log \log \log \log N)^{-2} \cdot \alpha_{\textup{pre}}\right)\right).
\]
From the last two observations we quickly deduce that the proportion of generic $a$ for which equation (\ref{eGen2}) fails is also bounded by
\[
O\left(\exp\left(-(\log \log \log \log N)^{-2} \cdot \alpha_{\textup{pre}}\right)\right),
\]
and a similar argument applies for the proportion of $a$ failing equation (\ref{eGen3}).

We have now found an upper bound for the proportion of $a$ failing equation (\ref{eGen2}) or equation (\ref{eGen3}). We therefore obtain an upper bound for the proportion of $a$ that are not $(N, m, X)$-acceptable. To finish the proof, we merely need to bound the union of $X(a)$ over these $a$. This follows from Proposition \ref{p6.4}.
\end{proof}

We remark that we can always find variable indices as in Definition \ref{dVar} if $a$ is $(N, m, X)$-acceptable and $N$ is sufficiently large. This is a simple computation once we use that
\begin{align}
\label{eBoundm}
m < \log \log \log \log \log \log N,
\end{align}
since otherwise Theorem \ref{tHeuristic} is trivial. We now have all the required setup for our next proposition, where we fix one prime for all indices smaller than $k_{\textup{gap}}$ except the variable indices.

\begin{prop}
\label{pF}
Let $\ell$ be an integer such that $|\ell|$ is a prime $3$ modulo $4$. There are $c, A, N_0 > 0$ such that for all integers $N > N_0$, all integers $m \geq 2$, all $N$-good boxes $X$, all $(N, m, X)$-acceptable $a \in \textup{Map}(M \sqcup M_\varnothing, \{\pm 1\})$, all sequences of valid Artin pairings $(\textup{Art}_{a, i})_{2 \leq i \leq m - 1}$, all non-zero multiplicative characters $F : \textup{Mat}(n_{a, m} + 1, n_{a, m}, \FF_2) \rightarrow \{\pm 1\}$, all sets of variable indices $S$ for $F$ and all $Q \in \prod_{i \in [k_{\textup{gap}}] - S} X_i$ such that
\begin{align}
\label{eXjQlarge}
|X_j(a, Q)| \geq 4^{-k_{\textup{gap}}} \cdot |X_j|
\end{align}
for all $j \in S$, we have
\[
\left|\sum_{x \in X(a, Q, (\textup{Art}_{a, i})_{2 \leq i \leq m - 1})} F\left(\textup{Art}_{a, m, i(x)}\right)\right| \leq 
\frac{A \cdot |X(a, Q)|}{(\log \log \log \log N)^{\frac{c}{m 6^m}}}.
\]
Here $X(a, Q, (\textup{Art}_{a, i})_{2 \leq i \leq m - 1})$ is defined as the subset of $x \in X(a, (\textup{Art}_{a, i})_{2 \leq i \leq m - 1})$ with $\pi_i(x)$ equal to $\pi_i(Q)$ for $i \in [k_{\textup{gap}}] - S$.
\end{prop}

\begin{proof}[Proof that Proposition \ref{pF} implies Proposition \ref{pArt}]
The proof is almost identical to the one given in \cite[Proof that Proposition 6.9 implies Proposition 6.6]{CKMP}. We have to show that equation (\ref{eXjQlarge}) is typically satisfied. We apply Proposition \ref{p6.3} to
\[
(X_{k + 1}(a, Q') \times \dots \times X_r(a, Q'), Q'),
\]
where $Q'$ is the unique element of $X_1 \times \dots \times X_k$. Crucially, all the required conditions for Proposition \ref{p6.3} are satisfied due to equation (\ref{eXjlarge}), completing our reduction step.
\end{proof}

It is time for our final reduction step. If $c_{j_1, j_2}(F) \neq 0$ for some $1 \leq j_1, j_2 \leq n_{a, m}$, Smith's method applies without any significant changes. If however $c_{j_1, j_2}(F) = 0$ for all $1 \leq j_1, j_2 \leq n_{a, m}$, Smith's method breaks down. It is here that we make essential use of our generalized R\'edei reciprocity law. 

We shall now add the algebraic structure needed to apply our reflection principles. The required equidistribution will then be a consequence of the Chebotarev Density Theorem and Proposition \ref{pdF}. From now on we shall make heavy use of the notation introduced in Subsections \ref{ssNot} and \ref{ssGov}.

\begin{mydef}
\label{dWGov}
Take a $N$-good box $X$, a $(N, m, X)$-acceptable $a \in \textup{Map}(M \sqcup M_\varnothing, \{\pm 1\})$ and a non-zero multiplicative character $F : \textup{Mat}(n_{a, m} + 1, n_{a, m}, \FF_2) \rightarrow \{\pm 1\}$. Let $S$ be a set of variable indices for $F$. Fix a choice of $j_1$ and $j_2$ with $c_{j_1, j_2}(F) \neq 0$ as in Definition \ref{dVar}. Put $S' := S \cap [k_{\textup{gap}}]$. For each $i \in S'$, let $Z_i$ be subsets of $X_i$ with cardinality
\[
M_{\textup{box}} := \left \lfloor (\log \log \log \log N)^{\frac{1}{5(m + 1)}} \right \rfloor.
\]
Note that $M_{\textup{box}} \geq 2$ for $N$ greater than an absolute constant by equation (\ref{eBoundm}). Put
\[
Z := \prod_{i \in S'} Z_i, \quad Z' := \prod_{i \in S' - \{i_1(j_1, j_2)\}} Z_i.
\]
If $j_1 \leq n_{a, m}$, we say that $Z$ is well-governed for $(a, F)$ if for every distinct $a_1, a_2 \in Z_{i_1(j_1, j_2)}$ there is a governing expansion $\mathfrak{G}_{j, a_1a_2}$ on $(Z', S' - \{i_1(j_1, j_2)\}, a_1a_2)$, in which all primes in $Q$, $(2)$ and $\infty$ split completely. Put
\[
M_\circ(Z) := \prod_{i \in S'} \prod_{z_1, z_2 \in Z_i} \Q(\sqrt{z_1z_2}) \prod_{\substack{a_1, a_2 \in Z_{i_1(j_1, j_2)} \\ a_1 \neq a_2}} \prod_{T \subsetneq S' - \{i_1(j_1, j_2)\}} \prod_{\bar{x} \in \overline{Y}_T} L(\phi_{\bar{x}, a_1a_2})
\]
and
\[
M(Z) :=  \prod_{\substack{a_1, a_2 \in Z_{i_1(j_1, j_2)} \\ a_1 \neq a_2}} \prod_{\bar{x} \in \overline{X}_{S' - \{i_1(j_1, j_2)\}}} L(\phi_{\bar{x}, a_1a_2}).
\]
If $j_1 = n_{a, m} + 1$, we say that $Z$ is well-governed for $(a, F)$ if there is a governing expansion $\mathfrak{G}_\ell$ on $(Z, S', \ell)$, in which all primes in $Q$ coprime to $2 \ell$ split completely and all primes of $\Q(\sqrt{\ell})$ above $(2)$ and $\infty$ split completely, and furthermore for each $j \in S'$ and each distinct $a_1, a_2 \in X_j$ there is a governing expansion $\mathfrak{G}_{j, a_1a_2}$ on $(\prod_{i \in S' - \{j\}} X_i, S' - \{j\}, a_1a_2)$, in which all primes in $Q$, $(2)$ and $\infty$ split completely. Put
\[
M_\circ(Z) := \prod_{T \subsetneq S'} \prod_{\bar{x} \in \overline{X}_T} L(\phi_{\bar{x}, \ell}) \prod_{j \in S'} \prod_{\substack{a_1, a_2 \in X_j \\ a_1 \neq a_2}} \prod_{\bar{x} \in \overline{X}_{S' - \{j\}}} L(\phi_{\bar{x}, a_1a_2})
\]
and
\[
M(Z) := M_\circ(Z) \prod_{\bar{x} \in \overline{X}_{S'}} L(\phi_{\bar{x}, \ell}).
\]
so $M(Z)$ is a central Galois extension of $M_\circ(Z)$ in both cases.

Take some $Q \in \prod_{i \in [k_{\textup{gap}}] - S'} X_i$. Then we define, for $i > k_{\textup{gap}}$, $X_i(a, Q, M_\circ(Z))$ to be the subset of primes $p \in X_i$ such that $p$ splits completely in $M_\circ(Z)$, $p \in X_i(a, Q)$ and
\[
\left(\frac{z}{p}\right) = a(j, i) \textup{ for all } j \in S' \textup{ and all } z \in Z_j.
\]
Note that these conditions are equivalent to $\textup{Frob}_p$ being equal to a given central element in the Galois group of the compositum of $M_\circ(Z)$ and $\Q(\sqrt{x})$ with $x$ running through $-1$, the prime divisors of $Q$ and the primes in $Z_j$ for $j \in S'$.

We let
\[
\widetilde{Z} := Q \times Z \times \prod_{i > k_{\textup{gap}}} X_i(a, Q, M_\circ(Z)).
\]
We call $\widetilde{Z}$ a satisfactory product space for $(X, a, F, Q)$ if 
\begin{itemize}
\item $Z$ is well-governed for $(a, F)$;
\item we have for all $i < j$ with $i, j \in S'$, all $z_i \in Z_i$ and all $z_j \in Z_j$
\[
\left(\frac{z_i}{z_j}\right) = a(i, j);
\]
\item we have $Z_i \subseteq X_i(a, Q)$.
\end{itemize}
\end{mydef}

Once we added the necessary algebraic structure to our box, we can construct a suitable additive system $\mathfrak{A}$ to which we apply Proposition \ref{pdF}. This is the goal of the next lemma, which provides the critical link between our algebraic results and Proposition \ref{pdF}.

\begin{lemma}
\label{lAS}
Let a $(N, m, X)$-acceptable $a \in \textup{Map}(M \sqcup M_\varnothing, \{\pm 1\})$, a sequence of valid Artin pairings $(\textup{Art}_{a, i})_{2 \leq i \leq m - 1}$, a non-zero multiplicative character $F : \textup{Mat}(n_{a, m} + 1, n_{a, m}, \FF_2) \rightarrow \{\pm 1\}$ and a set of variable indices $S$ for $F$ be given. Take $\widetilde{Z}$ to be a satisfactory product space for $(X, a, F, Q)$. Then there is a $(2^{n_{\textup{max}}(n_{\textup{max}} + m + 2)}, S)$-acceptable additive system $\mathfrak{A}$ with $\overline{Y}_\varnothing^\circ(\mathfrak{A}) = \widetilde{Z} \cap X(a, (\textup{Art}_{a, i})_{2 \leq i \leq m - 1})$ such that
\begin{equation}
\label{eRef}
\sum_{x \in \bar{x}(\varnothing)} \iota(F(\textup{Art}_{a, m, x})) = \phi_{\pi_{S' - \{i_1(j_1, j_2)\}}(\bar{x}), c}(\textup{Frob}(p_1) \cdot \textup{Frob}(p_2))
\end{equation}
for all $\bar{x} \in C(\mathfrak{A})$, where $(p_1, p_2) := \pi_{i_2(j_1, j_2)}(\bar{x})$. Here $c$ equals $\textup{pr}_1(\pi_{i_1(j_1, j_2)}(\bar{x})) \textup{pr}_2(\pi_{i_1(j_1, j_2)}(\bar{x}))$ if $j_1 \leq n_{a, m}$ and equals $\ell$ otherwise.
\end{lemma}

\begin{proof}
We shall proceed to explicitly construct $\mathfrak{A}$ by induction. We start by introducing some notation. Let $w \in D_{a, 2}^\vee$ be one of the chosen basis vectors and let $x \in X(a)$ be given. A raw cocycle for $(x, w)$ is a sequence $(\psi_{x, w, i})_{0 \leq i \leq k}$ of maximal length with $\psi_{x, w, i} \in \text{Cocy}(G_\Q, N(x)[2^i])$, $2\psi_{x, w, i + 1} = \psi_{x, w, i}$, $L(\psi_{x, w, i}) \Q(\sqrt{x})/\Q(\sqrt{x})$ unramified and
\[
\psi_{x, w, 1} = \sum_{i = 1}^{r'(a, X)} \pi_i(w) \chi_i
\]
with $\chi_i$ as in Definition \ref{dMatrix}. We now make a choice of raw cocycle for every $(x, w)$ with $x \in X(a)$. Recall that $i_1(j_1, j_2)$, $i_2(j_1, j_2)$, $j_1$ and $j_2$ are the integers associated to our set of variable indices $S$ as in Definition \ref{dVar}. Set 
$$
\overline{Y}_\varnothing^\circ(\mathfrak{A}) := \widetilde{Z} \cap X(a, (\textup{Art}_{a, i})_{2 \leq i \leq m - 1}).
$$
First suppose that $j_1 \leq n_{a, m}$. To shorten our formulas, we define for $\bar{x} \in \overline{X}_S$ and $i \in S$
\[
\text{prp}(\bar{x}, i) = \textup{pr}_1(\pi_i(\bar{x})) \cdot \textup{pr}_2(\pi_i(\bar{x})).
\]
Let $T \subsetneq S$. We shall construct our maps $F_{T'}$ with $T' \subseteq T$ in such a way that $\overline{Y}_T(\mathfrak{A})$ is precisely the set of cubes $\bar{x}$ satisfying $\bar{x}(\varnothing) \subseteq \overline{Y}_\varnothing^\circ(\mathfrak{A})$ and the following properties

\begin{itemize}
\item we have for all $T' \subsetneq T$ with $i_2(j_1, j_2) \not \in T'$, all $\bar{y} \in \bar{x}(T')$ and all $j \neq j_2$
\[
\sum_{y \in \bar{y}(\varnothing)} \psi_{y, w_j, |T'|} = 0;
\]
\item we have for all $T' \subsetneq T$ with $i_2(j_1, j_2) \not \in T'$ and all $\bar{y} \in \bar{x}(T')$
\[
\sum_{y \in \bar{y}(\varnothing)} \psi_{y, w_{j_2}, |T'|} = \left\{
\begin{array}{ll}
\phi_{\pi_{T' - \{i_1(j_1, j_2)\}}(\bar{y}), \text{prp}(\bar{x}, i_1(j_1, j_2))} & \mbox{if } i_1(j_1, j_2) \in T'\\
0 & \mbox{if } i_1(j_1, j_2) \not \in T';
\end{array}
\right.
\]
\item we have for all $T' \subsetneq T$ with $i_2(j_1, j_2) \not \in T'$, $\bar{y} \in \bar{x}(T')$, all $j$ and $i \in S - T'$
\begin{equation}
\label{eRam}
\sum_{y \in \bar{y}(\varnothing)} \psi_{y, w_j, |T'| + 1}(\sigma_{\pi_i(\bar{x})}) = 0.
\end{equation}
\end{itemize}

Now suppose that $j_1 = n_{a, m} + 1$. Let $T \subseteq S$. In this case we construct our maps $F_{T'}(\mathfrak{A})$ such that $\overline{Y}_T(\mathfrak{A})$ equals the cubes $\bar{x}$ with $\bar{x}(\varnothing) \subseteq \overline{Y}_\varnothing^\circ(\mathfrak{A})$ and

\begin{itemize}
\item we have for all $T' \subsetneq T$ with $i_2(j_1, j_2) \not \in T'$, all $\bar{y} \in \bar{x}(T')$ and all $j$
\[
\sum_{y \in \bar{y}(\varnothing)} \psi_{y, w_j, |T'|} = 0;
\]
\item we have for all $\varnothing \subsetneq T' \subsetneq T$ with $i_2(j_1, j_2) \not \in T'$, all $\bar{y} \in \bar{x}(T')$, all $j$ and all $i \in S - T'$
\[
\sum_{y \in \bar{y}(\varnothing)} \psi_{y, w_j, |T'| + 1}(\sigma_{\pi_i(\bar{x})}) = 0;
\]
\item we have for all $T' \subsetneq T$ with $i_2(j_1, j_2) \not \in T'$, all $\bar{y} \in \bar{x}(T')$ and all $j$
\[
\sum_{y \in \bar{y}(\varnothing)} \psi_{y, w_j, |T'| + 1}(\sigma_\ell) = 0.
\]
\end{itemize}

Let us prove by induction that $\overline{Y}_T(\mathfrak{A})$ is as claimed. We shall construct the map $F_T(\mathfrak{A})$ during the induction. Until otherwise stated, we shall treat the case $j_1 \leq n_{a, m}$. At the end we indicate the modifications necessary to deal with the case $j_1 = n_{a, m} + 1$. Take $\bar{x} \in \overline{Y}_T(\mathfrak{A})$. If $i_2(j_1, j_2) \in T$ or $T = S - \{i_2(j_1, j_2)\}$, we simply let $F_T$ be the zero map. Henceforth we will assume that $i_2(j_1, j_2) \not \in T$ and $|T| < |S| - 1$. Then we define
\[
\psi_j := 
\left\{
	\begin{array}{ll}
		\sum\limits_{x \in \overline{x}(\varnothing)} \psi_{x, w_j, |T|} & \mbox{if } j \neq j_2 \text{ or } i_1(j_1, j_2) \not \in T \\
		\phi_{\pi_{T - \{i_1(j_1, j_2)\}}(\bar{x}), \text{prp}(\bar{x}, i_1(j_1, j_2))} + \sum\limits_{x \in \overline{x}(\varnothing)} \psi_{x, w_{j_2}, |T|}  & \mbox{otherwise.}
	\end{array}
\right.
\]
Then \cite[Proposition 2.9]{KPPell} demonstrates that $\psi_j$ is a quadratic character. We claim that $\psi_j$ is an unramified character of $\Q(\sqrt{x})$ for all $x \in \bar{x}(\varnothing)$.

If $p = \pi_i(\bar{x})$ with $i \not \in T$, this is clear. So suppose that $i \in T$ and write $\pi_i(\bar{x}) = \{p_1, p_2\}$ with $p_1 = \pi_i(x)$. It is clear that $\psi_j$ does not ramify at $p_1$, so it suffices to show that $\psi_j$ does not ramify at $p_2$. Let $\bar{y}_k \in \bar{x}(T - \{i\})$ be the cube with $\pi_i(\bar{y}_k) = p_k$. Then we have
\[
\psi_j(\sigma_{p_2}) = \sum_{x \in \overline{x}(\varnothing)} \psi_{x, w_j, |T|}(\sigma_{p_2}) = \sum_{y \in \overline{y}_1(\varnothing)} \psi_{y, w_j, |T|}(\sigma_{p_2})  + \sum_{y \in \overline{y}_2(\varnothing)} \psi_{y, w_j, |T|}(\sigma_{p_2}) = 0 + 0 = 0.
\]
The first sum is clearly zero, since all the $\psi_{y, w_j, |T|}$ with $y \in \bar{y}_1(\varnothing)$ are unramified at $p_2$. Furthermore, the second sum is zero by equation (\ref{eRam}) with $T' := T - \{i\}$. This proves our claim.

Next we claim that $\pi_i(\bar{x})$ splits completely in $L(\psi_j)$ for all $i \not \in T$. Indeed, $\pi_i(\bar{x})$ has residue field degree $1$ in every $\psi_{x, w_j, |T|}$ for $x \in \bar{x}(\varnothing)$, because $2\psi_{x, w_j, |T| + 1} = \psi_{x, w_j, |T|}$. Furthermore, $\pi_i(\bar{x})$ splits completely in $L(\phi_{\pi_{T - \{i_1(j_1, j_2)\}}(\bar{x}), \text{prp}(\bar{x}, i_1(j_1, j_2))})$, establishing the claim.

Pick some $x \in \bar{x}(\varnothing)$ and let $p \in \pi_i(x)$ for some $i \in T$. It is straightforward to deduce from $\bar{x}(\varnothing) \subseteq X(a)$ that
\[
\psi_j|_{G_{\Q(\sqrt{x})}}(\text{Frob}(\mathfrak{p}))
\]
does not depend on $x$, where $\mathfrak{p}$ is the unique ideal above $p$ in $\Q(\sqrt{x})$. From this, it becomes clear, from the additivity of $\psi_j$, that this defines an additive map $F_{T, j, 1}$ to $\mathbb{F}_2^{|T|}$.

It follows from Lemma \ref{lVar} that there exists a set $A \subseteq [r]$ and a bijection $f: [n_{a, 2} + 1] \rightarrow A$ such that $A \cap S = \varnothing$ and
\[
\pi_{f(i)}(w_k) = \delta_{i, k}
\]
for $1 \leq i, k \leq n_{a, 2}$ and furthermore
\[
\pi_{f(n_{a, 2} + 1)}(w_k) = 0
\]
for all $1 \leq k \leq n_{a, 2}$. Then we define an additive map $F_{T, j, 2}$ to $\mathbb{F}_2^{n_{a, 2} + 1}$ by
\[
(\psi_j(\sigma_{\pi_i(x)}))_{i \in A}.
\]
Finally, we define an additive map $F_{T, j, 3}$ to $\mathbb{F}_2^{|S| - |T|}$ by sending $\bar{x}$ to
\[
\left(\sum_{x \in \bar{x}(\varnothing)} \psi_{x, w_j, |T| + 1}(\sigma_{\pi_i(\bar{x})})\right)_{i \in S - T}.
\]
We define our map $F_T(\mathfrak{A})$ to be $(F_{T, j, 1}, F_{T, j, 2}, F_{T, j, 3})_{1 \leq j \leq n_{a, |T| + 1}}$. Note that the maps $F_{T, j, 1}$ and $F_{T, j, 2}$ encode precisely when $\psi_j = 0$. From this it becomes clear that $\overline{Y}_T(\mathfrak{A})$ has the claimed shape.

Our next task is to verify that our additive system is $(2^{n_{\text{max}}(n_{\text{max}} + m + 2)}, S)$-acceptable. For the first requirement, this follows from the construction of $F_T$ above and the inequality $n_{a, 2} \leq n_{\text{max}}$. We still need to deal with the second requirement. Take $\bar{x} \in C(\mathfrak{A})$. If there is some $i \in S$ such that
\[
|\bar{x}(S - \{i\}) \cap \overline{Y}_{S - \{i\}}^\circ(\mathfrak{A})| = 2,
\]
then we are done. Henceforth we assume that
\[
|\bar{x}(S - \{i\}) \cap \overline{Y}_{S - \{i\}}^\circ(\mathfrak{A})| = 1
\]
for all $i \in S$ and let $x_0$ be the unique element in $\bar{x}(\varnothing)$ outside $\bar{x}(S - \{i\}) \cap \overline{Y}_{S - \{i\}}^\circ(\mathfrak{A})$ for all $i \in S$. Then we need to prove that $x_0 \in \overline{Y}_\varnothing^\circ(\mathfrak{A})$. Clearly, $x_0 \in \widetilde{Z} \cap X(a)$. Take an integer $2 \leq m' \leq m - 1$, integers $1 \leq j_1' \leq n_{a, m'} + 1$ and $1 \leq j_2' \leq n_{a, m'}$. It suffices to prove that
\[
\iota(E_{j_1', j_2'}(\text{Art}_{a, m', x_0})) = \iota(E_{j_1', j_2'}(\text{Art}_{a, m'})).
\]
Choose a subset $T$ of $S$ of size $m'$ not containing $i_1(j_1, j_2)$ and $i_2(j_1, j_2)$. Then the above identity follows from Theorem \ref{tRefMin} applied to any cube in $\bar{x}(T)$ containing $x_0$.

We still need to prove equation (\ref{eRef}). Recall that $j_1 \leq n_{a, m}$. Take some indices $(j_3, j_4)$ with $(j_3, j_4) \neq (j_1, j_2)$. We claim that
\[
\sum_{x \in \bar{x}(\varnothing)} \iota(E_{j_3, j_4}(\textup{Art}_{a, m, x})) = 0.
\]
First suppose that $j_3 \leq n_{a, m}$. Then this follows from two applications of Theorem \ref{tRefMin}. In case $j_3 = n_{a, m} + 1$ we apply Theorem \ref{tRefAgr} twice to obtain
\[
\sum_{x \in \bar{x}(\varnothing)} \iota(E_{j_3, j_4}(\textup{Art}_{a, m, x})) = 0.
\]
Here we use equation (\ref{eAgrN}), if $\ell > 0$, and equation (\ref{eAgrI}), if $\ell < 0$. We deduce from another double application of Theorem \ref{tRefAgr} that
\[
\sum_{x \in \bar{x}(\varnothing)} \iota(E_{j_1, j_2}(\textup{Art}_{a, m, x})) = \phi_{\pi_{S' - \{i_1(j_1, j_2)\}}(\bar{x}), \text{prp}(\bar{x}, i_1(j_1, j_2))}(\text{Frob}(p_1) \cdot \text{Frob}(p_2)).
\]
Adding these identities together yields equation (\ref{eRef}). This proves the lemma for $j_1 \leq n_{a, m}$.

It remains to indicate the necessary changes in case $j_1 = n_{a, m} + 1$. In this case we let $F_T$ be the zero map if $i_2(j_1, j_2) \in T$. Otherwise we define
\[
\psi_j := \sum_{x \in \overline{x}(\varnothing)} \psi_{x, w_j, |T|}.
\]
Now we proceed by defining the maps $F_{T, j, i}$ just as in the case $j_1 \leq n_{a, m}$. Then we see that $\mathfrak{A}$ is certainly $(2^{n_{\text{max}}(n_{\text{max}} + m + 2)}, S)$-acceptable. Now we have for all $(j_3, j_4)$ with $j_3 \leq n_{a, m}$
\[
\sum_{x \in \bar{x}(\varnothing)} \iota(E_{j_3, j_4}^{c_{j_3, j_4}(F)}(\textup{Art}_{a, m, x})) = 0
\]
simply because $c_{j_3, j_4}(F) = 0$ by our choice of variable indices. Furthermore, Theorem \ref{tRefMin} shows that for all $(j_3, j_4)$ with $j_3 = n_{a, m + 1}$ and $j_2 \neq j_4$
\[
\sum_{x \in \bar{x}(\varnothing)} \iota(E_{j_3, j_4}(\textup{Art}_{a, m, x})) = 0.
\]
Finally, Theorem \ref{tReflectionEll} implies that
\[
\sum_{x \in \bar{x}(\varnothing)} \iota(E_{j_1, j_2}(\textup{Art}_{a, m, x})) = \phi_{\pi_{S'}(\bar{x}), \ell}(\textup{Frob}(p_1) \cdot \textup{Frob}(p_2))
\]
with $(p_1, p_2) := \pi_{i_2(j_1, j_2)}(\bar{x})$. Hence we conclude that
\[
\sum_{x \in \bar{x}(\varnothing)} \iota(F(\textup{Art}_{a, m, x})) = \phi_{\pi_{S'}(\bar{x}), \ell}(\textup{Frob}(p_1) \cdot \textup{Frob}(p_2)),
\]
which completes the proof of our lemma.
\end{proof}

\begin{prop}
\label{pFinal}
Let $\ell$ be an integer such that $|\ell|$ is a prime $3$ modulo $4$. There are $c, A, N_0 > 0$ such that for all integers $N > N_0$, all integers $m \geq 2$, all $N$-good boxes $X$, all $(N, m, X)$-acceptable $a \in \textup{Map}(M \sqcup M_\varnothing, \{\pm 1\})$, all sequences of valid Artin pairings $(\textup{Art}_{a, i})_{2 \leq i \leq m - 1}$, all non-zero multiplicative characters $F : \textup{Mat}(n_{a, m} + 1, n_{a, m}, \FF_2) \rightarrow \{\pm 1\}$, all sets of variable indices $S$ for $F$, all $Q \in \prod_{i \in [k_{\textup{gap}}] - S} X_i$ and all satisfactory product spaces $\widetilde{Z}$ for $(X, a, F, Q)$
\[
\left|\sum_{x \in \widetilde{Z} \cap X(a, Q, (\textup{Art}_{a, i})_{2 \leq i \leq m - 1})} F\left(\textup{Art}_{a, m, i(x)}\right)\right| \leq 
\frac{A \cdot |\widetilde{Z} \cap X(a, Q)|}{(\log \log \log \log N)^{\frac{c}{m 6^m}}}.
\]
\end{prop}

\begin{proof}[Proof that Proposition \ref{pFinal} implies Proposition \ref{pF}] 
The proof is very similar to the proof of Proposition 7.5 implies Proposition 7.4 in Smith \cite{Smith}. We only indicate the necessary changes here. There is a small gap in Smith's argument, namely when he applies the Chebotarev Density Theorem on page 81. Indeed, Smith does not argue why there are no Siegel zeroes. Fortunately, this can be easily overcome by an appeal to the classical result of Heilbronn \cite{HE} and the fact that our box $X$ is Siegel-less.

We need to construct an additive system $\mathfrak{A}'$ on $S'$ that guarantees the existence of the governing expansions $\mathfrak{G}_\ell$ and $\mathfrak{G}_{j, a_1a_2}$. This is done in Lemma \ref{lASGov}.

Now let $Z$ and $Z'$ be well-governed for $(a, F)$ and suppose that $Z \cap Z' = \{x\}$. Let $K$ be the field obtained by adjoining $\sqrt{p}$ to $\Q$ where $p$ runs over all the prime divisors of $x$. Then, for Smith's reduction step to work, we need to prove that
\[
[KM_\circ(Z) M_\circ(Z') : K] = [KM_\circ(Z) : K]^2 = [KM_\circ(Z') : K]^2,
\]
which follows from \cite[Lemma 6.8 \& Lemma 6.10]{KPPell}.
\end{proof}

\begin{proof}[Proof of Proposition \ref{pFinal}]
Take $\sigma \in \Gal(M(Z)/M_\circ(Z))$ and define
\[
X_{i_2(j_1, j_2)}(a, Q, M_\circ(Z), \sigma)
\]
to be the subset of $p \in X_{i_2(j_1, j_2)}(a, Q, M_\circ(Z))$ that map to $\sigma$ under Frobenius. By \cite[Lemma 6.9 \& Lemma 6.10]{KPPell} we have an isomorphism
\begin{align}
\label{eGaloisAddIso}
\Gal(M(Z)/M_\circ(Z)) \cong \mathscr{G}_{S'}(Z)
\end{align}
by sending $\sigma$ to the map 
\[
\bar{x} \mapsto
\left\{
	\begin{array}{ll}
		\phi_{\bar{x}, \ell}(\sigma) & \mbox{if } j_1 = n_{a, m} + 1 \\
		\phi_{\pi_{S' - \{i_1(j_1, j_2)\}}(\bar{x}), \text{prp}(\bar{x}, i_1(j_1, j_2))}(\sigma) & \mbox{otherwise.}
	\end{array}
\right.
\]
The Chebotarev Density Theorem and Lemma \ref{lImd} imply that
\[
|X_{i_2(j_1, j_2)}(a, Q, M_\circ(Z), \sigma)| = \frac{|X_{i_2(j_1, j_2)}(a, Q, M_\circ(Z))|}{2^{(M_{\text{box}} - 1)^{|S'|}}} \cdot \left(1 + O\left(e^{-2k_{\text{gap}}}\right)\right).
\]
Then it follows from Proposition \ref{p6.3} that for almost all choices of 
\[
Q' \in \prod_{i \in [r] - [k_{\text{gap}}] - \{i_2(j_1, j_2)\}} X_i(a, Q, M_\circ(Z)) \quad \text{ with } \quad \left(\frac{\pi_i(Q')}{\pi_j(Q')}\right) = a(i, j),
\]
we have
\begin{align}
\label{eChebotarev}
|X_{i_2(j_1, j_2)}(a, Q \times Q', M_\circ(Z), \sigma)| = \frac{|X_{i_2(j_1, j_2)}(a, Q \times Q', M_\circ(Z))|}{2^{(M_{\text{box}} - 1)^{|S'|}}} \cdot \left(1 + O\left(e^{-k_{\text{gap}}}\right)\right)
\end{align}
for each $\sigma$, where $X_{i_2(j_1, j_2)}(a, Q \times Q', M_\circ(Z))$ is the subset of $X_{i_2(j_1, j_2)}(a, Q, M_\circ(Z))$ consistent with $Q'$, and similarly for $X_{i_2(j_1, j_2)}(a, Q \times Q', M_\circ(Z), \sigma)$.

We now apply Proposition \ref{pdF} to the space $Z \times [M_{\text{box}}]$ with 
\[
\epsilon = \frac{1}{(\log \log \log \log N)^{\frac{c}{(m + 1)6^m}}}
\]
for some sufficiently small constant $c$. Let $g_0 \in \mathscr{G}_S(Z \times [M_{\text{box}}])$ be the function guaranteed by Proposition \ref{pdF}. If we pick primes $x_1, \dots, x_{M_{\text{box}}} \in X_{i_2(j_1, j_2)}(a, Q \times Q', M_\circ(Z))$, then we have an isomorphism
\[
\varphi: Z \times [M_{\text{box}}] \cong \{Q\} \times \{Q'\} \times Z \times \{x_1, \dots, x_{M_{\text{box}}}\}.
\]
To the primes $x_1, \dots, x_{M_{\text{box}}} \in X_{i_2(j_1, j_2)}(a, Q \times Q', M_\circ(Z))$, we can associate a function $g_{x_1, \dots, x_{M_{\text{box}}}} \in \mathscr{G}_S(Z \times [M_{\text{box}}])$ by setting
\[
(\bar{z}, (i, j)) \mapsto \phi_{\bar{z}}(\text{Frob } x_i) + \phi_{\bar{z}}(\text{Frob } x_j),
\]
where $\phi_{\bar{z}}$ is $\phi_{\bar{z}, \ell}$ or $\phi_{\pi_{S' - \{i_1(j_1, j_2)\}}(\bar{z}), \text{prp}(\bar{z}, i_1(j_1, j_2))}$ depending on the value of $j_1$. In case $g = g_0$, we get the desired oscillation from Proposition \ref{pdF} applied to the function $F(\text{Art}_{a, m, i(x)})$ pulled back to $Z \times [M_{\text{box}}]$ via $\varphi$ and the additive system $\mathfrak{A}$ from Lemma \ref{lAS} also pulled back to $Z \times [M_{\text{box}}]$ via $\varphi$.

It remains to split the set $X_{i_2(j_1, j_2)}(a, Q \times Q', M_\circ(Z))$ in blocks of size $M_{\text{box}}$ (and a small remainder) such that we have $g_{x_1, \dots, x_{M_{\text{box}}}} = g_0$ for almost every block. For this we we claim that given $\text{Frob}(x_1)$, there is a unique choice of $\text{Frob}(x_2), \dots, \text{Frob}(x_{M_{\text{box}}})$ such that
\[
g_{x_1, \dots, x_{M_{\text{box}}}} = g_0,
\]
and furthermore $\text{Frob}(x_2), \dots, \text{Frob}(x_{M_{\text{box}}})$ are linear functions of $\text{Frob}(x_1)$. Once we establish the claim, we use equation (\ref{eChebotarev}) to partition $X_{i_2(j_1, j_2)}(a, Q \times Q', M_\circ(Z))$ in the desired way.

To prove the claim, we remark that there is an isomorphism between $\mathscr{G}_S(Z \times [M_{\text{box}}])$ and the sets of maps $g$ from $[M_{\text{box}}] \times [M_{\text{box}}]$ to $\mathscr{G}_{S'}(Z)$ satisfying
\[
g(i, j) + g(j, k) = g(i, k).
\]
Hence, thinking of $g_0$ as a map from $[M_{\text{box}}] \times [M_{\text{box}}]$ to $\mathscr{G}_{S'}(Z)$, we see that for any $1 < j \leq M_{\text{box}}$
\[
\phi_{\bar{z}}(\text{Frob } x_1) + \phi_{\bar{z}}(\text{Frob } x_j) = g_0(1, j) \in \mathscr{G}_{S'}(Z),
\]
which uniquely specifies $\text{Frob}(x_j)$ as linear function of $\text{Frob}(x_1)$ and $g_0$ by equation (\ref{eGaloisAddIso}). Finally, we see that with this choice of $\text{Frob}(x_2), \dots, \text{Frob}(x_{M_{\text{box}}})$, we also have for all $i, j \in [M_{\text{box}}]$
\[
\phi_{\bar{z}}(\text{Frob } x_i) + \phi_{\bar{z}}(\text{Frob } x_j) = g_0(i, j)
\]
so that $g_{x_1, \dots, x_{M_{\text{box}}}} = g_0$ as desired.
\end{proof}

\appendix
\section{Density computations}
\label{aDensity}

\subsection{Stevenhagen's conjecture revisited}
\label{aSte}
Let $\ell$ be an integer such that $|\ell|$ is a prime $3$ modulo $4$. Define for any integer $n \geq 0$ the quantity
\[
\text{Pr}_{\ell, 2}(n) := \lim_{N \rightarrow \infty} \frac{|S_{\Z, N, \ell} \cap D_{\ell, 2}(n)|}{|[N] \cap D_{\ell, 2}(n)|},
\]
where $D_{\ell, k}(n)$ is defined at the beginning of Section \ref{sMain}. Let us first prove that the limit exists. To do so, we look at 
\[
\liminf_{N \rightarrow \infty} \frac{|S_{\Z, N, \ell} \cap D_{\ell, 2}(n)|}{|[N] \cap D_{\ell, 2}(n)|} \text{ and } \limsup_{N \rightarrow \infty} \frac{|S_{\Z, N, \ell} \cap D_{\ell, 2}(n)|}{|[N] \cap D_{\ell, 2}(n)|}.
\]
Theorem \ref{tHeuristic} gives increasingly better lower bounds for $\liminf$, and increasingly better upper bounds for $\limsup$. We conclude that the $\liminf$ and $\limsup$ are equal, and hence the limit exists. From the Markov chain behavior in Theorem \ref{tHeuristic}, we also see that
\[
\text{Pr}_{\ell, 3}(m, n) := \lim_{N \rightarrow \infty} \frac{|S_{\Z, N, \ell} \cap D_{\ell, 2}(m) \cap D_{\ell, 3}(n)|}{|[N] \cap D_{\ell, 2}(m) \cap D_{\ell, 3}(n)|}
\]
exists and equals $\text{Pr}_{\ell, 2}(n)$ for every $m \geq n$. Then we deduce from the identity
\[
\frac{|S_{\Z, N, \ell} \cap D_{\ell, 2}(n)|}{|[N] \cap D_{\ell, 2}(n)|} = \sum_{i = 0}^n \frac{|S_{\Z, N, \ell} \cap D_{\ell, 2}(n) \cap D_{\ell, 3}(i)|}{|[N] \cap D_{\ell, 2}(n) \cap D_{\ell, 3}(i)|} \cdot \frac{|[N] \cap D_{\ell, 2}(n) \cap D_{\ell, 3}(i)|}{|[N] \cap D_{\ell, 2}(n)|}
\]
by taking $N \rightarrow \infty$ that
\begin{equation}
\label{eLinRec}
\text{Pr}_{\ell, 2}(n) = \sum_{i = 0}^n \text{Pr}_{\ell, 3}(n, i) \cdot \frac{P(n, n, i)}{2^n} = \sum_{i = 0}^n \text{Pr}_{\ell, 2}(i) \cdot \frac{P(n, n, i)}{2^n}.
\end{equation}
We claim that
\begin{equation}
\label{eLinSol}
\frac{1}{2^{n + 1} - 1} = \sum_{i = 0}^n \frac{1}{2^{i + 1} - 1} \cdot \frac{P(n, n, i)}{2^n}.
\end{equation}
Let us first show that the claim implies Theorem \ref{tMain}. Since we clearly have $\text{Pr}_{\ell, 2}(0) = 1$, the claim and equation (\ref{eLinRec}) imply that
\begin{align}
\label{eStevenhagen}
\text{Pr}_{\ell, 2}(n) = \frac{1}{2^{n + 1} - 1}.
\end{align}
Now consider the decomposition
\[
\frac{|S_{\Z, N, \ell}|}{|S_{\Q, N, \ell}|} = \sum_{n = 0}^\infty \frac{|S_{\Z, N, \ell} \cap D_{\ell, 2}(n)|}{|[N] \cap D_{\ell, 2}(n)|} \cdot \frac{|[N] \cap D_{\ell, 2}(n)|}{|S_{\Q, N, \ell}|}.
\]
Then equation (\ref{eStevenhagen}), Theorem \ref{t4rank} and Fatou's lemma imply
\begin{align}
\label{eLimInf}
\limsup_{N \rightarrow \infty} \frac{|S_{\Z, N, \ell}|}{|S_{\Q, N, \ell}|} \geq \liminf_{N \rightarrow \infty} \frac{|S_{\Z, N, \ell}|}{|S_{\Q, N, \ell}|} \geq \sum_{n = 0}^\infty \frac{2^{-n^2} \eta_\infty \eta_n^{-2}}{2^{n + 1} - 1}.
\end{align}
Similarly, we get
\begin{align}
\label{eLimSup}
\limsup_{N \rightarrow \infty} \frac{|S_{\Q, N, \ell} \setminus S_{\Z, N, \ell}|}{|S_{\Q, N, \ell}|} \geq \liminf_{N \rightarrow \infty} \frac{|S_{\Q, N, \ell} \setminus S_{\Z, N, \ell}|}{|S_{\Q, N, \ell}|} \geq \sum_{n = 0}^\infty \frac{2^{-n^2} \eta_\infty \eta_n^{-2} \cdot (2^{n + 1} - 2)}{2^{n + 1} - 1}.
\end{align}
But it is a classical fact that
\[
\sum_{n = 0}^\infty 2^{-n^2} \eta_\infty \eta_n^{-2} = 1.
\]
Therefore equation (\ref{eLimInf}) and equation (\ref{eLimSup}) imply that
\[
\liminf_{N \rightarrow \infty} \frac{|S_{\Z, N, \ell}|}{|S_{\Q, N, \ell}|} = \limsup_{N \rightarrow \infty} \frac{|S_{\Z, N, \ell}|}{|S_{\Q, N, \ell}|} = \sum_{n = 0}^\infty \frac{2^{-n^2} \eta_\infty \eta_n^{-2}}{2^{n + 1} - 1},
\]
and Theorem \ref{tMain} follows.

It remains to prove the claimed equation (\ref{eLinSol}). Look at the probability space of pairs $(T, U)$ with the uniform measure, where $T$ is a surjective linear map $\FF_2^{[n + 1]} \rightarrow \FF_2^{[n]}$ and $U$ is a pairing $\FF_2^{[n + 1]} \times \FF_2^{[n]} \rightarrow \FF_2$. All our probabilities will be with respect to this probability space. Now fix a non-zero element $x \in \FF_2^{[n + 1]}$. Note that
\begin{equation}
\label{eP1}
\mathbb{P}(x \in \text{ker}(T)) = \frac{1}{2^{n + 1} - 1}.
\end{equation}
We write $\text{leftker}(U)$ for the set of vectors $v \in \FF_2^{[n + 1]}$ such that $U(v, w) = 0$ for all $w \in \FF_2^{[n]}$. Write $A_{i, x}$ for the event that $x \in \text{leftker}(U) \text{ and } \text{dim}_{\FF_2} \text{leftker}(U) = i + 1$. Then we have
\begin{equation}
\label{eP2}
\mathbb{P}(x \in \text{ker}(T)) = \sum_{i = 0}^n \mathbb{P}(x \in \text{ker}(T) | A_{i, x}) \cdot \mathbb{P}(A_{i, x}).
\end{equation}
Observe that $\mathbb{P}(A_{i, x}) = P(n, n, i)/2^n$. Next we have
\[
\mathbb{P}(x \in \text{ker}(T) | A_{i, x}) = \sum_{V} \mathbb{P}(x \in \text{ker}(T) | A_{i, x}, T(\text{leftker}(U)) = V) \cdot \mathbb{P}(T(\text{leftker}(U)) = V | A_{i, x}),
\]
where the sum is over $i$-dimensional subspaces $V$ of $\FF_2^{[n]}$. But we have
\[
\mathbb{P}(x \in \text{ker}(T) | A_{i, x}, T(\text{leftker}(U)) = V) = \frac{1}{2^{i + 1} - 1},
\]
since restricting $T$ to $\text{leftker}(U)$ gives a random surjective linear map $T'$ from $\text{leftker}(U)$ to $V$, with $x \in \text{leftker}(U)$. Hence we conclude that
\[
\mathbb{P}(x \in \text{ker}(T) | A_{i, x}) = \frac{1}{2^{i + 1} - 1}.
\]
Inserting this in equation (\ref{eP1}) and (\ref{eP2}), we obtain the desired identity.

\subsection{Study of the limiting value}
\label{aGamma}
Recall the definition of $\gamma$ in equation (\ref{eGamma}). The goal of this subsection is to prove the following result by combinatorial means.

\begin{theorem}
We have $\gamma = 1/2$.
\end{theorem}

\begin{proof}
We have to show that
\[
\frac{1}{2} = \sum_{k = 0}^\infty \frac{\prod_{j = 1}^\infty (1 - 2^{-j})}{2^{k^2} \cdot (2^{k + 1} - 1) \cdot \prod_{j = 1}^k (1 - 2^{-j})^2}.
\]
We rewrite this as
\[
\frac{1}{\prod_{j = 1}^\infty (1 - 2^{-j})} = \sum_{k = 0}^\infty \frac{1}{(1 - 2^{-k - 1}) \cdot \prod_{j = 1}^k (2^j - 1)^2}.
\]
This last identity follows from the formal identity
\begin{align}
\label{eFormal}
\frac{1}{\prod_{j = 1}^\infty (1 - x^{-j})} = \sum_{k = 0}^\infty \frac{1}{(1 - x^{-k - 1}) \cdot \prod_{j = 1}^k (x^j - 1)^2}
\end{align}
valid for $|x| > 1$. Using the geometric series identity $\frac{1}{1 - x} = 1 + x + x^2 + \dots$ on each term of the product, we obtain the equality
\[
\frac{1}{\prod_{j = 1}^\infty (1 - x^{-j})} = \prod_{j = 1}^\infty (1 + x^{-j} + x^{-2j} + \dots) = \sum_{t = 0}^\infty a_t x^{-t},
\]
where
\[
a_t := \left|\left\{(\alpha_1, \alpha_2, \dots) : \alpha_i \in \Z_{\geq 0}, \alpha_i \text{ eventually zero}, \sum_{i = 1}^\infty i \alpha_i = t\right\}\right|.
\]
This is readily verified to be a finite, and therefore well-defined, sum. Furthermore, it is well-known that $a_t$ equals the number of partitions $p(t)$ of $t$, which may also be directly verified by writing a partition of $t$ in $k$ parts as
\begin{align}
\label{ePartTrick}
t = (\alpha_1 + \dots + \alpha_k) + (\alpha_2 + \dots + \alpha_k) + \dots + \alpha_k.
\end{align}
We may similarly expand the other side of equation (\ref{eFormal}) as
\[
\sum_{k = 0}^\infty \frac{1}{(1 - x^{-k - 1}) \cdot \prod_{j = 1}^k (x^j - 1)^2} = \sum_{k = 0}^\infty \left(\left(\sum_{t = 0}^\infty x^{-t(k + 1)}\right) \cdot \left(\prod_{j = 1}^k \left(\sum_{t = 1}^\infty x^{-jt}\right)\right)^2\right) = \sum_{t = 0}^\infty b_t x^{-t},
\]
where
\[
b_t := \sum_{k = 0}^\infty \left|\left\{(\boldsymbol{\alpha}, \boldsymbol{\beta}, \gamma) \in \Z_{\geq 0}^{2k + 1} : \alpha_i > 0, \beta_i > 0, \sum_{i = 1}^k i \alpha_i + \sum_{i = 1}^k i \beta_i + (k + 1) \gamma = t\right\}\right|.
\]
It therefore remains to prove that $p(t) = b_t$. Before we proceed, we remark that
\[
\sum_{k = 0}^\infty \left|\left\{(\alpha_1, \dots, \alpha_k) \in \Z_{> 0}^k : \sum_{i = 1}^k i \alpha_i = t\right\}\right|
\]
is equal to the number of partitions of $t$ into \emph{distinct} parts, which follows once more by writing $t$ as in equation (\ref{ePartTrick}). We know that $p(t)$ equals the number of Young tableaux, and we will now use this to give a bijective proof. Given a Young tableau, we will produce two partitions into the same number of distinct parts as follows. 

Take the top row of the tableau, use it as the first part of the first partition and remove it from the tableau. Next take the left column of the tableau, use it as the first part of the second partition, and remove it from the remaining tableau. We continue this process, until we are left with either nothing or a potential leftover row of some length $r$. At this point, we have produced two partitions into say $k$ distinct parts, and a leftover row of length $r$. Furthermore, we know that every element of the first partition must always be strictly bigger than $r$. We now remove $r$ from each part of the first partition, and take $\gamma := r(k + 1)$. This gives the desired bijection.
\end{proof}

\end{document}